\documentclass[10pt,reqno,a4paper]{amsart}

\usepackage{amsthm}
\usepackage{mathtools}
\usepackage{amssymb}
\usepackage{hyperref}
\usepackage{enumitem}
\usepackage{graphicx}
\usepackage{comment}
\usepackage{xcolor}
\usepackage{bm}

\numberwithin{equation}{section}
\numberwithin{figure}{section}
\theoremstyle{plain}
\newtheorem{thm}{\protect\theoremname}[section]
\theoremstyle{plain}
\newtheorem{lem}[thm]{\protect\lemmaname}
\theoremstyle{plain}
\newtheorem{prop}[thm]{\protect\propositionname}
\theoremstyle{remark}
\newtheorem{rem}[thm]{\protect\remarkname}

\providecommand{\theoremname}{Theorem}

\providecommand{\lemmaname}{Lemma}
\providecommand{\remarkname}{Remark}
\providecommand{\propositionname}{Proposition}

\makeatother

\begin{document}
	\global\long\def\R{\mathbf{\mathbb{R}}}%
	\global\long\def\C{\mathbf{\mathbb{C}}}%
	\global\long\def\Z{\mathbf{\mathbb{Z}}}%
	\global\long\def\N{\mathbf{\mathbb{N}}}%
	\global\long\def\T{\mathbb{T}}%
	\global\long\def\Im{\mathrm{Im}}%
	\global\long\def\Re{\mathrm{Re}}%
	\global\long\def\Hc{\mathcal{H}}%
	\global\long\def\M{\mathbb{M}}%
	\global\long\def\P{\mathbb{P}}%
	\global\long\def\L{\mathcal{L}}%
	\global\long\def\F{\mathcal{\mathcal{F}}}%
	\global\long\def\s{\sigma}%
	\global\long\def\Rc{\mathcal{R}}%
	\global\long\def\W{\tilde{W}}%
	\global\long\def\G{\mathcal{G}}%
	\global\long\def\d{\partial}%
	\global\long\def\mc#1{\mathcal{\mathcal{#1}}}%
	\global\long\def\Right{\Rightarrow}%
	\global\long\def\Left{\Leftarrow}%
	\global\long\def\les{\lesssim}%
	\global\long\def\hook{\hookrightarrow}%
	\global\long\def\D{\mathbf{D}}%
	\global\long\def\rad{\mathrm{rad}}%
	\global\long\def\d{\partial}%
	\global\long\def\jp#1{\langle#1\rangle}%
	\global\long\def\norm#1{\|#1\|}%
	\global\long\def\ol#1{\overline{#1}}%
	\global\long\def\wt#1{\widehat{#1}}%
	\global\long\def\tilde#1{\widetilde{#1}}%
	\global\long\def\br#1{(#1)}%
	\global\long\def\Bb#1{\Big(#1\Big)}%
	\global\long\def\bb#1{\big(#1\big)}%
	\global\long\def\lr#1{\left(#1\right)}%
	\global\long\def\la{\lambda}%
	\global\long\def\al{\alpha}%
	\global\long\def\be{\beta}%
	\global\long\def\ga{\gamma}%
	\global\long\def\La{\Lambda}%
	\global\long\def\De{\Delta}%
	\global\long\def\na{\nabla}%
	\global\long\def\fl{\flat}%
	\global\long\def\sh{\sharp}%
	\global\long\def\calN{\mathcal{N}}%
	\global\long\def\avg{\mathrm{avg}}%
	\global\long\def\bbR{\mathbf{\mathbb{R}}}%
	\global\long\def\bbC{\mathbf{\mathbb{C}}}%
	\global\long\def\bbZ{\mathbf{\mathbb{Z}}}%
	\global\long\def\bbN{\mathbf{\mathbb{N}}}%
	\global\long\def\bbT{\mathbb{T}}%
	\global\long\def\bfD{\mathbf{D}}%
	\global\long\def\calF{\mathcal{\mathcal{F}}}%
	\global\long\def\calH{\mathcal{H}}%
	\global\long\def\calL{\mathcal{L}}%
	\global\long\def\calO{\mathcal{O}}%
	\global\long\def\calR{\mathcal{R}}%
	\global\long\def\calA{\mathcal{A}}%
    \global\long\def\calC{\mathcal{C}}%
	\global\long\def\calM{\mathcal{M}}%
	\global\long\def\calE{\mathcal{E}}%
	\global\long\def\calB{\mathcal{B}}%
	\global\long\def\calU{\mathcal{U}}%
	\global\long\def\calV{\mathcal{V}}%
	\global\long\def\calK{\mathcal{K}}%
    \global\long\def\calZ{\mathcal{Z}}%
	\global\long\def\Lmb{\Lambda}%
	\global\long\def\eps{\varepsilon}%
	\global\long\def\lmb{\lambda}%
	\global\long\def\gmm{\gamma}%
	\global\long\def\rd{\partial}%
	\global\long\def\aleq{\lesssim}%
	\global\long\def\chf{\mathbf{1}}%
	\global\long\def\td#1{\widetilde{#1}}%
	\global\long\def\sgn{\mathrm{sgn}}%
	
\subjclass[2020]{35B44 (primary), 35Q55, 37K10} 

\keywords{Calogero--Moser derivative nonlinear Schrödinger equation, continuum Calogero--Moser model, quantized blow-up, infinite hierarchy of conservation laws. }

\title[Quantized blow-up dynamics for CM--DNLS]{Quantized blow-up dynamics for Calogero--Moser derivative nonlinear Schrödinger equation}
\begin{abstract}
	We consider the Calogero–Moser derivative nonlinear Schr\"{o}dinger equation (CM-DNLS), an $L^2$-critical nonlinear Schrödinger type equation enjoying a number of numerous structures, such as nonlocal nonlinearity, self-duality, pseudo-conformal symmetry, and complete integrability. 

    In this paper, we construct smooth finite-time blow-up solutions to (CM-DNLS) that exhibit a sequence of discrete blow-up rates, so-called \emph{quantized blow-up rates}. Our strategy is a forward construction of the blow-up dynamics based on modulation analysis. Our main novelty is to utilize the \textit{nonlinear adapted derivative} suited to the \textit{Lax pair structure} and to rely on the \textit{hierarchy of conservation laws} inherent in this structure to control higher-order energies. 
    This approach replaces a repulsivity-based energy method in the bootstrap argument, which significantly simplifies the analysis compared to earlier works. Our result highlights that the integrable structure remains a powerful tool, even in the presence of blow-up solutions.

    In (CM-DNLS), one of the distinctive features is \emph{chirality}. However, our constructed solutions are not chiral, since we assume the radial (even) symmetry in the gauge transformed equation. This radial assumption simplifies the modulation analysis.
    
\end{abstract}

\author{Uihyeon Jeong}
\email{juih26@kaist.ac.kr}
\address{Department of Mathematical Sciences, Korea Advanced Institute of Science
	and Technology, 291 Daehak-ro, Yuseong-gu, Daejeon 34141, Korea}
\author{Taegyu Kim}
\email{k1216300@kias.re.kr}
\address{School of Mathematics, Korea Institute for Advanced Study, 85 Hoegiro Dongdaemun-gu, Seoul 02455, Korea}

\maketitle
\setcounter{tocdepth}{1}
\tableofcontents{}

\section{Introduction}

We consider the initial value problem 
\begin{align}\label{CMdnls}
	\begin{cases}
		i\partial_tu+\partial_{xx}u+2D_+(|u|^{2})u=0, \quad (t,x) \in \mathbb{R}\times\mathbb{R}
		\\
		u(0)=u_0
	\end{cases} \tag{CM-DNLS}
\end{align}
where $u:[-T,T]\times \mathbb{R}\to \mathbb{C}$, $D_+ = D \Pi_+$, $D=-i\partial_x$ and $\Pi_+$ is the projection to positive frequency. \eqref{CMdnls} is a newly introduced equation of the nonlinear Schrödinger type. It has attracted interest due to its diverse mathematical properties. Notably, \eqref{CMdnls} is mass-critical, exhibiting pseudo-conformal invariance, and completely integrable. This paper aims to construct \textit{smooth} finite-time blow-up solutions to \eqref{CMdnls} in the energy space with a sequence of blow-up rates, namely \emph{quantized blow-up rates}.

In \cite{AbanovBettelheimWiegmann2009FormalContinuum}, \eqref{CMdnls} is derived as a continuum limit of classical Calogero--Moser system, which is known to be completely integrable \cite{Calogero1971ClassicCalogerMoser, CalogeroMarchioro1974ClassicCalogerMoser, Moser1975ClassicCalogerMoser, OlshanetskyPerelomov1976Invent}. This model is also referred to as the continuum Calogero–Moser model (CCM) or Calogero–Moser NLS (CM-NLS). 

Gérard and Lenzmann \cite{GerardLenzmann2024CPAM} began a rigorous mathematical analysis on \eqref{CMdnls}, establishing its complete integrability on the \textit{Hardy--Sobolev} space $H_{+}^{s}(\mathbb{R})$:
\[
H_{+}^{s}(\mathbb{R})\coloneqq H^{s}(\R)\cap L_{+}^{2}(\R),\qquad L_{+}^{2}(\R)\coloneqq\{f\in L^{2}(\mathbb{R}):\text{supp}\,\widehat{f}\subset[0,\infty)\}.
\]
The positive frequency condition $u(t)\in L^2_+$ is called \textit{chirality}, and it is preserved under the \eqref{CMdnls} flow. With chirality, the Lax pair structure for \eqref{CMdnls} holds:
\begin{equation}
    \frac{d}{dt}\mathcal{L}_{\textnormal{Lax}}=[\mathcal{P}_{\textnormal{Lax}},\mathcal{L}_{\textnormal{Lax}}]\label{eq:LaxPair-evol}
\end{equation}
with the $u$-dependent operators $\mathcal{L}_{\textnormal{Lax}}$
and $\mathcal{P}_{\textnormal{Lax}}$ defined by\footnote{\eqref{eq:LaxPair} differs slightly from the version in \cite{GerardLenzmann2024CPAM}, but equals on $H_{+}^{s}(\mathbb{R})$. This formulation was presented in the introduction of \cite{KillipLaurensVisan2023CMDNLSL2plusexplicitformulaarxiv}.} 
\begin{align}
    \mathcal{L}_{\textnormal{Lax}}=-i\partial_{x}-u\Pi_{+}\overline{u},\quad\text{and}\quad\mathcal{P}_{\textnormal{Lax}}=i\partial_{xx}+2iuD_{+}\overline{u}.\label{eq:LaxPair}
\end{align}
A similar Lax pair structure for the full intermediate NLS was found in \cite{PelinovskyGrimshaw1995IntermediateDefocusingCMDNLSLaxPair} which includes a defocusing version of \eqref{CMdnls}. There is another manifestation of integrability on the Hardy space: an explicit formula \cite{KillipLaurensVisan2023CMDNLSL2plusexplicitformulaarxiv}, which was developed by Gérard and collaborators in other integrable models \cite{GerardGrellier2015ExplictFormulaCubicSzego1,Gerard2023BOequexplicitFormula,GerardPushnitski2024CMP}.
Notably, the Lax pair structure \eqref{eq:LaxPair-evol} also holds in the standard Sobolev space $H^s(\bbR)$ \cite{KimKimKwon2024arxiv}, though the explicit formula does not.

We briefly point out symmetries and conservation laws. \eqref{CMdnls} enjoys (time and space) translation and phase rotation symmetries associated to the conservation laws of energy, mass, and momentum:
\begin{gather*}
    \tilde{E}(u)=\frac{1}{2}\int_{\mathbb{R}}\left|\partial_x u-i\Pi_+(|u|^2)u\right|^2dx, 
    \\
    M(u)=\int_{\mathbb{R}}|u|^{2}dx, \quad \tilde{P}(u)=\Re\int_{\mathbb{R}}(\overline{u}Du-\frac{1}{2}|u|^{4})dx. \nonumber
\end{gather*}
In addition, it has Galilean invariance
\begin{align*}
	u(t,x)\mapsto e^{ic x-ic^2t}u(t,x-2ct), \quad c \in \mathbb{R},
\end{align*}
\textit{$L^{2}$-scaling symmetry} 
\begin{align*}
    u(t,x)\mapsto\lambda^{-\frac{1}{2}}u(\lambda^{-2}t,\lambda^{-1}x),\quad \lambda>0,
\end{align*}
and pseudo–conformal symmetry
\begin{equation}
    u(t,x)\mapsto \frac{1}{|t|^{1/2}}e^{i\frac{x^{2}}{4t}}u\left(-\frac{1}{t},\frac{x}{|t|}\right).\label{eq:pseudo-conf-transf}
\end{equation}
We refer that the symmetries mentioned above (and their associated virial identities) are the same as those of the mass-critical nonlinear Schr\"odinger equations. However, when focusing on chiral solutions, certain restrictions are required to maintain their chirality. For instance, Galilean invariance necessitates $c\geq 0$, and pseudo-conformal symmetry fails to preserve chirality.

The \emph{ground state} or \emph{soliton}, defined as the nonzero energy minimizer, is fundamental in the global behavior of solutions. In particular in \cite{GerardLenzmann2024CPAM}, it is proved that the ground state has zero energy (so it is static) and has an explicit formula:
\begin{equation*}
    \mathcal{R}(x)=\frac{\sqrt{2}}{x+i}\in H_{+}^{1}(\mathbb{R})\quad\text{with}\quad M(\mathcal{R})=2\pi\quad\text{and}\quad\widetilde{E}(\mathcal{R})=0,
\end{equation*}
We note that $\calR(x)$ is unique up to scaling, phase rotation, and translation symmetries, and it is also chiral solution to \eqref{CMdnls}. In general, any nontrivial $H^{1}(\bbR)$ traveling wave solutions (i.e., solutions of the form $u(t,x)=e^{i\omega t}\calR_{c,\omega}(x-ct)$ for some $\omega,c\in\bbR$) are represented by $\mathrm{Gal}_c \calR$ up to scaling, phase rotation, and translation symmetries \cite{GerardLenzmann2024CPAM}. 
Applying the pseudo-conformal
transform \eqref{eq:pseudo-conf-transf} to $\calR$, one can obtain
an explicit finite-time blow-up solution:
\[
S(t,x)\coloneqq\frac{1}{t^{1/2}}e^{ix^{2}/4t}\mathcal{R}\left(\frac{x}{t}\right)\in L^{2}(\bbR),\qquad\forall t>0.
\]
However, $S(t)$ not only does \emph{not} belong to a finite energy class $H^1(\bbR)$, it is also \emph{not
chiral}.

Now, we provide a short overview of prior studies on \eqref{CMdnls}. In \cite{MouraPilod2010CMDNLSLocal}, the authors proved that \eqref{CMdnls} is locally well-posed in $H^s(\bbR)$ for all $s > \frac{1}{2}$. The same local well-posedness also holds for chiral data, as shown in \cite{GerardLenzmann2024CPAM}.
The global well-posedness for sub-threshold and threshold data, i.e.,
$M(u_{0})\leq M(\mathcal{R})=2\pi$, in $H_{+}^{k}(\mathbb{R})$ for
all $k\in\mathbb{N}_{\geq1}$ was also established by \cite{GerardLenzmann2024CPAM}. Specifically, for sub-threshold data, the authors obtained \textit{a priori} bounds, i.e. 
for all $k\in\mathbb{N}_{\geq1}$, 
\begin{equation}\label{eq:subthreshold apriori bound}
    M(u_{0})<M(\mathcal{R}) \;\Rightarrow \;\sup_{t\in\R}\norm{u(t)}_{H_{+}^{k}(\R)}\lesssim_{k, M(u_{0})}\norm{u_{0}}_{H_{+}^{k}(\R)}.
\end{equation}
These bounds \eqref{eq:subthreshold apriori bound} are consequences of the hierarchy of conservation laws. As is shown, the conservation laws are \emph{coercive} only in the sub-threshold regime. Moreover, Killip, Laurens, and Vi\c{s}an \cite{KillipLaurensVisan2023CMDNLSL2plusexplicitformulaarxiv} established the local well-posedness in $L_{+}^{2}(\mathbb{R})$ for sub-threshold
data. For threshold data, the authors of \cite{GerardLenzmann2024CPAM} showed that finite time blow-up does not happen in $H^1$ using a similar argument in \cite{Merle1993Duke}.

Beyond the threshold, the dynamics has also been explored. The authors in \cite{GerardLenzmann2024CPAM} constructed $N$-soliton solutions with mass $N\cdot M(\calR)$ using the Lax pair structure, which blow up in infinite time with $\|u(t)\|_{H^{s}}\sim|t|^{2s}$ as $|t|\to\infty$ for any $s>0$. With mass close to the threshold $M(\calR)$, Hogan and Kowalski \cite{HoganKowalski2024PAA} proved, via the explicit formula, the existence of solutions that either blow up in finite time or in infinite time.  
Although \eqref{CMdnls} is completely integrable, due to above results, it had been an intriguing question whether there exist finite-time blow-up solutions.
The first finite-time blow-up solutions were constructed
by the second author, K. Kim, and Kwon \cite{KimKimKwon2024arxiv}, arising from smooth chiral data. 

On the other hand, soliton resolution was proved by the second author and Kwon \cite{KimTKwon2024arxivSolResol} in a fully general setting. They showed that any $H^{1,1}$ solutions asymptotically decomposed into a sum of modulated solitons and a radiation part for global solutions. An analogous result for finite-time blow-up solutions was also obtained. Subsequently, the single-bubble finite-time blow-up dynamics for \eqref{CMdnls} were sharply classified in \cite{JKKK2026arxiv}. More precisely, assuming a single-soliton blow-up profile, the possible blow-up rates fall into a dichotomy between the quantized regimes and an exotic regime. Thus, the quantized solutions constructed in the present work realize one branch of the later classification.
In addition, the zero dispersion limit of \eqref{CMdnls} was investigated
by Badreddine \cite{Badreddine2024SIAM}. We further refer to \cite{Badreddine2024PAA,Badreddine2024AIHPC}
for the periodic model and to \cite{Sun2024LMPsystem} for the system extension of \eqref{CMdnls}.

\subsection*{Gauge transform} \text
The equation \eqref{CMdnls} enjoys an additional structure, which is shared with the derivative nonlinear Schr\"odinger equation. We define a \textit{gauge transform} by
\begin{align*}
	v(t,x)=-\mathcal{G}(u)(t,x):=-u(t,x)e^{-\frac{i}{2}\int_{-\infty}^x|u(y)|^2dy}.
\end{align*}
Then, $v(t,x)$ satisfies the gauge transformed Calogero--Moser DNLS \eqref{CMdnls-gauged},
\begin{align}\label{CMdnls-gauged}
	\begin{cases}
		i\partial_tv+\partial_{xx}v+|D|(|v|^{2})v-\frac{1}{4}|v|^4v=0, \quad (t,x) \in \mathbb{R}\times\mathbb{R},
		\\
		v(0)=v_0,
	\end{cases}\tag{$\mathcal{G}$-CM}
\end{align}
where $|D|$ is the Fourier multiplier operator with symbol $|\xi|$. We will mainly work on \eqref{CMdnls-gauged}. Our main motivation for considering \eqref{CMdnls-gauged} is to exploit its \textit{self-dual structure}. For a detailed explanation, see Section \ref{sec:linearized dynamics}. 

Since the gauge transform is a diffeomorphism of $H^s(\bbR)$ for any $s\geq 0$, it preserves all the symmetries of \eqref{CMdnls}. Consequently, \eqref{CMdnls-gauged} also retains the conservation laws of energy, mass, and momentum:
\begin{align}
    \begin{gathered}E(v)=\frac{1}{2}\int_{\mathbb{R}}\left|\partial_{x}v+\frac{1}{2}\mathcal{H}(|v|^{2})v\right|^{2}dx,\\
    M(v)=\int_{\mathbb{R}}|v|^{2}dx,\quad P(v)=\int_{\mathbb{R}}\Im(\overline{v}\partial_{x}v)dx,
    \end{gathered}
    \label{eq:intro energy gauge}
\end{align}
where $\mathcal{H}$ is the Hilbert transform. We have chosen the minus sign in the transform $v=-\mathcal{G}(u)$ so that the static solution $\calR$ of \eqref{CMdnls} is transformed into a positive static solution to \eqref{CMdnls-gauged} 
\begin{align*}
Q(x)\coloneqq-\mathcal{G}(\mathcal{R})(x)=\frac{\sqrt{2}}{\sqrt{1+x^{2}}}\in H^{1}(\mathbb{R}),\quad M(Q)=2\pi,\quad E(Q)=0.
\end{align*}
More importantly, for \textit{any} solutions $v(t)\in H^{s}(\R)$
to \eqref{CMdnls-gauged}, the Lax pair structure for \eqref{CMdnls-gauged} also holds:
\begin{align}
    \frac{d}{dt}\widetilde{\mathbf{D}}_{v}=[-iH_{v},\widetilde{\mathbf{D}}_{v}], \label{eq:intro Lax gauge}
\end{align}
where $\widetilde{\mathbf{D}}_{v}$
and $H_{v}$ are given by
\begin{align*}
\widetilde{\mathbf{D}}_{v}=\partial_{x} +\frac{1}{2}v\mathcal{H}\overline{v},\quad 
H_v=-\partial_{xx}+\frac{1}{4}|v|^{4}-v|D|\overline{v}.
\end{align*}

Following up the finite-time blow-up result in \cite{KimKimKwon2024arxiv}, our goal is to construct other blow-up solutions with various blow-up rates. In fact, we exhibit a sequence of blow-up rates arising from smooth radial data for \eqref{CMdnls-gauged}. By applying the inverse gauge transform, we also obtain smooth finite-time blow-up solutions to \eqref{CMdnls} with the same quantized rates. 
\begin{thm}[Quantized blow-up]\label{thm:main}
For any natural number $L \ge 1$, there exists a smooth radial initial data $v_{0} \in H^{\infty}(\mathbb{R})$ such that the corresponding solution $v(t,r)$ to
\eqref{CMdnls-gauged} blows up in finite time $0<T=T(v_0)<+\infty$ as follows: there exist $\ell=\ell(v_0)\in (0,\infty)$, $\gamma^*=\gamma^*(v_0)\in \mathbb{R}/2\pi \mathbb{Z}$ and $v^*=v^*(v_0) \in H^1(\mathbb{R})$ such that
\begin{equation}
    v(t,r) - \frac{e^{i \gamma^*}}{\sqrt{\ell(T-t)^{2L}}} Q\left( \frac{r}{\ell(T-t)^{2L}}\right) \to v^* \quad \textnormal{in }L^2 \quad \textnormal{as } t\to T. \label{eq:thm strong converge}
\end{equation}
Moreover, the mass of $v_0$ can be taken arbitrarily close to $M(Q)$. Consequently, analogous smooth blow-up solutions for \eqref{CMdnls} are also constructed.
\end{thm}

\subsection*{Comments on Theorem \ref{thm:main}}

\ 

\emph{1. Method and novelties.}
Our approach adopts the forward construction using modulation analysis, following the robust argument established in previous works \cite{RodnianskiSterbenz2010, RaphaelRodnianski2012, MerleRaphaelRodnianski2013Invention, RaphaelSchweyer2013CPAMHeat, RaphaelSchweyer2014AnalPDEHeatQuantized, MerleRaphaelRodnianski2015CambJMath}.
On a more technical level, our analysis is motivated by studies on the self-dual Chern–Simons–Schrödinger equation (CSS) \cite{KimKwon2020blowup, KimKwonOh2020blowup, Kim2022CSSrigidityArxiv}, particularly in the use of nonlinearly conjugated variables and conjugation identities, which were also used in the earlier work \cite{KimKimKwon2024arxiv}. Additionally, we employ the decomposition of solutions depending on topologies, as introduced in \cite{KimKimKwon2024arxiv}.

Our main novelty lies in substantial simplifying the standard higher-order energy method by employing \emph{the hierarchy of conservation laws arising from integrability}.  
As explained earlier in \eqref{eq:subthreshold apriori bound}, the coercivity of conservation laws fails for the super-threshold data. More precisely, $\norm{v(t)}_{H^{k}}$ is not controlled for $M(v) > M(Q)$. For higher-order derivatives of $v(t)$, we use an intrinsic nonlinear structure. Defining $v_k\coloneq\td \bfD_v^k v$, the conservation laws are nothing but just $\norm{v_k}_{L^2}$, and so we have a global control of $v_k$ for free. Remarkably, $v_k$ solves the common equation \eqref{eq:intro wk equ}. This allows us to take a canonical decomposition. See more details in Section \ref{sec:linearized dynamics}.

A further advantage of this formulation is that it avoids the need for repulsivity or monotonicity information. Indeed, in view of the structure of one of the Lax pairs around $Q$, see \eqref{eq:DQ BQ decomp}, we are able to exploit the coercivity of the conservation laws in \eqref{eq:boots energy AQ}. See also Remark~\ref{rem:key estimate}.

\emph{2. Quantized blow-up rates.} The study of quantized (type-II) blow-up dynamics first originated in the context of the energy-supercritical nonlinear heat equations (NLH), from the discovery of formal mechanisms \cite{HerreroVelazquez1992, HerreroVelazquez1994CRASPSI, FilippasHerreroVelazquez2000} to classification results \cite{Mizoguchi2004AdvDE, Mizoguchi2007Math.Ann.,Mizoguchi2011TranAMS}. These results were extended to critical parabolic equations through modulation analysis \cite{RaphaelSchweyer2014AnalPDEHeatQuantized} by Rapha\"{e}l and Schweyer; see also the further developments in \cite{HadzicRaphael2019JEMS, CollotGhoulMasmoudiNguyen2022CPAM}, as well as related results using inner-outer gluing methods \cite{delPinoMussoWeiZhang2020HeatQuantized, Junichi2020AIHPC}. The modulation analysis has also been effectively applied to dispersive settings. Merle, Rapha\"{e}l, and Rodnianski \cite{MerleRaphaelRodnianski2015CambJMath} first established quantized blow-up dynamics for the energy-supercritical NLS, and subsequent works extended these results to other energy-supercritical dispersive equations \cite{Collot2018MemAmer, GhoulIbrahimNguyen2018JDE, MerleRaphaelRodnianskiSzeftel2022Invention}. Recently, the first author constructed quantized blow-up dynamics for the energy-critical corotational wave maps \cite{Jeong2024arxiv}. At the level of construction, our result provides the second such example within the class of critical dispersive equations. Subsequent work \cite{JKKK2026arxiv} further shows that, in the single-bubble regime, the quantized rates constructed here are part of a sharp classification of finite-time blow-up dynamics for \eqref{CMdnls}.

Building on previous studies, the authors of \cite{KimKimKwon2024arxiv} observed that \eqref{CMdnls-gauged}, with smooth initial data, exhibits formal similarities to the 3D energy-critical NLS and NLH. Notably, our result aligns precisely with those for the 3D NLH \cite{FilippasHerreroVelazquez2000, delPinoMussoWeiZhang2020HeatQuantized}.

For non-smooth data (e.g., $u_{0}\in H^{s}(\mathbb{R})$ in a limited range of $s$), we conjecture the existence of a continuum of admissible blow-up rates, supported by similar observations in the 3D energy-critical NLS \cite{OrtolevaPerelman2013, Schmid2023arXiv}; see also \cite{KriegerSchlagTataru2008Invent, KriegerSchlagTataru2009Duke, Perelman2014CMP, JendrejLawrieRodriguez2022ASENS, KimKimKwon2024arxiv}.

\emph{3. Chirality and radial assumption.} 
In \eqref{CMdnls}, one of the distinctive features is chirality. However, our constructed solution is not chiral. Since these dynamics for \eqref{CMdnls} are obtained by applying the inverse gauge transform to \eqref{eq:thm strong converge} in the class of radial (even) functions, they cannot be extended to chiral data ($H^\infty_+(\mathbb{R})$). This radial assumption is required to establish the coercivity for a linearized operator. See Remark \ref{rem:even assume} for more details. Nevertheless, the construction of quantized chiral blow-up solutions can be regarded as a technical issue. It is plausible that such solutions can be constructed using a similar approach to that in \cite{KimKimKwon2024arxiv}, where chiral blow-up solutions were obtained.  

On the other hand, the chiral blow-up solutions can be expressed by the explicit formula in \cite{KillipLaurensVisan2023CMDNLSL2plusexplicitformulaarxiv}, so we wonder how the blow-up information is encoded in the formula. 
We refer to \cite{HoganKowalski2024PAA} for results that utilize this formula, showing its applicability beyond the threshold mass.


\emph{4. Finite codimensional stability and rotational instability.}  Our blow-up dynamics exhibits codimension $2L-1$ stability, which arises from the unstable directions involved in standard Brouwer type fixed point arguments. To explore the nearby dynamics around this quantized blow-up, we approximate the blow-up dynamics via its modulation laws \eqref{eq:intro mod equ}. In \cite{KimKimKwon2024arxiv}, which corresponds to the case when $L=1$, the authors predicted that a phenomenon called \emph{rotational instability} would occur in the direction where blow-up does not happen. This phenomenon was first observed in previous studies on other models \cite{KimKwon2019,Kim2022CSSrigidityArxiv,vandenBergWilliams2013,MerleRaphaelRodnianski2013Invention} and later in \eqref{CMdnls} \cite{KimKimKwon2024arxiv}.
The observed behavior describes that the spatial scale contracts to a small scale, then the phase rotation parameter abruptly changes in a short time scale by the fixed amount
of angle, and finally the spatial scale starts
to grow (the spreading) as a backward evolution of the blow-up.

We can observe that these rotational instabilities emerge in $L$ out of the $2L-1$ unstable directions. Moreover, the remaining $L-1$ directions correspond to instabilities that also appear in other quantized blow-up solutions such as \cite{RaphaelSchweyer2014AnalPDEHeatQuantized, Jeong2024arxiv}, and their mechanisms can be explicitly identified. For the detailed mechanism of the instabilities, see Remark \ref{rem:rotational instab}.


\emph{5. Further regularity of asymptotic profile.}
By using the hierarchy of conservation laws, we obtain scale bounds on the higher-order energies. This allows us to expect $v^*\in H^{2L-\frac{1}{2}-}$ from the observations in \cite{RaphaelSchweyer2013CPAMHeat, RaphaelSchweyer2014AnalPDEHeatQuantized}. However, when we apply the method using the nonlinear variables and the Lax structure, we do not impose the coercivity for higher-order linearized operators, which hinders us from achieving further regularity. See Remark \ref{rem:asym profile} for details.
We also believe that the regularity of the asymptotic profile associated with the quantized blow-up solutions to \eqref{CMdnls} can be obtained up to $ H^{2L-\frac{1}{2}-}$.

\subsection*{Strategy of the proof}
We use the notation introduced in Section \ref{sec:notation}. Our proof basically follows the scheme of previous works \cite{RodnianskiSterbenz2010, RaphaelRodnianski2012, MerleRaphaelRodnianski2013Invention, RaphaelSchweyer2013CPAMHeat, RaphaelSchweyer2014AnalPDEHeatQuantized, MerleRaphaelRodnianski2015CambJMath}, but differs in three main aspects.

First, we use nonlinear adapted derivatives (nonlinear variables), which significantly simplify the construction of the modified profile compared to linear adapted derivatives. This approach was originally introduced in the context of (CSS) \cite{KimKwonOh2020blowup, Kim2022CSSrigidityArxiv} and was also used in prior results for \eqref{CMdnls} \cite{KimKimKwon2024arxiv}.
In our case, we introduce more suitable nonlinear variables to take advantage of integrability. These variables preserve a nonlinear algebraic structure, allowing information about all higher-order nonlinear variables to be derived from the first nonlinear variable. Moreover, this enables the scheme to work without requiring coercivity for higher-order linearized operators.

Second, we adopt a method of changing decompositions \emph{depending on topologies} (or, on norms), as introduced in \cite{KimKimKwon2024arxiv}. This approach was designed to deal with issues arising from the slow spatial decay of the soliton $Q$, which prevents decompositions from being admissible. Typically, such issues are resolved by adopting artificial cut-offs in the profiles. However, this introduces difficulties in handling the non-local structure. Instead of relying on cut-offs, we perform decompositions only in the topologies where the profiles and decompositions are well-defined.

Third, we obtain estimates for the higher-order energies using the \emph{hierarchy of conservation laws}. This is our key novelty, which simplifies the bootstrap argument by obviating the need for higher-order energy estimates in the assumptions. 

Despite these refinements, one difficulty remains in the proof. Since the modulation parameters are defined through nonlinear variables, they are not directly read off from the initial data. In particular, although the hierarchy of conservation laws later propagates the higher-order scaling bounds, it does not by itself relate the initial data to the initial modulation parameters. To overcome this, we prove the approximation \eqref{eq:intro init approx} by a tail computation for the initial data, comparing the higher-order nonlinear variables at $s=s_0$ with the corresponding linearized quantities. The difficulty is that the linearized operator contains the Hilbert transform $\calH$, so that the inverse operators defined through the Green's function are not straightforwardly well-defined.

Now, we provide a brief sketch of the proof. For fixed $L\geq 1$, we first choose the profiles for the nonlinear variables. \eqref{CMdnls-gauged} can be rewritten as
\begin{align}
    (\partial_t+iH_v)v=0. \label{eq:G-CM rewrite}
\end{align}
We can derive the equations for the nonlinear variables 
\begin{equation*}
    \td\bfD_v^jv\coloneqq v_j
\end{equation*}
by using the Lax pair structure \eqref{eq:intro Lax gauge}.
For $1\leq j\leq 2L$, by conjugating $\td\bfD_v^j$ to \eqref{eq:G-CM rewrite} via \eqref{eq:intro Lax gauge}, we derive
\begin{align}
    (\partial_t+iH_v)v_j=0. \label{eq:intro wk equ}
\end{align}
This simple structure, emerging from integrability, ensures that the linearized dynamics of $v_j$ around $Q$ behaves the same for all $j\geq 1$. This allows us to choose the profiles for $v_j$ to be the same as those for $v_1$. The profile for $v_1$ has already established in the previous work \cite{KimKimKwon2024arxiv}, namely $-(ib+\eta)\frac{y}{2}Q + (i\nu+\mu)\frac{1}{2}Q$. However, our symmetry restriction implies that $v_1$ is odd, the profile for $v_1$ is given by $P_1=-(ib+\eta)\frac{y}{2}Q$. 

Under this observation, we derive the formal system of modulation parameters. We renormalize the flow as
\begin{align*}
	s(t)=s_0+\int_0^t\frac{1}{\lambda(\tau)^{2}}d\tau,\quad y=\frac{x}{\lambda},\quad w(s,y)=\lambda^{\frac{1}{2}}e^{-i\gamma}v(t,\lambda y)|_{t=t(s)},
\end{align*}
and we also use nonlinear variables $w_{j}$ defined by 
\begin{align*}
	w_{j}\coloneqq\td\bfD_{w}^jw=\lambda^{\frac{2j+1}{2}} e^{-i\gamma}(\td\bfD_{v}^jv)(t,\lambda y)|_{t=t(s)}.
\end{align*}
In the radial setting, we will use dynamics of $w$ and $w_{2k-1}$ for each $1\le k \le L$. See remark \ref{rem:why only 2k-1}. We assume
\begin{align*}
    w(s)\approx Q+T_1(b_1,\eta_1),\quad w_{2k-1}(s)\approx P_{2k-1}(b_k,\eta_k),
\end{align*}
where the profiles $T_1$ and $P_{2k-1}$ are defined by
\begin{equation*}
    T_1(b_1,\eta_1)\coloneqq-(ib_1+\eta_1) \frac{1+y^2}{4}Q, \, 
    P_{2k-1}(b_k, \eta_k)\coloneqq-(ib_k+\eta_k) \cdot (-i)^{k-1}\cdot \frac{y}{2}Q.
\end{equation*}
Here, the choice of $T_1$ is motivated by the generalized kernel of the linearized operator of \eqref{eq:G-CM rewrite} around $Q$. Substituting $Q+T_1$ into $w$ and $P_{2k-1}$ into $w_{2k-1}$ in their renormalized equations (\eqref{eq:w equ} and \eqref{eq:w2k-1 equ}), we derive the dynamics of the modulation parameters:
\begin{align}
    \begin{gathered}
        \frac{\lambda_{s}}{\lambda}+b_1=0,\quad\gamma_{s}-\frac{\eta_1}{2}=0, 
        \\
        \begin{aligned}
            (b_k)_s = b_{k+1}-(2k-\tfrac{1}{2})b_1b_{k}-\tfrac{1}{2}\eta_1\eta_k, \quad b_{L+1}\coloneqq0,
            \\
            (\eta_k)_s=\eta_{k+1}-(2k-\tfrac{1}{2})b_1\eta_{k}+\tfrac{1}{2}\eta_1b_k, \quad \eta_{L+1}\coloneqq 0,
        \end{aligned}\quad 1\le k \le L.    
    \end{gathered}
	\label{eq:intro mod equ}
\end{align}
Then, we can find a special solution to \eqref{eq:intro mod equ} satisfying
\begin{align*}
    b_1(s)= \frac{2L}{4L-1}\cdot \frac{1}{s}, \quad \eta_k\equiv 0 \quad \text{for } 1\leq k\leq L.
\end{align*}
From $-\frac{\lambda_s}{\lambda}=b_1$ and $\frac{ds}{dt}=\frac{1}{\lambda(t)^{2}}$, we derive that $\lambda(t)\sim (T-t)^{2L}$, leading to the quantized blow-up dynamics. Linearizing the ODE system \eqref{eq:intro mod equ} around this special solution reveals the presence of $2L-1$ unstable directions. Thus, we use Brouwer's fixed point theorem to control these directions in the neighborhood of this special solution to the ODE system. 

To justify this argument, we use a bootstrap argument, which requires the description of an initial data set in which the initial values of the modulation parameters are close to those of the special solution. Since the modulation parameters are defined through nonlinear variables, they are not directly determined by the initial data. The key point is therefore to compare the nonlinear variables at $s=s_0$ and the corresponding linearized quantities, namely
\begin{equation*}
    (\td\bfD_Q^{j-1}L_Q\wt\eps(s_0),\iota yQ\chi)_r
    \approx 
    (w_j(s_0),\iota yQ\chi)_r \quad \textnormal{for } \iota=1,i,
\end{equation*}
where $\td\bfD_Q^{j-1}L_Q$ is the linearized operator of $v_j=\td\bfD_v^jv$ around $Q$. 
This is precisely the content of Lemma~\ref{lem:init approx}. In Section~\ref{sec:trapped regime}, the proof of Theorem~\ref{thm:main} is reduced to this lemma, while Section~\ref{sec:tail compute} is devoted to its proof by a tail computation for the initial data. This comparison makes it possible to relate the initial data to the initial modulation parameters, and hence to choose initial data for which these parameters are close to those of the special solution. In this step, we also use the generalized kernel elements together with the algebraic structure of $\td\bfD_Q^{j-1}L_Q$ to control the left-hand side above; see \eqref{eq:T_k} and Lemma~\ref{lem:profile choosing}.

To close the bootstrap argument, it is necessary to ensure sufficient smallness of the higher-order energies and to propagate this smallness. Notably, since the profile terms $P_{2k-1}(b_k,\eta_k)$ originate from $w_{2k-1}$, we need estimates for $w_{2k-1}$ and its radiation part. Our key point is the hierarchy of the conservation laws derived from the Lax pair structure \eqref{eq:intro Lax gauge}, i.e., $\|v_j\|_{L^2}$ is conserved. From this, the following estimates for higher-order energies can be obtained without bootstrap bounds:
\begin{align}\label{eq:intro init approx}
    \|w_j(s)\|_{L^2}\lesssim_{v_0,j}\lambda(s)^j \quad\textnormal{for all}\quad s\geq s_0.
\end{align}
This directly provides the $\dot H^{2k-1}$ level estimate, $\|w_{2k-1}\|_{L^2} \lesssim \lambda^{2k-1}$. For the $\dot H^{2k}$ level, we derive $\|A_Qw_{2k-1}\|_{L^2}\lesssim \lambda^{2k}$ where $A_Q$ satisfies $H_Q=A_Q^* A_Q$, which is consistent with the energy estimate in \cite{KimKimKwon2024arxiv}. To obtain this, we use the repulsivity properties \eqref{eq:BQBQstar equal I} and \eqref{eq:DQ BQ decomp}.
We observe that
\begin{align*}
    \|\td\bfD_Qw_{2k-1}\|_{L^2}\approx \|\td\bfD_ww_{2k-1}\|_{L^2}=\|w_{2k}\|_{L^2}
    \lesssim \lambda^{2k},
\end{align*}
where this follows from $w\approx Q$. By using \eqref{eq:BQBQstar equal I} and \eqref{eq:DQ BQ decomp}, we deduce
\begin{align*}
    \|A_Qw_{2k-1}\|_{L^2}=\|B_Q^*A_Qw_{2k-1}\|_{L^2}=\|\td\bfD_Qw_{2k-1}\|_{L^2}\lesssim \lambda^{2k}.
\end{align*}
Therefore, by employing the coercivity of $A_Q$, we obtain the estimate at the $\dot H^{2k}$ level $\|\eps_{2k-1}\|_{\dot \calH^1}\lesssim \lambda^{2k}$, where $\eps_{2k-1}=w_{2k-1}-P_{2k-1}(b_k,\eta_k)$ is the radiation part of $w_{2k-1}$. This, in turn, allows us to establish the modulation estimates, thereby closing the bootstrap. 

Throughout the analysis, we adopt the method of changing decompositions depending on topologies (or, on norms), as introduced in \cite{KimKimKwon2024arxiv}. Due to the slow decay of the soliton $Q$, the profiles $P_{2k-1}$ are not well-defined in $L^2$. Instead of introducing cut-offs in profiles, we decompose solutions only in the topologies where decompositions are naturally defined.

\begin{rem}\label{rem:even assume}
    Without the radial assumption, the coercivity of $A_Q$ fails to hold, making it tricky to establish $\|\eps_{2k-1}\|_{\dot \calH^1}\lesssim \lambda^{2k}$. Nonetheless, we believe that this issue can be addressed using a similar approach to that in \cite{KimKimKwon2024arxiv}. This difficulty arises because the kernel of $A_Q$ (over $\bbC$) contains not only $yQ$ but also $Q$. Moreover, the profiles corresponding to $iQ$ and $Q$ cannot be found in $w_{2k-1}$. Instead, they appear only in the form of $iyQ$ and $yQ$ in $w_{2k-2}$. This stems from the non-locality of the Hilbert transform $\calH$, which causes the first-order operator $A_Q$ (or $\td \bfD_Q$) to have two kernel elements over $\bbC$.
\end{rem}





\subsection*{Organization of the paper}
In Section~\ref{sec:notation}, we introduce the notation
and preliminaries for our analysis, including the hierarchy of conservation laws. In Section~\ref{sec:profile}, we review the linear theory of the linearized operators of \eqref{CMdnls-gauged} around the soliton. In addition, we define modified profiles and derive the formal ODE system for the modulation parameters. In Section~\ref{sec:trapped regime}, we set up the trapped regime by describing the initial data, formulating the bootstrap assumptions, and establishing the higher-order energy and modulation estimates. The proof of Theorem~\ref{thm:main} is then reduced to an approximation lemma for the initial data, Lemma~\ref{lem:init approx}. Finally, in Section~\ref{sec:tail compute}, we prove this lemma by a tail computation for the initial data.

\subsection*{Acknowledgments}
The authors appreciate Kihyun Kim and Soonsik Kwon for helpful discussions and suggestions for this work. U.~Jeong was partially supported by the National Research Foundation of Korea (RS-2019-NR040050, RS-2022-NR069873, RS-2024-00333393). T.~Kim was partially supported by the National Research Foundation of Korea (RS-2019-NR040050, RS-2022-NR069873, RS-2024-00333393) and by a KIAS Individual Grant (MG105201) at Korea Institute for Advanced Study.


\section{Notation and preliminaries}\label{sec:notation}
In this section, we gather the notations and commonly used formulas.
For quantities $A\in\bbC$ and $B\geq0$, we write $A\lesssim B$
if $|A|\leq CB$ holds for some implicit constant $C$. For $A,B\geq0$,
we write $A\sim B$ if $A\lesssim B$ and $B\lesssim A$. If $C$
is allowed to depend on some parameters, then we write them as subscripts
of $\lesssim,\sim,\gtrsim$ to indicate the dependence.

We write $\langle x\rangle\coloneqq(1+x^{2})^{\frac{1}{2}}$. We define
the smooth even cut-off $\chi_{R}$ by $\chi_{R}(x)=\chi(R^{-1}x)$,
$\chi(x)=1$ for $|x|\leq1$, and $\chi(x)=0$ for $|x|\geq2$. We
also denote the sharp cut-off on a set $A$ by $\mathbf{1}_{A}$.
Denote by $\delta_{jk}$ the Kronecker-Delta symbol, i.e., $\delta_{jk}=1$
if $j=k$ and $\delta_{jk}=0$ if $j\neq k$. 
We denote by some functions $\calK_j$ for $j=1,2,3,4$ and $\mathring \calK_i$ for $i=1,2$,
\begin{equation}
    \begin{gathered}
        \calK_1=\Lambda Q, \quad \calK_2=iQ, \quad \calK_3=i(1+y^2)Q, \quad \calK_4=(1+y^2)Q,
        \\
        \mathring \calK_1=iyQ, \qquad \mathring \calK_2=yQ.
    \end{gathered}\label{eq:kernel element}
\end{equation}
For any $x\coloneqq(x_1,\cdots,x_d)\in\bbR^d$, we set $|x|^2=x_1^2+\cdots x_d^2$ and
\begin{align*}
    \calB^d_{\epsilon}(x)\coloneqq\{y\in\bbR^d: |x-y|\leq \epsilon\},\quad \partial\calB^d_{\epsilon}(x)\coloneqq\{y\in\bbR^d: |x-y|= \epsilon\}.
\end{align*}
In particular, we denote by $\calB^d\coloneqq  \calB^d_{1}(0)$, $\partial\calB^d\coloneqq  \partial\calB^d_{1}(0)$, and $a\calB^d\coloneqq  \calB^d_{a}(0)$ for real $a>0$.

Denote by $L^{p}(\mathbb{R})$ and $H^{s}(\bbR)$ the standard $L^{p}$
and Sobolev spaces on $\bbR$. As we work on $\R$, we often omit
$\R$ when there is no confusion. Let $X_e$ and $X_o$ denote the even and odd projections of the function space $X$, respectively.

The Fourier transform (on $\bbR$) is denoted by 
\begin{align*}
	\mathcal{F}(f)(\xi)=\widehat{f}(\xi)\coloneqq\int_{\mathbb{R}}f(x)e^{-ix\xi}dx.
\end{align*}
The inverse Fourier transform is then given by $\mathcal{F}^{-1}(f)(x)\coloneqq\tfrac{1}{2\pi}\int_{\R}\wt f(\xi)e^{ix\xi}d\xi$.
We denote $|D|$ by the Fourier multiplier with symbol $|\xi|$, that
is, $|D|\coloneqq\mathcal{F}^{-1}|\xi|\mathcal{F}$. We denote by
$\mathcal{H}$ the Hilbert transform: 
\begin{align}
	\mathcal{H}f\coloneqq\left(\frac{1}{\pi}\text{p.v.}\frac{1}{x}\right)*f=\mathcal{F}^{-1}(-i\text{sgn}(\xi))\mathcal{F}f.\label{eq:HilbertFourierFormula}
\end{align}
In view of \eqref{eq:HilbertFourierFormula}, we have $\partial_{x}\mathcal{H}f=|D|f=\mathcal{H}\partial_{x}f$
for $f\in H^{1}$. 

We mostly use the \emph{real inner product} defined by 
\begin{align*}
	(f,g)_{r}\coloneqq\Re\int_{\R}f\overline{g}dx.
\end{align*}
We also use the modulated form of a function $f$ by parameters $\lmb$,
$\gmm$, and $x$: 
\begin{equation*}
	[f]_{\la,\ga,x}\coloneqq\frac{e^{i\ga}}{\la^{1/2}}f\lr{\frac{\cdot-x}{\la}}.
\end{equation*}
We denote by $\Lambda_{s}$ the $\dot{H}^{s}$-scaling generator in
$\R$ as 
\begin{align*}
	\Lambda_{s}f  \coloneqq\frac{d}{d\lambda}\bigg|_{\lambda=1}\lambda^{\frac{1}{2}-s}f(\lambda\cdot)=\left(\frac{1}{2}-s+x\partial_{x}\right)f,\quad 
	  \Lmb\coloneqq\Lmb_{0}.
\end{align*}
We use the adapted Sobolev spaces $\dot{\mathcal{H}}^{k}$ for $k\in \bbN$ whose norm is defined by 
\begin{align*}
	\|f\|_{\dot{\mathcal{H}}^{k}}^{2}  \coloneqq \sum_{j=0}^{k}\left\Vert \langle x\rangle^{-j}\partial_{x}^{k-j}f\right\Vert_{L^{2}}^{2}.
\end{align*}
Note that $\dot{\mathcal{H}}^{k}\hookrightarrow\dot{H}^{k}$ but $\dot{\mathcal{H}}^{k}\cap L^{2}=H^{k}$.
We also note that $\dot{\mathcal{H}}^{k}$ is compactly embedded in $H_{\textnormal{loc}}^{k-1}$.

We denote by some linear operators (at $v$ or $Q$)
\begin{align*}
    \begin{gathered}
        \mathbf{D}_{v}f  \coloneqq\partial_{x}f+\frac{1}{2}\mathcal{H}(|v|^{2})f,\quad
        \widetilde{\mathbf{D}}_{v}f  \coloneqq\partial_{x}f+\frac{1}{2}v\mathcal{H}(\overline{v}f)
        \\
        L_{v}f \coloneqq\partial_{x}f+\frac{1}{2}\mathcal{H}(|v|^{2})f+v\mathcal{H}(\text{Re}(\overline{v}f)), \; L_{v}^{*}f  \coloneqq-\partial_{x}f+\frac{1}{2}\mathcal{H}(|v|^{2})f-v\mathcal{H}(\text{Re}(\overline{v}f)).
    \end{gathered}
\end{align*}
Here, $L_{v}^{*}$ is the $L^{2}$-adjoint operator of $L_{v}$ with respect
to $(\cdot,\cdot)_{r}$. We denote by an expression of nonlinear part
\begin{align*}
    N_{v}(\eps) & \coloneqq\eps\mathcal{H}(\Re(\overline{v}\eps))+\tfrac{1}{2}(v+\eps)\mathcal{H}(|\eps|^{2}).
\end{align*}
A second order operator $H_{v}$ is given by 
\begin{align*}
    H_{v}f\coloneqq\left(-\partial_{xx}+\frac{1}{4}|v|^{4}-v|D|\overline{v}\right)f.
\end{align*}
Another linear operators $A_{Q}, B_Q$, and their adjoint operators $A_{Q}^{*},B_{Q}^{*}$ are defined by 
\begin{align*}
    \begin{gathered}
        A_{Q}f\coloneqq\partial_{x}(x-\mathcal{H})\langle x\rangle^{-1}f,\quad 
        A_{Q}^{*}f\coloneqq-\langle x\rangle^{-1}(x+\mathcal{H})\partial_{x}f
        \\
        B_{Q}f\coloneqq(x-\mathcal{H})\langle x\rangle^{-1}f,\quad 
        B_{Q}^{*}f\coloneqq\langle x\rangle^{-1}(x+\mathcal{H})f.
    \end{gathered}
\end{align*}
We note that $A_{Q}=\partial_{x}B_{Q}$ and $A_{Q}^{*}=-B_{Q}^{*}\partial_{x}$.

We introduce some algebraic identities with respect to $\calH$, 
\begin{align}
    \mathcal{H}\left(\frac{1}{1+x^2}\right)=\frac{x}{1+x^2}, \quad
    \textnormal{or} \quad
    \calH(Q^2)=xQ^2, \label{eq:algebraic 1}
\end{align}
and we also have
\begin{align}
    \mathcal{H}\left(\frac{2x^{2}}{(1+x^2)^{2}}\right)=\frac{x^{3}-x}{(1+x^2)^{2}},
    \quad
    \mathcal{H}\left(\frac{2x^{3}}{(1+x^2)^{2}}\right)=-\frac{3x^{2}+1}{(1+x^2)^{2}}. \label{eq:algebraic 2}
\end{align}

\begin{lem}[Formulas for Hilbert transform]
\label{LemmaHilbertUsefulEquation} We have the following: 
\end{lem}

\begin{enumerate}
\item For $f,g\in H^{\frac{1}{2}+}$, we have 
\begin{align}
fg=\mathcal{H}f\cdot\mathcal{H}g-\mathcal{H}(f\cdot\mathcal{H}g+\mathcal{H}f\cdot g)\label{eq:HilbertProductRule}
\end{align}
in a pointwise sense. 
\item For $f\in\langle x\rangle^{-1}L^{2}$, we have 
\begin{align}
[x,\mathcal{H}]f(x)=\frac{1}{\pi}\int_{\mathbb{R}}f(y)dy.\label{eq:CommuteHilbert}
\end{align}
\item If $f\in H^{1}$ and $\partial_{x}f\in\langle x\rangle^{-1}L^{2}$,
then we have 
\begin{align}
\partial_{x}[x,\mathcal{H}]f=[x,\mathcal{H}]\partial_{x}f=0,\quad\text{i.e.,}\quad[\mathcal{H}\partial_{x},x]f=\mathcal{H}f.\label{eq:CommuteHilbertDerivative}
\end{align}
\end{enumerate}
The
proof of Lemma \ref{LemmaHilbertUsefulEquation} can be found in Appendix of \cite{KimKimKwon2024arxiv}.

We present the Lax pair structure for \eqref{CMdnls-gauged}, along with the hierarchy of conservation laws, which serves as our key ingredient. This structure was rigorously proved for the first time in \cite{GerardLenzmann2024CPAM}.
\begin{prop}[Integrability for \eqref{CMdnls-gauged}]
\label{prop:Lax} For $v\in C([0,T];H^{s})$ which
solves \eqref{CMdnls-gauged} with $s\geq2$, the followings hold true:
\begin{enumerate}
    \item(Lax equation) We have
    \begin{align}
        \partial_{t}\widetilde{\mathbf{D}}_{v}=[-iH_{v},\widetilde{\mathbf{D}}_{v}].\label{eq:LaxEqu Unconditional}
    \end{align}
    \item(Hierarchy of conservation laws) For $j\in \bbN$ with $0\leq j\leq 2s$, the quantities
    \begin{align}
        I_j(v)\coloneqq(\td \bfD_v^j v,v)_r \label{eq:hierarchy}
    \end{align}
    are conserved.

    \item Formally, we have
    \begin{gather}
        \partial_{t}(\widetilde{\mathbf{D}}_{v}^jv)+iH_{v}\widetilde{\mathbf{D}}_{v}^jv=0, \label{eq:vk equ}
        \\
    	\widetilde{\mathbf{D}}_{v}^2h=-H_{v}h+\tfrac{1}{2}\widetilde{\mathbf{D}}_{v}v\mathcal{H}(\overline{v}h)-\tfrac{1}{2}v\mathcal{H}(\overline{\widetilde{\mathbf{D}}_{v}v}h)\label{eq:Dv square}
    \end{gather}
\end{enumerate}
\end{prop}
For the proof of $(1)$ and \eqref{eq:Dv square}, see Proposition 3.5 and the proof of Proposition 3.7 in \cite{KimKimKwon2024arxiv}, respectively. The equation \eqref{eq:vk equ} follows from a straightforward calculation using the equation for $v$ and \eqref{eq:LaxEqu Unconditional}. The proof of $(2)$ is standard. See \cite{GerardLenzmann2024CPAM}, Lemma 2.4 for example.

\begin{prop}[Conjugation identity]\label{prop:AQ-BQ definition}
    The following hold true: 
    \begin{enumerate}
        \item (Conjugation identity) We have 
        \begin{align}
    		L_QiL_Q^*f=iH_Qf=iA_{Q}^*A_{Q}f=i\td\bfD_Q^*\td\bfD_Qf \label{eq:conjugate identity}
    	\end{align}
        for sufficiently good $f$, i.e., $f\in H^{2}$. 
        \item (Repulsivity) For $\partial_{x}f\in H^{1}$, 
        \begin{align}
        A_{Q}A_{Q}^{*}f=-\partial_{xx}f, \label{eq:repulsivity}
        \end{align}
        \item (Kernel for $A_{Q}$ on $\dot{\mathcal{H}}^{1}$) Let $v\in\dot{\mathcal{H}}^{1}$
        with $A_{Q}v=0$. Then, 
        \begin{align}
        v\in\textnormal{span}_{\mathbb{C}}\{Q,xQ\}.\label{eq:kernel AQ}
        \end{align}
    \end{enumerate}
    For $B_{Q}$, the following hold true: 
    \begin{enumerate}
        \item[(4)] \setcounter{enumi}{4} $A_{Q}=\partial_{x}B_{Q}$ and $A_{Q}^{*}=-B_{Q}^{*}\partial_{x}$.
        In addition, $A_{Q}^{*}A_{Q}=-B_{Q}^{*}\partial_{xx}B_{Q}$. 
        \item We have 
        \begin{align}
        B_{Q}^{*}B_{Q}f=f-\frac{1}{2\pi}Q\int_{\R}Qfdx,\quad B_{Q}B_{Q}^{*}=I.\label{eq:BQBQstar equal I}
        \end{align}
    \end{enumerate}
    Furthermore, the linear operator $\widetilde{\D}_Q$ satisfies the following decomposition:
    \begin{enumerate}
        \item[(6)] For $f\in H^1$,
        \begin{align}\label{eq:DQ BQ decomp}
        	\widetilde{\D}_Qf=B_Q^*\partial_xB_Qf. 
        \end{align}
    \end{enumerate}
\end{prop}
Motivated by \cite{KimKwon2020blowup,KimKwonOh2020blowup}, the authors of \cite{KimKimKwon2024arxiv} first identified $L_QiL_Q^*=i\td\bfD_Q^*\td\bfD_Q $ by linearizing \eqref{eq:vk equ} around $Q$. However, since the supersymmetric conjugate does not exhibit repulsive dynamics, they derived an alternative conjugation identity $L_QiL_Q^*=iA_Q^*A_Q$ with the repulsivity \eqref{eq:repulsivity}. Additionally, recognizing that a derivative can be separated from $A_Q$, they introduced $B_Q$ and computed \eqref{eq:BQBQstar equal I}. We refer to \cite{KimKimKwon2024arxiv} for the original motivation and derivation of the operators $A_Q$ and $B_Q$.
Building on $i \td\bfD_Q^* \td\bfD_Q = i A_Q^* A_Q$, we observe that \eqref{eq:DQ BQ decomp} holds. This new observation, developed in the present work, allows us to utilize the hierarchy of conservation laws \eqref{eq:hierarchy} when the solution is near $Q$.

\begin{proof}[Proof of Proposition \ref{prop:AQ-BQ definition}]
    The proof of (1)--(5) is given in \cite{KimKimKwon2024arxiv}, Proposition 3.8, and the identity in (6), first observed in the present work, is also recorded there. Since (6) is one of the new observations of this paper, we include its short proof here. 
    
    Collecting terms in $B_Q^*\partial_xB_Qf$ without the Hilbert transform $\calH$, we have
    \begin{align*}
        x\langle x \rangle^{-1}\partial_x(x\langle x \rangle^{-1}f)
        +\langle x \rangle^{-1}\calH \partial_x(-\calH)(\langle x \rangle^{-1}f)
        =
        \partial_x f.
    \end{align*}
    Here, we used $\calH \partial_x(-\calH)=-\calH^2 \partial_x=\partial_x$.
    For the terms with the Hilbert transform, using \eqref{eq:CommuteHilbertDerivative}, we conclude \eqref{eq:DQ BQ decomp} by
    \begin{align*}
        &\langle x \rangle^{-1}\calH\partial_x(x\langle x \rangle^{-1}f)
        -x\langle x \rangle^{-1}\partial_x\calH(\langle x \rangle^{-1}f)
        \\
        &=\langle x \rangle^{-1}\partial_x[\calH,x](\langle x \rangle^{-1}f)
        +
        \langle x \rangle^{-1}\calH(\langle x \rangle^{-1}f)
        =\langle x \rangle^{-1}\calH(\langle x \rangle^{-1}f)=\tfrac{1}{2}Q\calH(Qf). \qedhere
    \end{align*}
\end{proof}

\section{Profiles for the nonlinear variables}\label{sec:profile}
To employ the hierarchy of conservation laws, we use the nonlinear adapted variables $\td\bfD_v^kv$, with decomposition and modulation analysis performed in terms of these variables. In this section, we construct the blow-up profiles for the nonlinear variables, and establish the ODE dynamics of $b$ and $\eta$.

\subsection{Linearized dynamics}\label{sec:linearized dynamics}
In this subsection, we review the linearization of \eqref{CMdnls-gauged} around $Q$. For more details, refer to \cite{KimKimKwon2024arxiv}, Section 3. In addition, we find appropriate blow-up profiles for $v$ and the higher-order variables $\td\bfD_v^kv$. 

From the form of energy in \eqref{eq:intro energy gauge}, we rewrite the energy as 
\begin{align*}
E(v)=\frac{1}{2}\int_{\mathbb{R}}\left|\mathbf{D}_{v}v\right|^{2}dx.
\end{align*}
Similarly to $\mathcal{R}$, the function $Q$ is the unique solution to the Bogomol'ny\u{\i} equation up to symmetries, $\bfD_{Q}Q=0$. We linearize the map
$v\mapsto\mathbf{D}_{v}v$ as follows:
\begin{align*}
\mathbf{D}_{v+\eps}(v+\eps)=\mathbf{D}_{v}v+L_{v}\eps+N_{v}(\eps).
\end{align*}
Using the adjoint operator $L_{v}^{\ast}$, we have $i\nabla E(v)=iL_{v}^{*}\bfD_{v}v$.
Thus, we linearize \eqref{CMdnls-gauged} as 
\begin{align*}
iL_{w+\eps}^{*}\bfD_{w+\eps}(w+\eps)=iL_{w}^{*}\bfD_{w}w+i\calL_{w}\eps+R_{w}(\eps)
\end{align*}
where $i\calL_{w}\eps$ is the linear part, and $R_{w}(\eps)$ is
the nonlinear part for $\eps$. If one chooses $w=Q$, then using $\bfD_{Q}Q=0$,
one derive the self-dual factorization 
\begin{align*}
i\calL_{Q}=iL_{Q}^{*}L_{Q}.
\end{align*}
To modulate out the kernel directions of $i\mathcal{L}_{Q}$ and establish the coercivity of the orthogonal component, we first recall the kernel properties and the coercivity estimate from \cite{KimKimKwon2024arxiv}. Since our analysis is conducted on the even function spaces, the kernel space of $L_Q$ is determined as follows:
\begin{align*}
    \textnormal{ker}\,L_{Q}=\textnormal{ker}\,\mathcal{L}_{Q}
    =\textnormal{span}_{\mathbb{R}}\{\calK_1,\calK_2\} \textnormal{ on even function spaces}    
\end{align*}
where $\calK_j$'s are given by \eqref{eq:kernel element} and $\Lambda$ is the $L^{2}$-scaling generator, $\Lambda f:=\frac{f}{2}+x\partial_{x}f$. Here $\calK_3$ is defined as $i(1+y^2)Q$, whereas $iy^2Q$ was used in \cite{KimKimKwon2024arxiv}; this does not affect the kernel structure, since $\calK_1=iQ$ already accounts for the phase direction. The motivation for this choice will be explained in Section~\ref{sec:tail compute}. Each kernel relation is derived from the symmetries of phase rotation and scaling. If no even assumption is made, the function $\partial_x Q$, arising from translation, also becomes a kernel element.

Under the rough decomposition $v=Q+\wt\eps$, the nonlinear variable $\td\bfD_{v}^{j}v= v_j$ can be linearized as
\begin{align}
    v_j\eqqcolon \td\bfD_{v}^{j-1}(\td\bfD_{v}v)
    \approx \td\bfD_{Q}^{j-1}L_Q\wt\eps
    =B_Q^*\partial_y^{j-1}B_Q L_Q\wt\eps. \label{eq:non variable linear rough}
\end{align}
However, $L_Q$ is not defined on $\dot \calH^j$ for $j\geq2$ because of the Hilbert transform $\calH$. This leads us to define new linear operators.
We observe that
\begin{align*}
	&B_QL_Q\Re=(1+\partial_yy
	-\calH\partial_y)\langle y\rangle^{-1}\Re,
	\\
	&B_QL_Qi\Im=B_Q\bfD_{Q}i\Im=(y-\calH)\langle y\rangle^{-2}\partial_y\langle y\rangle i\Im.
\end{align*}
For $\wt\eps \in H^{2L}$ and $1\leq j \leq 2L$, we derive
\begin{align*}
	\begin{split}
		\partial_y^{j-1}B_QL_Q\Re\wt\eps
		&=
		\partial_y^{j-1}(1+\partial_yy)(\langle y\rangle^{-1}\Re\wt\eps)
		-\calH\partial_y^{j}(\langle y\rangle^{-1}\Re\wt\eps),
		\\
		\partial_y^{j-1}B_QL_Q i\Im\wt\eps
		&=\partial_y^{j-1}[y\langle y\rangle^{-2}\partial_y(\langle y\rangle i\Im\wt\eps)]-\calH\partial_y^{j-1}[\langle y\rangle^{-2}\partial_y(\langle y\rangle i\Im\wt\eps)].
	\end{split}
\end{align*}
In view of this, we define linear operators by
\begin{align*}
	\calL_{j}\coloneqq (\calL^{1}_j,\calL^{2}_j)=\calL^{1}_j\Re+i\calL^{2}_j \Im,
\end{align*}
where
\begin{align*}
	\begin{split}
		\calL^{1}_jf&\coloneqq 
		\partial_y^{j-1}(1+\partial_yy)(\langle y\rangle^{-1}f)
		-\calH\partial_y^{j}(\langle y\rangle^{-1}f),
		\\
		\calL^{2}_jg&\coloneqq
		\partial_y^{j-1}[y\langle y\rangle^{-2}\partial_y(\langle y\rangle g)]-\calH\partial_y^{j-1}[\langle y\rangle^{-2}\partial_y(\langle y\rangle g)].
	\end{split} 
\end{align*}
One can check that
\begin{align}
	\calL_{j}\wt\eps
	=\partial_y^{j-1}B_Q L_Q\wt\eps \quad \text{for}\quad 1\leq j\leq 2L, \quad \wt\eps \in H^{2L}. \label{eq:BQLQ and calL}
\end{align}
In \cite{KimKimKwon2024arxiv}, the authors used a coercivity lemma for $A_Q\td L_Q$ with a modified linearized operator $\td L_Q$ instead of $A_QL_Q$, since $L_Q$ is not well-defined on $\dot \calH^2$. However, we do not need to introduce such an operator, as we have $\calL_2 f = A_Q L_Q f$ for smooth $f$, and $\calL_2$ is well-defined on $\dot \calH^2$. More precisely, we have $A_Q\td L_Qf=\calL_2f$ for any $f\in \dot \calH^2$. Thus, the kernel space of $\calL_2$ is given by
\begin{align*}
    \textnormal{ker}\,\calL_{2}
    =\textnormal{span}_{\mathbb{R}}\{\calK_1,\calK_2,\calK_3,\calK_4\} \textnormal{ on even function spaces}.
\end{align*}
We note that $\partial_x Q$ and $ixQ$ are kernel elements of $\calL_2$ in the general space; see Remark~\ref{rem:kernel of Lj}.

On the other hand, we also use another linear operator, $A_Q$, which is motivated by the conjugation identity \eqref{eq:conjugate identity}. As mentioned in the introduction, instead of utilizing the repulsive property \eqref{eq:repulsivity} for $A_QA_Q^*$, we use \eqref{eq:DQ BQ decomp}. However, the coercivity of $A_Q$ remains essential. Thanks to \eqref{eq:kernel AQ}, we note the kernel for $A_Q$ in odd function spaces as follows:
\begin{align*}
    \textnormal{ker}\,A_Q
    =\textnormal{span}_{\mathbb{R}}\{ixQ,xQ\} \textnormal{ on odd function spaces}.
\end{align*}
Unlike the earlier case, the odd restriction is imposed because we will address terms of the form $A_Q v_{2k-1}$. Here, $v_{2k-1}$ is odd, as the operator $\td\bfD_v$ alternates between even and odd parity when $v$ is even.

We recall that $\calK_j$ and $\mathring \calK_i$ are given by \eqref{eq:kernel element}. Now, we introduce coercivity lemmas for our analysis.
\begin{lem}[Linear coercivity estimates]
\label{lem:coercivity for linear} The followings hold true:
\begin{enumerate}
    \item (Coercivity for $L_{Q}$ on $\dot{\mathcal{H}}^{1}_e$) Let $\psi_{1},\psi_{2}$ be elements of the dual space $(\dot{\mathcal{H}}^{1}_e)^{*}$ such that $(\psi_{i},\mathcal{K}_{j})_{r}=\delta_{i,j}$. Then, we have a coercivity estimate 
    \begin{align}
        \|v\|_{\dot{\mathcal{H}}^{1}_e}\sim\|L_{Q}v\|_{L^{2}},\quad\forall v\in\dot{\mathcal{H}}^{1}_e\cap\{\psi_{1},\psi_{2}\}^{\perp}. \label{eq:coercivity L1}
    \end{align}
    
    \item (Coercivity for $\calL_{2}$ on $\dot{\mathcal{H}}^{2}_e$) Let $\psi_{1},\psi_{2},\psi_{3},\psi_{4}$ be elements of the dual space $(\dot{\mathcal{H}}^{2}_e)^{*}$ such that $(\psi_{i},\mathcal{K}_{j})_{r}=\delta_{i,j}$. Then, we have a coercivity estimate 
    \begin{align}
        \|v\|_{\dot{\mathcal{H}}^{2}_e}\sim\|\calL_{2}v\|_{L^{2}},
        \quad\forall v\in\dot{\mathcal{H}}^{2}_e\cap\{ \psi_{1},\psi_{2},\psi_{3},\psi_{4}\}^{\perp}. \label{eq:coercivity L2}
    \end{align}

    \item (Coercivity for $A_{Q}$ on $\dot{\mathcal{H}}^{1}_o$) Let $\psi_{1},\psi_{2}$ be elements of the dual space $(\dot{\mathcal{H}}^{1}_o)^{*}$ such that $(\psi_{i},\mathring \calK_j )_{r}=\delta_{i,j}$. Then, we have a coercivity estimate 
    \begin{align}
        \|v\|_{\dot{\mathcal{H}}^{1}}\sim \|A_{Q}v\|_{L^{2}},\quad\forall v\in\dot{\mathcal{H}}^{1}_o\cap\{\psi_{1},\psi_{2}\}^{\perp}. \label{eq:coercivity A1}
    \end{align}
\end{enumerate}
\end{lem}
We note that, due to the degeneracy of the linearized operators, the adapted Sobolev spaces $\dot{\mathcal{H}}^{1}$ and $\dot{\mathcal{H}}^{2}$ are used instead of the usual ones.
\begin{proof}[Proof of Lemma \ref{lem:coercivity for linear}]
    For $(1)$ and $(2)$, the proof can be found in \cite{KimKimKwon2024arxiv}. Hence, it suffices to prove $(3)$. Again from \cite{KimKimKwon2024arxiv}, we already know that the kernel of $A_Q$ on $\dot \calH^1_o$ is the span of $\{iyQ, yQ\}$, and we have the subcoercivity (\cite{KimKimKwon2024arxiv}, Lemma A.4),
    \begin{align*}
        \|A_{Q}v\|_{L^{2}}+\||D|^{\frac{1}{2}}Qv\|_{L^{2}}\sim\|v\|_{\dot{\mathcal{H}}^{1}}.
    \end{align*}
    Using \eqref{eq:CommuteHilbertDerivative} with $(1+y^2)\langle y\rangle^{-2}=1$ and $\||D|^{\frac{1}{2}}Qv\|_{L^{2}}^2
        =(|D|Qv,Qv)_r$, we have
    \begin{align*}
        \||D|^{\frac{1}{2}}Qv\|_{L^{2}}^2
        &=(|D|Qv,\langle y\rangle^{-2}Qv)_r
        +(|D|Qv,y^2\langle y\rangle^{-2}Qv)_r
        \\
        &=(|D|Qv,\langle y\rangle^{-2}Qv)_r
        +(\calH[y\partial_y(Qv)],y\langle y\rangle^{-2}Qv)_r\lesssim \|v\|_{\dot \calH^1}\|Q^2v\|_{L^2}.
    \end{align*}
    Thus, we obtain
    \begin{align*}
        \||D|^{\frac{1}{2}}Qv\|_{L^{2}}^2
        \le C\|v\|_{\dot \calH^1}\|Q^2v\|_{L^2}\leq \tfrac{C}{10^{10}}\|v\|_{\dot \calH^1}^2+C10^{10}\|Q^2v\|_{L^2}^2.
    \end{align*}
    This means that we improve the subcoercivity estimate such as
    \begin{align*}
        \|A_{Q}v\|_{L^{2}}+\|Q^2v\|_{L^{2}}
        \sim\|v\|_{\dot{\mathcal{H}}^{1}}.
    \end{align*}
    Then, by a standard argument, we derive the coercivity \eqref{eq:coercivity A1}. For more details, we refer to Appendix in \cite{KimKimKwon2024arxiv}.
\end{proof}
In view of previous quantized blow-up results, one might expect that the coercivity of $\td\bfD_{Q}^{j-1}L_Q$ has to be established. However, we aim to bypass this by taking advantage of the hierarchy of conservation laws. Instead, we will apply the coercivity for $A_Q$ to higher-order nonlinear variables. 

In the sequel, we will use $v$ and nonlinear variables $v_{2k-1}$ for each $1\le k \le L$. Under the radial assumption, it turns out to be sufficient to track the odd order nonlinear variables $v_{2k-1}$. See remark \ref{rem:why only 2k-1}. Accordingly, we will construct profiles for $v$ and $v_{2k-1}$. Although we do not rely on the coercivity of $\calL_{2k-1}$, we still require information about its kernel to describe the initial data. 
We define
\begin{equation}
    \begin{aligned}
        T_{1}(b_1,\eta_1)\coloneqq&-\frac{1}{4}(ib_1+\eta_1) (1+y^2)Q,
    	\\
    	T_{2k-1}(b_k,\eta_k)\coloneqq&-\frac{(-i)^{k-1}}{4k\cdot(2k-2)!}(ib_k+\eta_k) y^{2k-4}(1+y^2)^2Q, \quad (k\geq 2).
    \end{aligned} \label{eq:T_k}
\end{equation}
We further define the profiles $P_{2k-1}$ for the nonlinear variable $v_{2k-1}$ by
\begin{align}
    P_{2k-1}(b_k,\eta_k)&\coloneqq-(ib_k+\eta_k) \cdot (-i)^{k-1}\cdot \tfrac{y}{2}Q. \label{eq:P profile}
\end{align}
Then, we have the following lemma:
\begin{lem}\label{lem:profile choosing}
    Let $R\gg 1$. For $1\leq 2j\leq \ell$ and $1\leq k\leq L$, we have
	\begin{align}
    \calL_{\ell}T_{2j-1}=0,\quad \calL_{2k-1}T_{2k-1}(b_k,\eta_k)=B_Q P_{2k-1}(b_{k},\eta_{k}). \label{eq:higher-order profile}
	\end{align}
    For all $1\leq \ell\leq 2L$ and $1\leq j\leq L$, we have
    \begin{align}
        \|\calL_{\ell}(T_{2j-1}\chi_R)\|_{L^2}
        \lesssim (|b_j|+|\eta_j|)R^{2j-\ell-\frac{1}{2}}. \label{eq:profile estimate}
    \end{align}
	Moreover, for $1\leq j\leq k$, we have
    \begin{align}
        \|\langle y\rangle^{-1}\calL_{2k-1}[T_{2j-1}(1-\chi_R)]\|_{L^2}
        \lesssim (|b_j|+|\eta_j|)R^{2j-2k-\frac{1}{2}},  \label{eq:local profile estimate1}
    \end{align}
    and for $k< j$, we have
    \begin{align}
        \|\langle y\rangle^{-1}\calL_{2k-1}(T_{2j-1}\chi_R)\|_{L^2}
        \lesssim (|b_j|+|\eta_j|)R^{2j-2k-\frac{1}{2}}. \label{eq:local profile estimate2}
    \end{align}
\end{lem}
\begin{proof}[Proof of Lemma \ref{lem:profile choosing}]
    Let $m, n$ be integers with $n\geq 2$. We first claim that 
	\begin{align}
		\calL_{n}^{1}(y^{m-1}(1+y^2)Q)=\calL_{n}^{2}(y^{m+1}Q)=0\quad \text{for}\quad 1\leq m\leq n-1,\label{eq:kernel lemma claim}
	\end{align}
    and
    \begin{align*}
        \calL_{n}^{1}(y^{n-1}(1+y^2)Q)=\calL_{n}^{2}(y^{n-1}(1+y^2)Q)=(n+1)\cdot (n-1)!B_Q(yQ).
    \end{align*}
	For $\calL_{n}^{1}(y^{m-1}(1+y^2)Q)=0$, we have
	\begin{align*}
		&\calL_{n}^{1}(y^{m-1} (1+y^2)Q) 
		\\
		&=\partial_y^{n-1}(1+\partial_yy)(\langle y\rangle^{-1}y^{m-1} (1+y^2)Q)
		-\calH\partial_y^{n}(\langle y\rangle^{-1}y^{m-1} (1+y^2)Q)
		\\
		&=\sqrt{2}[\partial_y^{n-1}(y^{m-1})+\partial_y^{n}(y^{m})-\calH\partial_y^{n}(y^{m-1})]=0.
	\end{align*}
	If we take $m=n$, we derive
	\begin{align}
		\calL_{n}^{1}(y^{n-1} (1+y^2)Q) 
		=\sqrt{2}(n-1)!+\sqrt{2} n!=(n+1)\cdot (n-1)!B_Q(yQ). \label{eq:higher-order profile pf1}
	\end{align}
	To obtain $\calL_{n}^{2}(y^{m+1}Q)=0$, we compute by
	\begin{align*}
		\calL_{n}^{2}(y^{m+1}Q) 
		&=\partial_y^{n-1}[y\langle y\rangle^{-2}\partial_y(\langle y\rangle y^{m+1}Q)] 
		-\calH\partial_y^{n-1}[\langle y\rangle^{-2}\partial_y(\langle y\rangle y^{m+1}Q)] 
		\\
		&=\sqrt{2}(m+1)[\partial_y^{n-1}(y^{m+1}\langle y\rangle^{-2})-\calH\partial_y^{n-1}(y^{m}\langle y\rangle^{-2})] 
		\\
		&\eqqcolon\mathcal{P}_{m,n}.
	\end{align*}
	We claim $\mathcal{P}_{m,n}=0$ by the induction for $m$ when $1\leq m\leq n-1$. For $m=1$, since $y\langle y\rangle^{-2}\in H^{n}$, we have
	\begin{align*}
		\calH\partial_y^{n-1}(y\langle y\rangle^{-2})
		=\partial_y^{n-1}\calH(y\langle y\rangle^{-2})
		=-\partial_y^{n-1}\langle y\rangle^{-2}.
	\end{align*}
	Thus, we derive
	\begin{align*}
		\mathcal{P}_{1,n}= 2\sqrt{2}\partial_y^{n-1}[y^2\langle y\rangle^{-2}+\langle y\rangle^{-2}]=2\sqrt{2}\partial_y^{n-1}[1]=0.
	\end{align*}
	If $n=2$, the claim holds true. For $n\geq 3$ and $m=2$, we have
	\begin{align*}
		\calH\partial_y^{n-1}(y^{2}\langle y\rangle^{-2})
		=
		\calH\partial_y^{n-1}((1+y^{2})\langle y\rangle^{-2})-\calH\partial_y^{n-1}(\langle y\rangle^{-2})
		=-\calH\partial_y^{n-1}(\langle y\rangle^{-2}).
	\end{align*}
	Again using $y\langle y\rangle^{-2}\in H^{n}$, we deduce
	\begin{align*}
		\mathcal{P}_{2,n}=
		3\sqrt{2}\partial_y^{n-1}[y^3\langle y\rangle^{-2}+y\langle y\rangle^{-2}]=3\sqrt{2}\partial_y^{n-1}[y]=0.
	\end{align*}
	Now, to use induction, we suppose $\mathcal{P}_{2j-1,n}=\mathcal{P}_{2j,n}=0$ for $2j<n-1$. Then, we have
	\begin{align*}
		\mathcal{P}_{2j+1,n}&=\mathcal{P}_{2j+1,n}+\tfrac{2j+2}{2j}\mathcal{P}_{2j-1,n}
		\\
		&=
		\sqrt{2}(2j+2)[\partial_y^{n-1}(y^{2j}(1+y^2)\langle y\rangle^{-2})-\calH\partial_y^{n-1}(y^{2j-1}(1+y^2)\langle y\rangle^{-2})]
		\\
		&=2\sqrt{2}(j+1)[\partial_y^{n-1}(y^{2j})-\calH\partial_y^{n-1}(y^{2j-1})]
		=0.
	\end{align*}
	Similarly, we deduce $\mathcal{P}_{2j+2,n}=0$, and this leads to us $\calL_{n+1}^{2}(y^{m}Q)=0$ for all $1\leq m\leq n-1$. Therefore, we arrive at \eqref{eq:kernel lemma claim}. On the other hand, if $m=n$, by using $ B_Q(yQ)=\sqrt{2}$, we observe that
	\begin{align}
		\begin{split}
			\mathcal{P}_{n,n}=\mathcal{P}_{n,n}+\tfrac{n+1}{n-1}\mathcal{P}_{n-2,n}
			&=\sqrt{2}(n+1)[\partial_y^{n-1}(y^{n-1})-\calH\partial_y^{n-1}(y^{n-2})]
			\\
			&=(n+1)\cdot (n-1)!B_Q(yQ).
		\end{split}
		\label{eq:higher-order profile pf2}
	\end{align}
    We also note that, from \eqref{eq:BQLQ and calL} with $L_Q(iQ)=0$, we have
    \begin{align}
        \calL_{n}(iQ)=\calL_n^2(Q)=0. \label{eq:kernel calL iQ}
    \end{align}
    Combining this, \eqref{eq:kernel lemma claim}, \eqref{eq:higher-order profile pf1}, \eqref{eq:higher-order profile pf2}, and the definition of the profiles $T_{2k-1}$, we conclude \eqref{eq:higher-order profile}.

    Now, we show \eqref{eq:profile estimate}. For the case $2j-1<\ell$, \eqref{eq:kernel lemma claim} and \eqref{eq:kernel calL iQ} yield
    \begin{align*}
        \|\calL_{\ell}[T_{2j-1}\chi_R]\|_{L^2}
        &=\|\calL_{\ell}[T_{2j-1}(1-\chi_R)]\|_{L^2}
        \\
        &\lesssim
        (|b_j|+|\eta_j|)\|\langle y\rangle^{2j-1}(1-\chi_R)\|_{\dot\calH^{\ell}}
        \lesssim (|b_j|+|\eta_j|)R^{2j-\ell-\frac{1}{2}}.
    \end{align*}
    When $2j-1\geq \ell$, we also get
    \begin{align*}
        \|\calL_{\ell}[T_{2j-1}\chi_R]\|_{L^2}\lesssim (|b_j|+|\eta_j|)\|\langle y\rangle^{2j-1}\chi_R\|_{\dot\calH^{\ell}}
        \lesssim (|b_j|+|\eta_j|)R^{2j-\ell-\frac{1}{2}}.
    \end{align*}
    Similarly, for \eqref{eq:local profile estimate1}, using the definition of $\calL_n$ and the fact $\|\langle y\rangle^{-1}\calH(f)\|_{L^2}\leq\|f\|_{L^2}$, we have
	\begin{align*}
		\|\langle y\rangle^{-1}\calL_{\ell}[T_{2j-1}(1-\chi_R)]\|_{L^2}
        &\lesssim
        (|b_j|+|\eta_j|)\|\langle y\rangle^{2j-2}(1-\chi_R)\|_{\dot\calH^{\ell}}
        \\
        &\lesssim (|b_j|+|\eta_j|)R^{2j-1-\ell-\frac{1}{2}}.
	\end{align*}
    Thus, taking $\ell=2k-1$, we derive \eqref{eq:local profile estimate1}. The proof of \eqref{eq:local profile estimate2} can be done by a similar argument, so we omit this.
\end{proof}
\begin{rem}\label{rem:kernel of Lj}
    In fact, we can infer the non-radial kernel of $\calL_j$ as follows: for $j\geq 2$,
    \begin{align}
        \textnormal{ker}\calL_{j}
        =\textnormal{span}_{\mathbb{R}} \{iQ,\Lambda Q,Q_y,iyQ, \iota y^{\ell-1}(1+y^2)Q: 1\leq \ell\leq j-1, \, \iota=1,i\}. \label{eq:general kernel nonrad}
    \end{align}
    A rigorous characterization of the kernel follows the same Fourier-side ODE argument as in \cite{KimKimKwon2024arxiv}. After setting $h=\langle y \rangle^{-1}f$, the equation $\calL_j f=0$ rewrites as a nonlocal ODE involving only $y$, $\partial_y$, and $\mathcal H$, with polynomial coefficients. Taking the Fourier transform then reduces it to a local ODE on each of the half-lines $\xi>0$ and $\xi<0$, with the only singularity at $\xi=0$ arising from the jump of $\operatorname{sgn}(\xi)$. Imposing parity and $\dot{\mathcal H}^j$ regularity then determine the kernel completely.

    However, since we do not utilize the coercivity of $\calL_j$ and consider only even data, Lemma \ref{lem:profile choosing} alone suffices. 
\end{rem}
\begin{rem}[Odd-order nonlinear variables]\label{rem:why only 2k-1}
    In the radial setting, we only use the odd-indexed parameters $v_{2k-1}$, $P_{2k-1}$, and $T_{2k-1}$.
    In view of \eqref{eq:general kernel nonrad} and Lemma \ref{lem:profile choosing}, if radial symmetry is not assumed, the profiles $P_{2k}$ for $v_{2k}$ are also needed. However, the radial symmetry removes the kernel elements $\iota y^{j-1}(1+y^2)Q$, $\iota=1,i$ for even $j\geq0$. Therefore, the profiles for $v_{2k}$ become zero, and it suffices to use $v_{2k-1}$. 
\end{rem}

\subsection{Dynamical laws of modulation parameters}\label{sec:dynamical laws} 
Building on the previous section, we aim to derive the formal dynamics of the modulation parameters. We change the original coordinates $(t,x)$ into the renormalized coordinates $(s,y)$. We define $(s,y)$ and the renormalized solution $w$ by 
\begin{align}
	s(t)=s_0+\int_0^t\frac{1}{\lambda(\tau)^{2}}d\tau,\quad y=\frac{x}{\lambda},\quad w(s,y)=\lambda^{\frac{1}{2}}e^{-i\gamma}v(t,\lambda y)|_{t=t(s)}. \label{eq:renormal}
\end{align}
Here, we abuse notation by writing $\lambda(s) = \lambda(t(s))$ and $\gamma(s) = \gamma(t(s))$, with analogous expressions for other parameters, treating them as functions of $s$ for simplicity.
The (renormalized) nonlinear adapted derivative $w_{k}$ is
defined by 
\begin{align}
	w_{k}\coloneqq\td\bfD_{w}^kw=\lambda^{\frac{2k+1}{2}} e^{-i\gamma}(v_{k})(t,\lambda y)|_{t=t(s)}. \label{eq:wk def}
\end{align}
We rewrite \eqref{CMdnls-gauged} and \eqref{eq:vk equ} as renormalized equations,
\begin{align}
	(\partial_s-\frac{\lambda_s}{\lambda}\Lambda +\gamma_s i)w+iL_w^*w_1&=0, \label{eq:w equ}
	\\
	(\partial_s-\frac{\lambda_s}{\lambda}\Lambda_{-(2k-1)} +\gamma_s i)w_{2k-1}+iH_ww_{2k-1}&=0. \label{eq:w2k-1 equ}
\end{align}
Using \eqref{eq:Dv square}, the equation \eqref{eq:w2k-1 equ} can be rewritten by
\begin{align}
    \begin{split}
        (\partial_s-\frac{\lambda_s}{\lambda}\Lambda_{-(2k-1)} +\gamma_s i)w_{2k-1}-iw_{2k+1}
        =-\frac{i}{2} \left[w_1\mathcal{H}(\overline{w}w_{2k-1})-w\mathcal{H}(\overline{w_1}w_{2k-1}) \right].    
    \end{split}
    \label{eq:w2k-1 equ modify}
\end{align}
In view of the previous section, we use the decompositions
\begin{align*}
    w=Q+\wt\eps=Q+T_1+ \eps ,\quad w_{2k-1}=P_{2k-1}+\eps_{2k-1}.
\end{align*}
Substituting this into \eqref{eq:w equ} and \eqref{eq:w2k-1 equ modify}, and following the computation in Section~4 of \cite{KimKimKwon2024arxiv}, we derive
\begin{align*}
    \eqref{eq:w equ}\approx
    -\left(\frac{\lambda_{s}}{\lambda}+b_1\right)\Lambda Q+\left(\gamma_{s}-\frac{\eta_1}{2}\right)iQ+\text{h.o.t}.
\end{align*}
Thus, we obtain the first modulation laws for $\la,\ga$; 
\begin{align}
	\frac{\lambda_{s}}{\lambda}+b_1=0,\quad\gamma_{s}-\frac{\eta_1}{2}=0. \label{eq:first mod equ}
\end{align}
Moreover, adopting the convention $Q\mathcal{H}(1) = 0$, for $1\leq k\leq L$, we formally deduce
\begin{align*}
	\eqref{eq:w2k-1 equ} \approx& 
	((b_k)_s-b_{k+1}+(2k-\tfrac{1}{2})b_1b_{k}+\tfrac{1}{2}\eta_1\eta_k)(-i)^{k}\tfrac{y}{2}Q
	\\
	&-
	((\eta_k)_s-\eta_{k+1}+(2k-\tfrac{1}{2})b_1\eta_{k}-\tfrac{1}{2}\eta_1b_k)(-i)^{k-1}\tfrac{y}{2}Q+\text{h.o.t},
\end{align*}
with $b_{L+1}, \eta_{L+1}\equiv 0$. Hence, the formal dynamical laws of $b=(b_1,\dots,b_L)$ and $\eta=(\eta_1,\dots,\eta_L)$ are given as follows: for $1\le k \le L$,
\begin{align}
    \begin{split}
        (b_k)_s &= b_{k+1}-(2k-\tfrac{1}{2})b_1b_{k}-\tfrac{1}{2}\eta_1\eta_k, \quad b_{L+1}\coloneqq0,
        \\
        (\eta_k)_s&=\eta_{k+1}-(2k-\tfrac{1}{2})b_1\eta_{k}+\tfrac{1}{2}\eta_1b_k, \quad \eta_{L+1}\coloneqq 0.
    \end{split} \label{eq:mod eqn}
\end{align}
The formal laws \eqref{eq:first mod equ} and \eqref{eq:mod eqn} are justified in the rigorous form in Lemma~\ref{lem:mod esti}.

Under the assumption $\eta_k^e(s)\equiv 0$ for all $1\leq k\leq L$, the ODE system \eqref{eq:mod eqn} has $L$ linearly independent solutions. If $b_{j+1}\equiv 0$, $k\leq j\leq L$ for some $k$, this is equivalent to having chosen $k=L$. Thus, we may restrict our attention to the case $b_{L}\not\equiv0$. Then, the system \eqref{eq:mod eqn} has a special solution as follows:
\begin{lem}[Special solutions for the $(b,\eta)$ system]
	The vector of functions 
	\begin{equation}
		b_k^e(s) = \frac{c_k}{s^k},\quad \eta_k^e(s)\equiv 0 \quad\textnormal{for} \quad 1\le k \le L \label{eq:special sol}
	\end{equation}
	solves \eqref{eq:mod eqn} where the sequence $(c_k)_{k=1,\dots,L}$ is given by
	\begin{equation}
		c_1 = \frac{2L}{4L-1},\quad c_{k+1} = -\frac{L-k}{4L-1}c_k,\quad 1\le k \le L. \label{eq:ck}
	\end{equation}
\end{lem}
A direct calculation verifies this lemma. Moreover, from \eqref{eq:first mod equ}, we have
\begin{align}
	 \lambda\sim s^{-\frac{2L}{4L-1}}, \quad \lambda\sim (b_1^e)^{\frac{2L}{4L-1}},\quad |b_j^e|\sim (b_1^e)^j \quad \text{for}\quad 1\leq j\leq L. \label{eq:b lambda relation sec3}
\end{align}
By undoing the renormalization and reverting $s$ back to $t$, we derive the quantized blow-up rate $\lambda(t)\sim (T-t)^{2L}$.
Based on these observations, we establish the quantized blow-up dynamics by constructing solutions whose modulation parameters closely follow the special ODE solution $(b^e,\eta^e)$.
To simplify the proof, we examine a fluctuation $(U,V)$ around  $(b^e,\eta^e)$. We define $U=(U_1,\dots,U_L)$ and $V=(V_1,\dots,V_L)$ as follows:
\begin{equation}
	U_k(s) \coloneqq s^k (b_k(s)-b_k^e(s)),\quad V_k(s) \coloneqq s^k \eta_k(s) \textnormal{ for } 1\le k \le L. \label{def:UV}
\end{equation}
Then, the fluctuation $(U,V)$ satisfies the lemma below:
\begin{lem}
	Let $(b(s),\eta(s))$ be a solution to \eqref{eq:mod eqn} and $(U,V)$ be defined by \eqref{def:UV}. Then $(U,V)$ solves
	\begin{equation}
		sU_s = \calM_{U} U + O \left( |U|^2+|V|^2 \right),\quad sV_s = \calM_{V} V + O \left( |U|^2+|V|^2 \right) \label{eq:linearization}
	\end{equation}
	where the $L\times L$ matrices $\calM_{U}$ and $\calM_{V}$ have of the form:
	\begin{equation}\label{def:matrix A}
		\calM_{U}=\begin{pmatrix} -\frac{2L+1}{4L-1} & 1 &  & & &   \\ -\frac{7}{2}c_2 & \frac{L-2}{4L-1} & 1 & & (0) &  \\ -\frac{11}{2}c_3 & & \frac{L-3}{4L-1} & 1 & &  \\ \vdots &  & & \ddots & \ddots &   \\ -(2L-\frac{5}{2})c_{L-1} & & (0) & & \frac{1}{4L-1} & 1   \\  -(2L-\frac{1}{2})c_{L} & & & & & 0     \end{pmatrix}
	\end{equation}
	and 
	\begin{equation}
		\calM_{V}=\begin{pmatrix} \frac{2L-1}{4L-1} & 1 &  & & &   \\ \frac{1}{2}c_2 & \frac{L-2}{4L-1} & 1 & & (0) &  \\ \frac{1}{2}c_3 & & \frac{L-3}{4L-1} & 1 & &  \\ \vdots &  & & \ddots & \ddots &   \\ \frac{1}{2}c_{L-1} & & (0) & & \frac{1}{4L-1} & 1   \\  \frac{1}{2}c_{L} & & & & & 0     \end{pmatrix}. \label{def:matrix M}
	\end{equation}
	Moreover, $\calM_{U}$ and $\calM_{V}$ are diagonalizable: $\calM_{U} = P^{-1} D_U P$ and $\calM_{V} = Q^{-1} D_V Q$, where $P$ and $Q$ are invertible matrices, and $D_U$ and $D_V$ are diagonal matrices with
    \begin{equation}\label{eq:diagonalization}
        D_U = \mathrm{diag}\big\{-1, \tfrac{2}{4L-1},\dots,\tfrac{L}{4L-1}  \big\},\quad D_V =  \mathrm{diag}\big\{\tfrac{1}{4L-1},\tfrac{2}{4L-1},\dots,\tfrac{L}{4L-1}  \big\}.
    \end{equation}
\end{lem}
\begin{proof}
	We observe the relation 
	\[(2k-\tfrac{1}{2})c_1-k =\tfrac{(4k-1)L}{4L-1} - k  = -\tfrac{L-k}{4L-1}.\]
	Then we obtain \eqref{eq:linearization}, \eqref{def:matrix A} and \eqref{def:matrix M} since
	\begin{align*}
		&(b_k)_s-b_{k+1}+(2k-\tfrac{1}{2})b_1b_{k}+\tfrac{1}{2}\eta_1\eta_k \\
		&= s^{-k-1} \left[ s(U_k)_s-kU_k- U_{k+1} + (2k-\tfrac{1}{2})(c_k U_1+ c_1 U_k) +O(|U|^2+|V|^2) \right]\\
		&=s^{-k-1} \left[ s(U_k)_s + (2k-\tfrac{1}{2})c_k U_1-\tfrac{L-k}{4L-1}U_k- U_{k+1} +O(|U|^2+|V|^2) \right]
	\end{align*}
	and 
	\begin{align*}
		&(\eta_k)_s-\eta_{k+1}+(2k-\tfrac{1}{2})b_1\eta_{k}-\tfrac{1}{2}\eta_1b_k \\
		&= s^{-k-1}\left[ s(V_k)_s-kV_k- V_{k+1} + (2k-\tfrac{1}{2}) c_1 V_k - \tfrac{1}{2}c_kV_1 +O(|U|^2+|V|^2) \right]\\
		&=s^{-k-1}\left[ s(V_k)_s -\tfrac{1}{2}c_k V_1-\tfrac{L-k}{4L-1}V_k- V_{k+1} +O(|U|^2+|V|^2) \right].
	\end{align*}
	$D_U$ in \eqref{eq:diagonalization} is obtained by substituting $\alpha=\frac{1}{2}$ for the result of Lemma 3.7 in \cite{MerleRaphaelRodnianski2015CambJMath}. $D_V$ in \eqref{eq:diagonalization} can also be derived by borrowing \textbf{step 2} from the proof of Lemma 3.7 in \cite{MerleRaphaelRodnianski2015CambJMath}.
\end{proof}
To make use of the simplicity provided by diagonalization, we will use the transformed fluctuations $(\calU,\calV)\coloneqq (PU, QV)$ in the next section.
\begin{rem}
    The linearized dynamics for the special solution $(b^e,\eta^e)$ \eqref{eq:linearization} represents $2L-1$ positive eigenvalues. This implies that this quantized blow-up solution has $2L-1$ unstable directions, which requires restrictions on our initial data.
\end{rem}
\begin{rem}[Formal instability mechanism]\label{rem:rotational instab}
    The system \eqref{eq:first mod equ} and \eqref{eq:mod eqn} can be solved explicitly as follows: the scaling and phase rotation parameters $(\lmb(t),\gmm(t))$ are given by 
    \begin{align*}
        \lambda(t)^{1/2}e^{-i\gamma(t)}&=\lambda_0^{1/2}e^{-i\gamma_0}\left(1-\frac{1}{2}\sum_{k=1}^L \frac{b_{k}(0)}{k!} \cdot\left(\frac{t}{\lambda_0^2}\right)^k -  \frac{i}{2}\sum_{k=1}^L \frac{\eta_{k}(0)}{k!} \cdot\left(\frac{t}{\lambda_0^2}\right)^k\right).
    \end{align*}
    In particular, the dynamics in the vicinity of our blow-up scenario can be represented by 
    \begin{align*}
        \lambda(t)^{1/2}e^{-i\gamma(t)}=z_L\left(T-t\right)^L+ z_0+\sum_{k=1}^{L-1}z_{k}\left(T-t\right)^{k} 
    \end{align*}
    for some complex-valued initial modulation parameters $(z_0,\dots,z_{L-1})\in \bbC^L$ and $z_L\in \bbC\setminus \{0\}$. We can fix $z_L=1$ by the scaling and the phase rotation invariance. Then, there are $2L$ degrees of freedom of $(z_0,\dots,z_{L-1})$; however, one degree of freedom for $z_0$ can be eliminated by shifting the blow-up time $T$.
    Thus, the degrees of freedom of these parameters intuitively illustrate the $2L-1$ codimension stability discussed above.

    In fact, for any $T>0$ and any collection of real-valued parameters $(c_1,\dots,c_{L-1})\times (d_0,\dots,d_{L-1}) \in \mathbb{R}^{2L-1}$, there exist $\lambda_0,\gamma_0,b_k(0),\eta_k(0)$ such that
    \begin{align}
        \lambda(t)^{1/2}e^{-i\gamma(t)}=\left(T-t\right)^L+id_0+ \sum_{k=1}^{L-1}(c_k+id_k)\left(T-t\right)^{k}. \label{eq:lmb-gmm-formula-1}
    \end{align}
    This shows that the freedom associated with the real part of $z_0$ is transferred to the choice of time $T$. Moreover, the formula \eqref{eq:lmb-gmm-formula-1} implies that the global solutions are generic when $d_0\neq 0$. Now, we describe the instability mechanism. We denote $c_0=0$. For fixed $0\leq j\leq L-1$, we assume $(c_j,d_j)\neq0$ and $(c_i,d_i)=0$ for $i\neq j$.
    \begin{equation}
        \begin{aligned}
        \lmb(t) & =(T-t)^{2j}[\{(T-t)^{(L-j)}+c_j\}^2+d_j^2],\\
        \gmm(t) & =\begin{cases}
        (1-\sgn\{(T-t)^j\})\frac{1}{2}\pi & \text{if }d_j=0,\\
        {\displaystyle \sgn\{d_j(T-t)^j\}\Big\{\frac{\pi}{2}-\arctan\Big(\frac{(T-t)^{L-j}+c_j}{|d_j|}\Big)\Big\}} & \text{if }d_j\neq0.
        \end{cases}
        \end{aligned}
        \label{eq:lmb-gmm-formula-2}
    \end{equation}
    where $\sgn(x)$ denotes the sign of $x$. If $j=0$, the formula \eqref{eq:lmb-gmm-formula-2} says that the solution is global. In addition, for small $d_0$, the phase rotation parameter $\gamma(t)$ rapidly changes in a short time scale $|T-t|\lesssim |d_0|$. This is the rotational instability as illustrated in the introduction. An important point is that the mechanism of rotational instability differs depending on whether $L$ is odd or even. When $L$ is odd, the mechanism is similar to that observed in \cite{KimKimKwon2024arxiv, KimKwon2019, Kim2022CSSrigidityArxiv}. However, for even $L$, while the abrupt change over a short time scale remains the same, the behavior differs in the final state. It first rotates by $\pi/2$, and then, rather than continuing to rotate another $\pi/2$, it reverses back to its original state.
    
    We also demonstrate other instability mechanisms for $1\leq j\leq L-1$. If we set $d_j=0$ in \eqref{eq:lmb-gmm-formula-2}, then the instability mechanism driven by $c_j$ results in a quantized blow-up with a rate of $(T-t)^{2j}$. This mechanism also appears in other quantized blow-up solutions such as \cite{RaphaelSchweyer2014AnalPDEHeatQuantized, Jeong2024arxiv}. On the other hand, if $c_j=0$, the mechanism driven by $d_j$ is more complicated. For $d_j\neq 0$, the corresponding solutions exhibit the quantized blow-up rate of $(T-t)^{2j}$. Simultaneously, their phase parameters rapidly change by $\pi/2$ near the blow-up time $T$, indicating that the instability associated with $d_j$ is another rotational instability. Therefore, there are exactly $L$ distinct types of rotational instabilities, including one induced by $d_0$.

\end{rem}

\section{Trapped regime}\label{sec:trapped regime}
To observe the quantized blow-up dynamics, we need to construct solutions whose modulation parameters align with the special solution \eqref{eq:special sol}. This requires a bootstrapping argument on the transformed fluctuations introduced earlier, as well as information on higher-order energies to close the bootstrap. We obtain estimates for these higher-order energies by using the hierarchy of conservation laws. 
This approach removes the need for repulsivity or monotonicity information. It enables us to employ the coercivity of $A_Q$ instead of that of the higher-order linearized operator, thereby simplifying the argument. We proceed with our analysis using the nonlinear variable $w_{k}$ as defined in \eqref{eq:wk def}. Since various modulation parameters, such as $(\wt b_1,\wt \eta_1)$ and $(b_k,\eta_k)$, appear in this section, we provide an overview in Remark \ref{rem:overview parameters} for clarity.

\subsection{Initial data}
In this subsection, we will specifically describe the initial data settings for the bootstrap process. Since we use the nonlinear variables $w_k$, a nonlinear version of the initial data condition is required. Thus, the main goal of this subsection is to formulate this condition. 

We denote the soliton tube $\mathcal{O}_{dec}$ for a small
parameter $\delta_{dec}>0$ to be determined later:
\begin{align*}
	\mathcal{O}_{dec}\coloneqq\{[Q+\eps_0^\prime]_{\mathring\lambda_{0},\mathring\gamma_{0}}(x)\in L^{2}_{e}:(\mathring\lambda_{0},\mathring\gamma_{0})\in\R_{+}\times\R/2\pi\Z,\,\|\eps_0^\prime\|_{L^{2}}<\delta_{dec}\}.
\end{align*}
By a Gram--Schmidt process, we can find $\calZ_k\in C^\infty_c$, $k=1,2,3,4$ such that
\begin{align}
    (\calK_j,\calZ_k)_r=\delta_{jk},\quad
    (iy^{2m+2}Q,\calZ_k)_r=(y^{2m}(1+y^2)Q,\calZ_k)_r=0 \label{eq:transversal}
\end{align}
for $1\leq m\leq L-1$.
We review the decomposition lemma, which was proved in \cite{KimKimKwon2024arxiv}. We denote by $\wt\calZ^\perp$ and $\calZ^\perp$
\[ \wt\calZ^\perp\coloneqq L^2_e \cap \{ \calZ_1,\calZ_2\}^{\perp}, \quad \calZ^\perp\coloneqq \langle y\rangle^2L^2_e \cap \{ \calZ_1,\calZ_2,\calZ_3,\calZ_4\}^{\perp}.\]

Now, we describe the initial data set. To do this, we will use the transformed fluctuation for initial data, $(\mathring \calU, \mathring \calV)=(\mathring \calU_1,\cdots,\mathring \calU_L, \mathring \calV_1,\cdots \mathring \calV_L)$, defined by
\begin{align}
\begin{gathered}
    (\mathring \calU, \mathring \calV)=(P\mathring U, Q\mathring V),
    \\
    \mathring U_k(s_0) \coloneqq s_0^k (\mathring b_k(s_0)-b_k^e(s_0)),\quad \mathring V_k(s_0) \coloneqq s_0^k \mathring \eta_k(s_0),\quad 1\le k \le L,
\end{gathered} \label{eq:calUV def}
\end{align}
where $b_k^e$ and $\eta_k^e$ are given by \eqref{eq:special sol}, and $\mathring b_k$ and $\mathring \eta_k$ are the initial modulation parameters appearing in \eqref{eq:init form} below. Thus, $(\mathring \calU,\mathring \calV)$ measures the deviation of the initial modulation parameters from the special solution.

We fix a small number $0<\kappa<\frac{c_1}{16L}=\frac{1}{8(4L-1)}$. We remark the definition of profiles $T_{2k-1}$ in \eqref{eq:T_k}. We define the initial data set $\calO_{init}$ by a collection of functions of the following forms
\begin{align}
    v_0=
    \left[
    Q+\sum_{j=1}^LT_{2j-1}(\mathring b_j, \mathring \eta_j)\chi_{\mathring\delta\mathring\la_{0}^{-1}}
    +\mathring{\eps}_{0}
    \right]_{\mathring\lambda_{0},\mathring\gamma_{0}}(x). \label{eq:init form}
\end{align}
which satisfies the following;
\begin{enumerate}
	\item $\mathring\lambda$ and $\mathring b_1$ relation:
	\begin{align}
		\frac{1}{1.1}\leq \frac{\mathring b_1^L}{\mathring\lambda^{2L-\frac{1}{2}}}(s_0) \leq 1.1 \label{eq:init lambda b relation aux}
	\end{align}

	\item Smallness of the stable mode:
	\begin{align}
		|\mathring\calU_1(s_0)|\leq s_0^{-\kappa}. \label{eq:init stable aux}
	\end{align} 

    \item Smallness of the unstable modes:
	\begin{align}
		s_0^\kappa( \mathring \calU_2(s_0),\cdots,\mathring\calU_L(s_0), \mathring\calV_1(s_0),\cdots,\mathring\calV_L(s_0))\in \calB^{2L-1} \label{eq:init unstable aux}
	\end{align}

    \item Smallness of the initial perturbation: $\mathring\eps_{0}\in\calZ^\perp\cap H^{2L}_e$ and
    \begin{align}
		\|\mathring\eps_{0}\|_{H^{2L}_e}\leq \mathring\lambda_0^{10L}. \label{eq:init radiation aux}
	\end{align} 
\end{enumerate}
We only use the $H^{2L}$-smallness in \eqref{eq:init radiation aux}. Thus, one may choose $\mathring\eps_0 \in \calZ^\perp \cap H^\infty_e$.

From \eqref{eq:b lambda relation sec3}, for $1\leq j\leq L$, we have
\begin{align}
	\mathring\lambda(s_0)\sim \mathring b_1(s_0)^{\frac{L}{2L-\frac{1}{2}}},\quad |\mathring b_j(s_0)|\sim \mathring b_1(s_0)^j,\quad |\mathring \eta_j(s_0)|\lesssim \mathring b_1(s_0)^{j+\kappa}. \label{eq:init b hat lambda hat relation}
\end{align}
Recalling $v_0=v(0)$ and $w(s_0)$ and $w_k(s_0)$ defined by \eqref{eq:renormal} and \eqref{eq:wk def}, respectively, we have
\begin{equation}
    \begin{aligned}
        \|w(s_0)-Q\|_{L^2}&\lesssim
        \mathring\delta^{\frac{3}{2}}\mathring\lambda^{\frac{1}{2L}}+\mathring\delta^{2L-\frac{1}{2}}+\|\mathring \eps_0\|_{L^2}\lesssim \mathring\delta^{\frac{3}{2}}\mathring\lambda^{\frac{1}{2L}}+\mathring\delta^{2L-\frac{1}{2}}+ \mathring\lambda_0^{10L}.
    \end{aligned} \label{eq:init energy aux}
\end{equation}
Here $\mathring \lambda_0=\mathring \lambda(s_0)$, and $\mathring\delta>0$ is a cut-off parameter, to be chosen sufficiently small later according to the decomposition lemma below.
We remark that control of $\|v_{k}(s_0)\|_{L^2}$ is not required. It suffices to know $\|v_{k}(s_0)\|_{L^2}<\infty$, which follows directly from $\mathring\eps_{0}\in H^{2L}$. Now, we introduce a lemma to obtain a rough decomposition for $w$.
\begin{lem}[Decomposition, \cite{KimKimKwon2024arxiv}]
	\label{lem:decomposition} There exist 
    small parameters $\delta',\delta'',\delta_{dec}>0$
	such that the following hold. 
	\begin{enumerate}
		\item[(1)] (The decomposition on $L^{2}_{e}$) There exists a map $(\mathbf{G},\wt{\eps}):\mathcal{O}_{dec}\to\mathbb{R}_{+}\times\mathbb{R}/2\pi\mathbb{Z}\times\wt{\mathcal{Z}}^{\perp}$
		satisfying 
		\begin{align*}
			v=[Q+\widehat{\eps}]_{\lambda,\gamma},\quad(\widehat{\eps},\mathcal{Z}_{k})_{r}=0\quad\text{for}\quad k=1,2
		\end{align*}
		with $\textnormal{\textbf{G}}=(\lambda,\gamma)$, and $\|\widehat{\eps}\|_{L^{2}_{e}}<\delta^{\prime}$. 
		\item[(2)] (The decomposition on $\langle y\rangle^{2}L^{2}_{e}$) Let $(\mathbf{G},\wt{\eps})$
		be the map obtained in (1). Then, there is a further decomposition
		map $(\mathbf{G},\mathbf{H},\eps):\mathcal{O}_{dec}\to\mathbb{R}_{+}\times\mathbb{R}/2\pi\mathbb{Z}\times\R^{2}\times\mathcal{Z}^{\perp}$
		satisfying 
		\begin{gather*}
			v=[Q+T_{1}( \wt b_1,  \wt \eta_1)+ \eps ]_{\lambda,\gamma},\quad T_1( \wt b_1, \wt\eta_1)+ \eps =\wt{\eps},
            \\
            ( \eps ,\mathcal{Z}_{k})_{r}=0 \quad\text{for}\quad k=1,2,3,4
		\end{gather*}
		with $\mathbf{G}=(\lambda,\gamma)$, $\mathbf{H}=( \wt  b_1, \wt \eta_1)$, and $\|\langle y\rangle^{-2} \eps \|_{L^{2}}<\delta^{\prime\prime}$.
		\item[(3)] ($C^{1}$-regularity) The map $v\mapsto(\lambda,\gamma,  \wt b_1,  \wt \eta_1)$
		for each decomposition is $C^{1}$. 
	\end{enumerate}
\end{lem}
A proof of an analogous lemma can be found in the Appendix of \cite{KimKimKwon2024arxiv}. Notably, if $v\in H^{2L}_e\cap \calO_{dec}$, then $\wt\eps(v)\in H^{2L}_{e}$ and $\eps (v) \in \dot\calH^{2}_{e}\cap \dot\calH^{2L}_{e}$.

In view of \eqref{eq:init energy aux}, taking $\mathring \delta$ and $\mathring \lambda(s_0)$ sufficiently small, we have $\calO_{init}\subset \calO_{dec}$. Therefore, we can apply Lemma \ref{lem:decomposition} to $\calO_{init}$. Moreover, since $\mathring\eps_0\in \calZ^\perp\cap H^{2L}_{e}$ and $\calZ_k$ satisfy \eqref{eq:transversal} and are compactly supported, we obtain that
\begin{align*}
    \sum_{j=1}^LT_{2j-1}(\mathring b_j, \mathring \eta_j)\chi_{\mathring\delta\mathring\la_{0}^{-1}}+\mathring{\eps}_{0}\in \wt\calZ^\perp,
    \quad \sum_{j=2}^LT_{2j-1}(\mathring b_j, \mathring \eta_j)\chi_{\mathring\delta\mathring\la_{0}^{-1}}+\mathring{\eps}_{0}\in \calZ^\perp
\end{align*}
by taking $\mathring\delta\mathring\lambda_0^{-1}\gg 1$. This means that $\mathring\lambda(s_0),\mathring\gamma(s_0),\mathring b_1(s_0), \mathring \eta_1(s_0)$ do not change even if Lemma \ref{lem:decomposition} is applied. In other words, we have
\begin{align}
    \begin{split}
      (\mathring\lambda(s_0),\mathring\gamma(s_0),\mathring b_1(s_0), \mathring \eta_1(s_0))=
      (\lambda(s_0),\gamma(s_0),\wt  b_1(s_0), \wt \eta_1(s_0)),
      \\
      \wt\eps(s_0)=\sum_{j=1}^LT_{2j-1}(\mathring b_j, \mathring  \eta_j)\chi_{\mathring\delta\mathring\la_{0}^{-1}}+\mathring{\eps}_{0},\quad
       \eps (s_0)=\wt\eps(s_0)-T_1(\mathring b_1, \mathring  \eta_1),
    \end{split}\label{eq:init decompose}
\end{align}
where $(\lambda(s_0),\gamma(s_0), \wt b_1(s_0), \wt \eta_1(s_0),\wt\eps(s_0), \eps (s_0))$ is defined by Lemma \ref{lem:decomposition}. 
\begin{rem}[Parameter dependence]
    For the reader's convenience, we clarify the dependencies of the parameters. We assume the following size conditions for the parameters:
    \begin{align*}
        0<s_0^{-\kappa}\ll1, \quad 0< \mathring\delta^{\frac{3}{2}}s_0^{-\frac{1}{4L-1}}\ll\delta_{dec},
        \quad \mathring\delta^{2L-\frac{1}{2}}\ll \delta_{dec}, \quad \max_{1\leq k \leq L}C_{k,\mathring \delta}\lambda(s_0)^{\frac{1}{8L}}\ll 1,
    \end{align*}
    where $C_{k,\mathring \delta}$ appears in \eqref{eq:proxi pf const}.
    The constant $0<\kappa<\frac{1}{8(4L-1)}$ was already fixed. This choice of $\kappa$ ensures the proximity for $\eta$ by \eqref{eq:mod para sim}. In view of Section \ref{sec:dynamical laws} with \eqref{eq:init unstable aux} and \eqref{eq:init stable aux}, the first condition ensures that $\mathring b_1(s_0)\sim s_0^{-1}\sim \lambda(s_0)^{\frac{4L-1}{2L}}$. Moreover, the second and third conditions imply $\mathcal{O}_{init}\subset \calO_{dec}$ by \eqref{eq:init energy aux}. The last condition is imposed to absorb the implicit constants in the proof of Proposition~\ref{prop:proximity init}. In addition, this also yields $\mathring\delta\lambda_0^{-1}\gg 1$.
\end{rem}

As in the previous works on (CSS) \cite{KimKwon2020blowup, KimKwonOh2020blowup, Kim2022CSSrigidityArxiv}, we define modulation parameters $b_k$ and $\eta_k$ through nonlinear variables $w_k$. Additionally, following the method in \cite{KimKimKwon2024arxiv}, we do not introduce a cut-off in the profile decomposition. We define $b_k$ and $\eta_k$ by
\begin{align}
    b_k\coloneqq \frac{(w_{2k-1},\frac{1}{2}(-i)^kyQ\chi)_r}{(\frac{1}{2}yQ,\frac{1}{2}yQ\chi)_r},
    \quad \eta_k\coloneqq \frac{(w_{2k-1},-\frac{1}{2}(-i)^{k-1}yQ\chi)_r}{(\frac{1}{2}yQ,\frac{1}{2}yQ\chi)_r}.\label{eq:bk def}
\end{align}
Then, we have decompositions for $w_{2k-1}$ such that
\begin{align}
    w_{2k-1}=P_{2k-1}(b_k,\eta_k)+\eps_{2k-1},\quad (\eps_{2k-1},yQ\chi)_r=(\eps_{2k-1},iyQ\chi)_r=0 \label{eq:w2k-1 decompose}
\end{align}
where $P_{2k-1}$ is given by \eqref{eq:P profile}. The choice of $P_{2k-1}$ is motivated from Lemma \ref{lem:profile choosing}, which shows that the profiles $T_{2k-1}$ \eqref{eq:T_k} for $w$ can be mapped to the profiles $P_{2k-1}$ \eqref{eq:P profile} for $w_{2k-1}$. Remarkably, $B_QP_{2k-1}$ is the same as the second equation of \eqref{eq:higher-order profile}. This implies that one may replace the new modulation parameters $b_k$ and $\eta_k$ by the initial parameters $\mathring b_k$ and $\mathring \eta_k$, and vice versa, up to a manageable error.

\begin{prop}[Proximity for initial modulation parameters]\label{prop:proximity init} Let $b_k$ and $\eta_k$ are given by \eqref{eq:bk def}. For $1\leq k\leq L$, we have
    \begin{align}
        |\mathring b_k(s_0)-b_k(s_0)|+|\mathring \eta_k(s_0)-\eta_k(s_0)|\lesssim |\mathring b_1(s_0)|^k \lambda(s_0)^{\frac{1}{8L}}. \label{eq:mod para sim}
    \end{align}
\end{prop}
\begin{rem}[Overview of modulation parameters]\label{rem:overview parameters}
    To avoid confusion from the various quantities appearing in the modulation analysis, we review them. All parameters are well-defined for all times $s\geq s_0$ unless otherwise specified:

    \textbullet $\mathring\lambda(s_0),\mathring\gamma(s_0),\mathring b_k(s_0),\mathring \eta_k(s_0)$: These parameters are introduced in the initial data set $\calO_{init}$, \eqref{eq:init form}. They are only defined at the initial time $s=s_0$.
    \begin{itemize}
        \item $\mathring\lambda(s_0),\mathring\gamma(s_0),\mathring b_k(s_0),\mathring \eta_k(s_0)$, $1\leq k\leq L$: These parameters are introduced in the initial data set $\calO_{init}$, \eqref{eq:init form}. They are only defined at the initial time $s=s_0$.

        \item $\lambda(s),\gamma(s),\wt b_1(s),\wt \eta_1(s)$: These are the decomposition parameters given by Lemma~\ref{lem:decomposition}. They coincide with the corresponding parameters with a ring at the initial time $s=s_0$ by \eqref{eq:init decompose}. That is,
        \begin{equation*}
            (\mathring\lambda(s_0),\mathring\gamma(s_0),\mathring b_1(s_0),\mathring \eta_1(s_0))=(\lambda(s_0),\gamma(s_0),\wt b_1(s_0),\wt \eta_1(s_0)).
        \end{equation*}
        In particular, for the scaling parameter, we write $\lambda(s_0)=\mathring\lambda(s_0)=\lambda_0=\mathring \lambda_0$.
        They are different from $(b_1,\eta_1)$ in general, although \eqref{eq:b1 similar} shows that the two pairs remain close on the bootstrap interval. Note that we do not use $\wt b_k$ and $\wt \eta_k$ for $k \geq 2$.

        \item $b_k(s), \eta_k(s)$, $1\leq k\leq L$: These are time-dependent modulation parameters defined through the higher-order nonlinear variables in \eqref{eq:bk def}. At the initial time $s=s_0$, they are close to the corresponding parameters with a ring in the sense of \eqref{eq:mod para sim}. These are the parameters we will mainly use in our analysis.

        \item $b_k^e(s), \eta_k^e(s)$: These denote the special solution of \eqref{eq:mod eqn} given by \eqref{eq:special sol}. They serve as the reference trajectory for the modulation parameters.

        \item $\td b_L(s), \td \eta_L(s)$: These parameters will be defined in \eqref{eq:def refine mod} to close the bootstrap; see Section \ref{sec:modul esti} for details.
    \end{itemize}
    We also remark that the radiation terms $\wt\eps(s)$ and $ \eps (s)$ are defined by Lemma \ref{lem:decomposition}, and $\eps_{2k-1}(s)$ is defined by \eqref{eq:w2k-1 decompose} using the nonlinear variables.
\end{rem}
Since the definition of $\calO_{init}$ and \eqref{eq:init decompose} do not directly provide information on the nonlinear variables $w_k$, we first state the following approximation lemma and postpone its proof to Section~\ref{sec:tail compute}, where it is obtained through a tail computation for the initial data.
\begin{lem}\label{lem:init approx}
	Let $v_0\in \calO_{init}$, $1\leq k\leq L$, and $\iota=1,i$. For the decomposition \eqref{eq:init decompose}, we have
	\begin{align}
		|(\td \bfD_Q^{2k-2}L_Q\wt\eps(s_0)-w_{2k-1}(s_0),\iota yQ\chi)_r| \lesssim_{k,\mathring \delta}  |\mathring b_1(s_0)|^k \lambda(s_0)^{\frac{1}{4L}}. \label{eq:init approx k any}
	\end{align}
\end{lem}

\begin{proof}[Proof of Proposition \ref{prop:proximity init} assuming Lemma~\ref{lem:init approx}]
    Since we do our analysis on the fixed initial time $s_0$, we omit the dependency of time.
    We remark on the decomposition \eqref{eq:init decompose} and $\td \bfD_Q^{2k-2}L_Q=B_Q^*\calL_{2k-1}$ on $H^{2L}$. One can check 
    \begin{align*}
        |\mathring b_{k}-b_{k}|+|\mathring \eta_{k}-\eta_{k}|\sim \sum_{\iota=1,i}|(P_{2k-1}(\mathring b_{k},\mathring \eta_{k})-P_{2k-1}( b_{k}, \eta_{k}),\iota yQ\chi)_r|.
    \end{align*}
    By \eqref{eq:higher-order profile}, we have
    \begin{equation}\label{eq:calL eps decom}
        \begin{aligned}
            B_Q(P_{2k-1}(\mathring b_{k},\mathring \eta_{k}))=& \calL_{2k-1}\wt\eps+\sum_{j=1}^{k}\calL_{2k-1}[T_{2j-1}(\mathring b_j, \mathring \eta_j)(1-\chi_{\mathring\delta\mathring\la^{-1}})]
            \\
            &-\sum_{j=k+1}^L\calL_{2k-1}[T_{2j-1}(\mathring b_j, \mathring \eta_j)\chi_{\mathring\delta\mathring\la^{-1}}]
            -\calL_{2k-1}\mathring{\eps}_{0}.
        \end{aligned}
    \end{equation}
    Moreover, by \eqref{eq:w2k-1 decompose}, we get
    \begin{equation*}
        B_Q(P_{2k-1}(b_{k}, \eta_{k}))=B_Qw_{2k-1}-B_Q\eps_{2k-1}.
    \end{equation*}
    We note that $\calH(yQ\chi)=\langle y\rangle^{-2}\calH(yQ\chi)+y\langle y\rangle^{-2}\calH(y^2Q\chi)$, and thus
    \begin{align}
        \|\langle y\rangle B_Q(yQ\chi)\|_{L^2}\lesssim 1. \label{eq:mod para sim pf1}
    \end{align}
    Taking the inner product of \eqref{eq:calL eps decom} with $B_Q(\iota yQ\chi)$, and then using $B_Q^*B_Q(yQ\chi)=yQ\chi$, \eqref{eq:init approx k any}, \eqref{eq:mod para sim pf1}, \eqref{eq:local profile estimate1}, \eqref{eq:local profile estimate2}, \eqref{eq:init radiation aux}, and \eqref{eq:w2k-1 decompose}, we derive
    \begin{align}\label{eq:proxi pf const}
        |\mathring b_{k}-b_{k}|+|\mathring \eta_{k}-\eta_{k}|
        \lesssim C_{k,\mathring \delta}\bigg(|\mathring b_1|^k \lambda^{\frac{1}{4L}}+\sum_{j=1}^L(|\mathring b_j|+|\mathring \eta_j|)\lambda^{2k+\frac{1}{2}-2j}\bigg)
        \lesssim C_{k,\mathring \delta}|\mathring b_1|^k \lambda^{\frac{1}{4L}}.
    \end{align}
    Thanks to \eqref{eq:init b hat lambda hat relation}, and taking $\lambda$ so that $C_{k,\mathring \delta}\lambda^{\frac{1}{8L}}\ll 1$ for all $1\leq k\leq L$, we conclude \eqref{eq:mod para sim}.
\end{proof}

Now, we establish the conditions for the new modulation parameters $b_k$ and $\eta_k$. While they differ from $\mathring b_k$ and $\mathring \eta_k$ introduced in the initial data \eqref{eq:init form}, we can approximately equal these parameters via \eqref{eq:mod para sim}. This enables us to convert the conditions for $\mathring b_k$ and $\mathring \eta_k$ into corresponding conditions for $b_k$ and $\eta_k$. 

We define $(\calU,\calV)=(\calU(s),\calV(s))$ in the same way as $(\mathring\calU(s_0), \mathring\calV(s_0))$ in \eqref{eq:calUV def}, but without the ring and replacing $s_0$ with $s$. If the initial data $v_0$ satisfies \eqref{eq:init lambda b relation aux}, \eqref{eq:init stable aux}, \eqref{eq:init unstable aux}, and \eqref{eq:init energy aux}, then by \eqref{eq:mod para sim} and taking $\mathring\delta$ sufficiently small, the following hold true:
\begin{enumerate}
	\item $\lambda$ and $b_1$ relation:
	\begin{align}
		\frac{1}{1.2}\leq \frac{b_1^L}{\lambda^{2L-\frac{1}{2}}}(s_0) \leq 1.2 \label{eq:init lambda b relation}
	\end{align}
	
	\item Smallness of the stable mode:
	\begin{align}
		|\calU_1(s_0)|\leq 2s_0^{-\kappa}. \label{eq:init stable} 
	\end{align}

    \item Smallness of the unstable modes:
	\begin{align}
		s_0^\kappa( \calU_2(s_0),\cdots,\calU_L(s_0), \calV_1(s_0),\cdots,\calV_L(s_0))\in 1.01\calB^{2L-1}. \label{eq:init unstable} 
	\end{align}
    
	\item Energy control:
	\begin{align}
		\|\wt\eps(s_0)\|_{L^2}\leq \tfrac{1}{10}\delta_{dec},\quad \|w_k(s_0)\|_{L^2}\leq C(k,v_0)\lambda^k(s_0) \text{ for } 1\leq k\leq 2L. \label{eq:init energy}
	\end{align}
\end{enumerate}
We note that the map $(\mathring \calU(s_0),\mathring\calV(s_0)) \mapsto (\calU(s),\calV(s))$ is continuous by the local well-posedness.
From the hierarchy of the conservation laws \eqref{eq:hierarchy}, the quantity, $I_{2k}(v)=(-1)^k\|\td \bfD_v^k v\|_{L^2}^2$ is conserved for $1\leq k\leq 2L$. Here, we used $\td \bfD_v^*=-\td \bfD_v$. By the definition of $w_k$, we have
\begin{align*}
    \|w_k(s_0)\|_{L^2}^2\lambda(s_0)^{-2k}=\|\td \bfD_{v_0}^k v_0\|_{L^2}^2=\|\td \bfD_v^k v\|_{L^2}^2=\|w_k(s)\|_{L^2}^2\lambda(s)^{-2k}.
\end{align*}
This and \eqref{eq:init energy} imply that
\begin{align}
	    \|w_k(s)\|_{L^2}\leq C(k,v_0)\lambda^k(s) \quad\textnormal{for any } s.\label{eq:boots energy}
\end{align}

Now, we introduce our main bootstrap proposition, which ensures that the solution remains trapped within the region defined by the bootstrap assumptions as time progresses, leading to the desired blow-up dynamics.

\begin{prop}[Trapped regime]\label{prop:bootstrap}
	For $0<\kappa<\frac{1}{8(4L-1)}$, let $v(s_0)$ be given of the form \eqref{eq:init form} satisfying \eqref{eq:init lambda b relation aux}, \eqref{eq:init stable aux}, and \eqref{eq:init energy aux}. Then, there exists an initial direction of the unstable modes satisfying \eqref{eq:init unstable aux} such that the solution $v(s)$ to \eqref{CMdnls-gauged} satisfies the following bounds for all $s\geq s_0$:
	\begin{align}
        \begin{gathered}
            \frac{1}{2}\leq \frac{b_1^L}{\lambda^{2L-\frac{1}{2}}}(s) \leq 2 ,
            \quad |\calU_1(s)|\leq 10s^{-\kappa}, 
            \\
            s^\kappa( \calU_2(s),\cdots,\calU_L(s), \calV_1(s),\cdots,\calV_L(s))\in 1.01\calB^{2L-1}.   
        \end{gathered}
		\label{eq:bootstrap assump}
	\end{align}
\end{prop}

For the initial data $v_0=v(s_0)$ satisfying \eqref{eq:init lambda b relation}--\eqref{eq:init energy}, we define a bootstrap exit time
\begin{align}
    s^*\coloneqq\sup\{s\geq s_0: \eqref{eq:bootstrap assump} \textnormal{ holds true on } [s_0,s] \}. \label{eq:bootstrap exit time}
\end{align}
We note that $s^*$ is well-defined, because \eqref{eq:bootstrap assump} holds true for $s=s_0$ by \eqref{eq:init lambda b relation}--\eqref{eq:init unstable}. 

To prove Proposition \ref{prop:bootstrap} using the contradiction argument, we assume $s^*<\infty$ for all initial directions of the unstable modes satisfying \eqref{eq:init unstable aux}. In view of \eqref{eq:b lambda relation sec3}, we have 
\begin{align}
    \lambda\sim b_1^{\frac{L}{2L-\frac{1}{2}}},\quad |b_j|\sim b_1^j,\quad |\eta_j|\lesssim b_1^{j+\kappa}
    \label{eq:b lambda relation}
\end{align}
for $s\in [s_0, s^*]$. Here, we observe that $\lambda^2 \mathbin{/} b_1\sim \lambda^{\frac{1}{2L}}\ll 1$.

In the rest of the paper, we assume the bootstrap assumption \eqref{eq:bootstrap assump}.

\subsection{Higher-order energy and modulation estimates}\label{sec:modul esti}
In this subsection, we establish higher-order energy estimates and justify the formal modulation laws. To this end, we employ \emph{different decompositions according to topologies}. This method was introduced in \cite{KimKimKwon2024arxiv} to avoid issues arising from the slow decay of the soliton $Q$ and the Hilbert transform $\calH$. 

We briefly outline this approach. In view of Lemma \ref{lem:decomposition} and \eqref{eq:w2k-1 decompose}, we use the following decompositions:
\begin{align}\label{eq:decomposition summary}
    \begin{split}
        w & =Q+\widehat{\eps}\\
         & =Q+T_1+ \eps ,\\
        w_{2k-1} & =w_{2k-1}\text{ itself}\\
         & =P_{2k-1}+\eps_{2k-1}.
    \end{split}
    \begin{split}
        &(\text{on }H^{1}_{e}, \ H^1 \text{ level})\\
        &(\text{on }\dot{\mathcal{H}}^{2}_{e}, \ \dot H^2 \text{ level})\\
        &(\text{on }L^{2}_{o}, \ \dot H^{2k-1} \text{ level})\\
        &(\text{on }\dot{\mathcal{H}}^{1}_{o}, \ \dot H^{2k} \text{ level})
    \end{split}
\end{align}
Since the profiles $T_1$ and $P_{2k-1}$ are not defined in $L^2$ due to their slow decay (or growth), decompositions such as $w=Q+T_1+\eps$ and $w_{2k-1}=P_{2k-1}+\eps_{2k-1}$ are not admissible in $L^2$. Typically, this issue is addressed by introducing cut-offs in the profiles. However, this approach makes it difficult to handle the non-local structure. Instead of using cut-offs, we just perform decompositions only in the topologies where the profiles are well-defined.

Using these decompositions, we derive coercivity estimates for the radiation terms $\wt \eps$, $ \eps $, and $\eps_{2k-1}$. In particular, the estimate for $\eps_{2k-1}$ is obtained by applying \eqref{eq:boots energy}. We also point out that the above decompositions cover estimates for all $H^{j}$ level, where $1\leq j\leq 2L$. Finally, to control the evolution of the modulation parameters, we justify the dynamical laws observed in Section \ref{sec:dynamical laws}.

\begin{lem}\label{lem:coercivity} For any $s\in [s_0, s^*]$, the following estimates hold:
	\begin{enumerate}
        \item (Consequence of the hierarchy of conservation law) For $1\leq j \leq 2L-1$, we have 
        \begin{align}
            \|A_Qw_{j}\|_{L^2}\lesssim\lambda^{j+1}. \label{eq:boots energy AQ}
        \end{align}
        
		\item (Coercivity estimates) We have
		\begin{align}
			\|\wt\eps\|_{\dot\calH^{1}}\sim\|w_1\|_{L^2}\lesssim \lambda, \label{eq:ehat H1 coercivity}
		\end{align}
		and
		\begin{align}
			\| \eps \|_{\dot\calH^{2}}\lesssim \lambda^2. \label{eq:e0 H2 coercivity}
		\end{align}
		Moreover, for $1\leq k \leq L$, we have the coercivity estimates for nonlinear variables,
		\begin{align}
			\|\eps_{2k-1}\|_{\dot\calH^{1}}\lesssim \lambda^{2k}. \label{eq:e2k-1 coercivity}
		\end{align}

        \item (Interpolation estimate) We have
        \begin{align}
            \|Q \eps \|_{L^\infty}+\|\partial_y \eps \|_{L^\infty}
            \lesssim \lambda^{\frac{3}{2}}, \label{eq:e L inf}
        \end{align}
        and
        \begin{align}
            \|\eps_{2k-1}\|_{L^\infty}
            \lesssim \lambda^{2k-\frac{1}{2}}.  \label{eq:e2k-1 L inf}
        \end{align}

        \item (Proximity for $b_1$ and $\eta_1$) We have
        \begin{align}
            |\wt b_1-b_1|+|\wt \eta_1-\eta_1|\lesssim \lambda^2. \label{eq:b1 similar}
        \end{align}
	\end{enumerate}
\end{lem}
\begin{rem}\label{rem:key estimate}
    We obtain \eqref{eq:boots energy AQ} directly from \eqref{eq:boots energy}. In this step, the repulsivity \eqref{eq:BQBQstar equal I} and \eqref{eq:DQ BQ decomp} are used to replace $\td \bfD_Q$ by $A_Q$. Moreover, the estimate \eqref{eq:boots energy AQ} yields the coercivity estimate for $\varepsilon_{2k-1}$ \eqref{eq:e2k-1 coercivity}, which serves as an alternative to the coercivity for higher-order operators such as $\td\bfD_{Q}^{2k-2}L_Q$ or $\calL_{2k-1}$.
\end{rem}
Proposition \ref{prop:proximity init} implies that \eqref{eq:b1 similar} holds at the initial time, $s = s_0$. However, to close the bootstrap, we need to show that \eqref{eq:b1 similar} remains true for as long as the bootstrap assumption holds. This requirement arises because the decompositions $w=Q+T_1+\eps$ and $w_1=P_1+\eps_1$ are both at the $\dot H^2$ level.
The proximity \eqref{eq:b1 similar} ensure that $\wt{b}_1$, as defined in Lemma \ref{lem:decomposition}, and $b_1$, as defined in \eqref{eq:bk def}, are nearly equal.
\begin{proof}[Proof of Lemma \ref{lem:coercivity}]
    We will prove the lemma in the following order: First, we prove \eqref{eq:ehat H1 coercivity}. Next, we establish \eqref{eq:boots energy AQ} and \eqref{eq:e2k-1 coercivity}. Then, we complete the proof of \eqref{eq:e0 H2 coercivity} and \eqref{eq:b1 similar}. Finally, we prove (3).

    \textbf{Step 1.} The proof of \eqref{eq:ehat H1 coercivity} is essentially the same as the proof of \cite[Lemma~5.14]{KimKimKwon2024arxiv}. Therefore, we provide only a brief outline. We start from
    \begin{align}
        w_1=L_Q\wt\eps+N_Q(\wt\eps),\quad N_Q(\wt\eps)=\wt\eps\mathcal{H}(\Re(Q\wt\eps))+\tfrac{1}{2}(Q+\wt\eps)\mathcal{H}(|\wt\eps|^2). \label{eq:coercive pf w1 equ}
    \end{align}
    Using \eqref{eq:CommuteHilbert} with $(1+y^2)\langle y\rangle^{-2}=1$, we derive that
    \begin{align*}
        \|w_1-L_Q\wt\eps\|_{L^2}\lesssim& \|\widehat{\eps}\|_{\dot{\mathcal{H}}^{1}}(\|\widehat{\eps}\|_{L^{2}}+\|\widehat{\eps}\|_{L^{2}}^{2})+\|Q\calH(|\wt\eps|^2)\|_{L^2}
        \\
        \lesssim & \|\widehat{\eps}\|_{\dot{\mathcal{H}}^{1}}(\|\widehat{\eps}\|_{L^{2}}+\|\widehat{\eps}\|_{L^{2}}^{2}),
    \end{align*}
    which implies
    \begin{align*}
        (1-\|\widehat{\eps}\|_{L^{2}})\|\widehat{\eps}\|_{\dot{\mathcal{H}}^{1}}
        \lesssim\|w_{1}\|_{L^{2}}
        \lesssim(1+\|\widehat{\eps}\|_{L^{2}})\|\widehat{\eps}\|_{\dot{\mathcal{H}}^{1}}.
    \end{align*}
    Since $\|\wt\eps\|_{L^2}$ is sufficiently small, \eqref{eq:coercivity L1} implies \eqref{eq:ehat H1 coercivity} as follows:
    \begin{align*}
        \|\wt\eps\|_{\dot \calH^1}\sim \|L_Q\wt\eps\|_{L^2}\sim \|w_1\|_{L^2}\lesssim \lambda.
    \end{align*}
    
    \textbf{Step 2.} Now, we show \eqref{eq:boots energy AQ} and \eqref{eq:e2k-1 coercivity}. \footnote{Although these arguments also appear in the current version of \cite{KimKimKwon2024arxiv}, they were first obtained in the present work.}
    
    We remark that $w_{2k-1}$ is odd since $w$ is even (radial). Thus, if \eqref{eq:boots energy AQ} is true, then we have
    \begin{align*}
        \|\eps_{2k-1}\|_{\dot\calH^{1}}\sim \|A_Q\eps_{2k-1}\|_{L^2}\lesssim \lambda^{2k}
    \end{align*}
    by the coercivity \eqref{eq:coercivity A1} with \eqref{eq:w2k-1 decompose}. Therefore, it suffices to show \eqref{eq:boots energy AQ}.
    Thanks to the repulsivity \eqref{eq:BQBQstar equal I} and \eqref{eq:DQ BQ decomp}, we have 
    \begin{align*}
    	\|A_Qw_j\|_{L^2}
        =\|B_Q\td \bfD_Qw_j\|_{L^2}
        &\leq \|B_Qw_{j+1}\|_{L^2}+\|B_Q(w_{j+1}-\td\bfD_Qw_j)\|_{L^2}
        \\
        &\lesssim \|w_{j+1}\|_{L^2}+\|B_Q(w_{j+1}-\td\bfD_Qw_j)\|_{L^2}.
    \end{align*}
    Note that we have $\|B_Qf\|_{L^2}\lesssim \|f\|_{L^2}$ for $f\in L^2$. By \eqref{eq:boots energy}, we have $\|w_{j+1}\|_{L^2}\lesssim \lambda^{j+1}$. Thus, the proof of \eqref{eq:boots energy AQ} is reduced to showing
    \begin{align}
        \|B_Q(w_{j+1}-\td\bfD_Qw_j)\|_{L^2}\lesssim \lambda^{j+1}. \label{eq:coerciv step 2 goal}
    \end{align}
    From $w_{j+1}=\td\bfD_ww_{j}$ and $w=Q+\wt\eps$, we decompose \eqref{eq:coerciv step 2 goal} into three parts as follows:
    \begin{align*}
    	2B_Q(w_{j+1}-\td\bfD_Qw_j)=B_Q[Q\calH(\ol{\wt\eps} w_j)+\wt\eps\calH(Qw_j)+\wt\eps\calH(\ol{\wt\eps} w_j)].
    \end{align*}
    For $\wt\eps\calH \ol{\wt\eps} w_j$, \eqref{eq:ehat H1 coercivity} and \eqref{eq:boots energy} imply
    \begin{align*}
    	\|B_Q[\wt\eps\calH (\ol{\wt\eps} w_j)]\|_{L^2}
        \lesssim \|\wt\eps\calH (\ol{\wt\eps} w_j)\|_{L^2}
        \lesssim \|\wt\eps\|_{L^\infty}^2\|w_j\|_{L^2}\lesssim \lambda^{j+1}.
    \end{align*}
    For $\wt\eps\calH Qw_j$, using \eqref{eq:CommuteHilbert} with $(1+y^2)\langle y\rangle^{-2}=1$, we compute
    \begin{align*}
    	\wt\eps\calH (Qw_j)=\langle y\rangle^{-2}\wt\eps\calH (Qw_j)+y\langle y\rangle^{-2}\wt\eps\calH(yQw_j)+y\langle y\rangle^{-2}\wt\eps\textstyle{\int_{\bbR}} Qw_j,
    \end{align*}
    and thus, we derive
    \begin{align*}
    	\|\wt\eps\calH (Qw_j)\|_{L^2}
    	\lesssim \|\wt\eps\|_{\dot\calH^1}\|w_j\|_{L^2}\lesssim \lambda^{j+1}.
    \end{align*}
    To control $Q\calH \ol{\wt\eps} w_j$, we use $B_QQ=0$. Again using \eqref{eq:CommuteHilbert}, we have
    \begin{align}\label{eq:bqy part}
    	B_Q[Q\calH(\ol{\wt\eps} w_j)]=B_Q[Q\calH(\langle y\rangle^{-2}\ol{\wt\eps} w_j)]+B_Q[yQ\calH(y\langle y\rangle^{-2}\ol{\wt\eps} w_j)].
    \end{align}
    For the second term of the RHS in \eqref{eq:bqy part}, we deduce
    \begin{align*}
    	\|B_Q[yQ\calH(y\langle y\rangle^{-2}\ol{\wt\eps} w_j)]\|_{L^2}
    	\lesssim  \|\wt\eps\|_{\dot\calH^1}\|w_j\|_{L^2}\lesssim \lambda^{j+1}.
    \end{align*}
    The first term of the RHS in \eqref{eq:bqy part} can be handled similarly, we conclude \eqref{eq:boots energy AQ}.

    \textbf{Step 3.} The proof of \eqref{eq:e0 H2 coercivity} closely follows that of \cite[Lemma~5.14]{KimKimKwon2024arxiv}. Therefore, we provide only an outline.
    
    For \eqref{eq:e0 H2 coercivity}, we use \eqref{eq:coercivity L2} and $\|A_Qw_1\|_{L^2}=\|A_Q\eps_1\|_{L^2}\lesssim \lambda^2$. Since $A_QL_Q=\calL_2$ on $H^{2}$, we easily check that 
    \begin{align}
        \|A_Q\eps_1-\calL_2  \eps \|_{L^2}=\|A_QN_Q(\wt\eps)\|_{L^2}, \label{eq:coercive pf 1}
    \end{align}
    where $\eps_1$ and $ \eps $ are defined in \eqref{eq:w2k-1 decompose} and Lemma \ref{lem:decomposition} respectively. We also remark $N_Q(\wt\eps)$ is given in \eqref{eq:coercive pf w1 equ}. Using again \eqref{eq:CommuteHilbert} with $(1+y^2)\langle y\rangle^{-2}=1$, $A_Q(Q)=A_Q(yQ)=0$, and \eqref{eq:CommuteHilbertDerivative}\footnote{We use these to control the term $A_Q(Q\calH(|\wt\eps|^2))$.}, we have
    \begin{align}
        \|A_QN_Q(\wt\eps)\|_{L^2}\lesssim \|\wt\eps\|_{\dot \calH^1}^2+\|\wt\eps\|_{L^\infty}\|\eps\|_{\dot \calH^2}
        \lesssim \lambda^2+\lambda^{\frac{1}{2}}\|\calL_2  \eps \|_{L^2}. \label{eq:coercive pf 2}
    \end{align}
    Combining \eqref{eq:coercive pf 1} and \eqref{eq:coercive pf 2} with $\lambda \ll 1$, we have
    \begin{align*}
        \|\calL_2  \eps \|_{L^2}\lesssim \lambda^2+\|A_Q\eps_1\|_{L^2}\lesssim \lambda^2.
    \end{align*}
    Now, we prove \eqref{eq:b1 similar}. Since $yQ\chi$ is odd, we have $B_Q^*B_Q(yQ\chi)=yQ\chi$. Thus, we observe that
    \begin{align}
        (L_Q\wt\eps-w_1,\iota yQ\chi)_r=(\calL_1\wt\eps-B_Qw_1,\iota B_Q(yQ\chi))_r,\quad \iota=1,i. \label{eq:coercive pf b differ 1}
    \end{align}
    We also have that $\calL_0T_1$ is well-defined, and we can check $\calL_1T_1=B_QP_1$ from \eqref{eq:higher-order profile}. Taking the decompositions $\wt\eps=T_1(\wt b_1, \wt\eta_1)+ \eps $ and $w_1=P_1(b_1,\eta_1)+\eps_1$, we have
    \begin{align*}
        \eqref{eq:coercive pf b differ 1}
        &=\tfrac{1}{2}(iyQ,\iota yQ\chi)_r(b_1-\wt b_1)
        +\tfrac{1}{2}(yQ,\iota yQ\chi)_r(\eta_1-\wt \eta_1)
        \\
        &+(\calL_1 \eps ,B_Q(yQ\chi))_r.
    \end{align*}
    Since $yQ\chi$ is odd, we have $yB_Q(yQ\chi)=B_Q(y^2Q\chi)$ by using \eqref{eq:CommuteHilbert}. From this with $(1+y^2)\langle y\rangle^{-2}=1$, \eqref{eq:e0 H2 coercivity}, and the definition of $\calL_1$, we have
    \begin{align*}
        |(\calL_1 \eps ,B_Q(yQ\chi))_r|\lesssim \| \eps \|_{\dot\calH^2}\lesssim \lambda^2.
    \end{align*}
    Therefore, it suffices to show $|(L_Q\wt\eps-w_1,\iota yQ\chi)_r|\lesssim \lambda^2$. From \eqref{eq:coercive pf w1 equ}, we reduce the proof of this to showing
    \begin{align}
        |(N_Q(\wt\eps),\iota yQ\chi)_r|\lesssim \lambda^2. \label{eq:coercive pf b differ goal}
    \end{align}
    Let us prove it for the case $\iota=1$, as the case $\iota=i$ follows in the same way and will be omitted. Except the term $Q\calH(|\wt\eps|^2)$ in $N_Q(\wt\eps)$, the proof of \eqref{eq:coercive pf b differ goal} is same as \eqref{eq:coercive pf 2}. To control $Q\calH(|\wt\eps|^2)$, we use \eqref{eq:CommuteHilbert} with $(1+y^2)\langle y\rangle^{-2}=1$. Since $|\wt\eps|^2$ is even, we have
    \begin{align*}
        \calH(|\wt\eps|^2)=(1+y^2)\calH(\langle y\rangle^{-2}|\wt\eps|^2)-\tfrac{y}{\pi}\textstyle{\int_{\bbR}}\langle y\rangle^{-2}|\wt\eps|^2dy.
    \end{align*}
    Thus, we have
    \begin{align*}
        |(Q\calH(|\wt\eps|^2), yQ\chi)_r|\lesssim
        |(\calH(\langle y\rangle^{-2},y\chi)_r|+(yQ,yQ\chi)_r\textstyle{\int_{\bbR}}\langle y\rangle^{-2}|\wt\eps|^2dy
        \lesssim \lambda^2,
    \end{align*}
    and this proves \eqref{eq:coercive pf b differ goal}. Therefore, we conclude \eqref{eq:b1 similar}.

    \textbf{Step 4.} Now, we show $(3)$. The proofs of \eqref{eq:e L inf} and \eqref{eq:e2k-1 L inf} can also be found in \cite{KimKimKwon2024arxiv}, Lemma 5.14. Therefore, we will only prove the first part of \eqref{eq:e L inf}, as the others can be proven in a similar manner. We decompose $Q\wt\eps$ as
    \begin{align}
        Q\wt\eps=QT_1\chi_{\lambda^{-1}}+Q\wt\eps_{\text{cut}},\quad 
        Q\wt\eps_{\text{cut}}=QT_1(1-\chi_{\lambda^{-1}})+Q \eps . \label{eq:coercive pf Linf}
    \end{align}
    From the first equation of \eqref{eq:coercive pf Linf}, \eqref{eq:b1 similar}, and \eqref{eq:b lambda relation}, we derive
    \begin{align*}
        \|Q\wt\eps_{\text{cut}}\|_{L^2}\leq \|Q\wt\eps\|_{L^2}+\|QT_1\chi_{\lambda^{-1}}\|_{L^2}
        \lesssim \lambda+\wt b_1\lambda^{-\frac{1}{2}}\lesssim \lambda.
    \end{align*}
    Similarly, the second equation of \eqref{eq:coercive pf Linf} implies that
    \begin{align*}
        \|Q\wt\eps_{\text{cut}}\|_{\dot H^1}\leq 
        \|QT_1(1-\chi_{\lambda^{-1}})\|_{\dot H^1}+\|Q \eps \|_{\dot H^1}
        \lesssim \wt b_1\lambda^{\frac{1}{2}}+\lambda^2\lesssim \lambda^2.
    \end{align*}
    Combining these and the Sobolev inequality $\|f\|_{L^\infty}^2\lesssim \|f\|_{L^2}\|f\|_{\dot H^1}$, we deduce $\|Q\wt\eps_{\text{cut}}\|_{L^\infty}\lesssim \lambda^{\frac{3}{2}}$. Therefore, using the second equation of \eqref{eq:coercive pf Linf}, we obtain the first part of \eqref{eq:e L inf} such as
    \begin{align*}
        \|Q \eps \|_{L^\infty}\leq \|Q\wt\eps_{\text{cut}}\|_{L^\infty}+\|QT_1(1-\chi_{\lambda^{-1}})\|_{L^\infty}
        \lesssim \lambda^{\frac{3}{2}}+\wt b_1 \lesssim \lambda^{\frac{3}{2}}.
    \end{align*}
    The others can be derived by a similar argument, so we omit their proofs here. For details, refer to \cite[Lemma~5.14]{KimKimKwon2024arxiv}. With this, the proof is complete.
\end{proof}
Thanks to \eqref{eq:b1 similar}, we will no longer distinguish between $(\wt b_1, \wt\eta_1)$ and $(b_1,\eta_1)$ moving forward.

Now, we prove the modulation estimates. By taking the time derivative to the orthogonality conditions in Lemma \ref{lem:decomposition} and \eqref{eq:w2k-1 decompose}, we derive these estimates. 
\begin{lem}[Modulation estimates]\label{lem:mod esti} We have
	\begin{align}
		\left|\frac{\lambda_s}{\lambda}+b_1\right|+\left|\gamma_s-\frac{\eta_1}{2}\right|\lesssim \lambda^2, \label{eq:mod equ 1}
	\end{align}
	and for $1\leq k\leq L-1$,
	\begin{align}
		\begin{split}
			&|(b_k)_s-b_{k+1}+(2k-\tfrac{1}{2})b_1b_{k}+\tfrac{1}{2}\eta_1\eta_k|\lesssim b_1^{k+1}\lambda^{\frac{1}{2L}}
			\\
			&|(\eta_k)_s-\eta_{k+1}+(2j-\tfrac{1}{2})b_1\eta_{k}-\tfrac{1}{2}\eta_1b_k|\lesssim  b_1^{k+1}\lambda^{\frac{1}{2L}}
		\end{split} \label{eq:mod equ 2}
	\end{align}
\end{lem}
\begin{proof}
    The proof of \eqref{eq:mod equ 1} is indeed exactly the same as the proof of \cite[Lemma~5.15]{KimKimKwon2024arxiv}. So we will just give a brief overview. From \eqref{eq:w equ} and the decompositions $w=Q+\wt\eps$ and $w_1=P_1+\eps_1$, we have
    \begin{align*}
        (\partial_{s}-\tfrac{\lambda_{s}}{\lambda}\Lambda+\gamma_{s}i)\widehat{\eps}  +(iL_{w}^{*}w_{1}-iL_{Q}^{*}P_{1})=\left(\tfrac{\lambda_{s}}{\lambda}+b_1\right)\Lambda Q-\left(\gamma_{s}-\tfrac{\eta_1}{2}\right)iQ.
    \end{align*}
    Using the orthogonality conditions $(\wt\eps,\calZ_1)_r=(\wt\eps,\calZ_2)_r=0$, \eqref{eq:transversal}, and Lemma~\ref{lem:coercivity}, we can derive \eqref{eq:mod equ 1}. See \cite[Lemma~5.15]{KimKimKwon2024arxiv} for details.
    
	To prove \eqref{eq:mod equ 2}, we start from
	\begin{align}
		(\partial_s-\tfrac{\lambda_s}{\lambda}\Lambda_{-(2k-1)} +\gamma_s i)w_{2k-1}+iH_ww_{2k-1}=0. \label{eq:w2j-1 equ modpf aux}
	\end{align}
	By using \eqref{eq:Dv square} to \eqref{eq:w2j-1 equ modpf aux}, we have
	\begin{align}
		(\partial_s-\tfrac{\lambda_s}{\lambda}\Lambda_{-(2k-1)} +\gamma_s i)w_{2k-1}-iw_{2k+1}=
		-\tfrac{i}{2}(w_1\mathcal{H}\overline{w}-w\mathcal{H}\overline{w_1})w_{2k-1}. \label{eq:w2j-1 equ modpf}
	\end{align}
	From the decompositions
	\begin{align}
		w_{2k-1}=P_{2k-1}+\eps_{2k-1},\quad w_{2k+1}=P_{2k+1}+\eps_{2k+1}, \label{eq:mod pf decomposition}
	\end{align}
	we rewrite \eqref{eq:w2j-1 equ modpf} into
	\begin{align}
		&-\tfrac{1}{2}(-i)^{k-1}yQ[(\partial_s-\tfrac{\lambda_s}{\lambda}(2k-\tfrac{1}{2})+\gamma_si)(ib_k+\eta_k)-(ib_{k+1}+\eta_{k+1})]\label{eq:jth modul}
		\\
		&=(\tfrac{\lambda_s}{\lambda}+b_1)y\partial_yP_{2k-1}-(\partial_s-\tfrac{\lambda_s}{\lambda}\Lambda_{-(2k-1)} +\gamma_s i)\eps_{2k-1}+i\eps_{2k+1}+\calN_{2k-1}, \label{eq:jth modul RHS}
	\end{align}
	where
	\begin{align}
		\calN_{2k-1}=&-\tfrac{i}{2}w_1\mathcal{H}(\overline{w}w_{2k-1})+\tfrac{i}{2}P_1\mathcal{H}(QP_{2k-1}) \label{eq:jth modul non 1}
        \\
        \begin{split}
            &+\tfrac{i}{2}w\mathcal{H}(\langle y\rangle^{-2}\overline{w_1}w_{2k-1}) -\tfrac{i}{2}Q\mathcal{H}(\langle y\rangle^{-2}\overline{P_1}P_{2k-1})
             \\
            &+\tfrac{i}{2}yw\mathcal{H}(y\langle y\rangle^{-2}\overline{w_1}w_{2k-1}) -\tfrac{i}{2}yQ\mathcal{H}(y\langle y\rangle^{-2}\overline{P_1}P_{2k-1}).
        \end{split} \label{eq:jth modul non 2}
	\end{align}
    Note that we used
    \begin{equation*}
        \calH(\overline{w_1}w_{2k-1})=\calH(\langle y\rangle^{-2}\overline{w_1}w_{2k-1})
        +y\calH(y\langle y\rangle^{-2}\overline{w_1}w_{2k-1}),
    \end{equation*}
    which can be obtained by applying \eqref{eq:CommuteHilbert} with $(1+y^2)\langle y\rangle^{-2}=1$ and the evenness of $\overline{w_1}w_{2k-1}$ (so, $\int y\langle y\rangle^{-2}\overline{w_1}w_{2k-1}=0$).
    In addition, the profiles parts of \eqref{eq:jth modul non 1} and \eqref{eq:jth modul non 2} can be calculated using \eqref{eq:algebraic 1} and \eqref{eq:algebraic 2}, and one obtains
    \begin{equation*}
        \tfrac{i}{2}P_1\mathcal{H}(QP_{2k-1})
        -\tfrac{i}{2}Q\mathcal{H}(\langle y\rangle^{-2}\overline{P_1}P_{2k-1})
        -\tfrac{i}{2}yQ\mathcal{H}(y\langle y\rangle^{-2}\overline{P_1}P_{2k-1})
        =-b_1 y\partial_y P_{2k-1}.
    \end{equation*}
	We have
	\begin{align*}
		(\eqref{eq:jth modul},(-i)^{k}yQ\chi)_r&=((b_k)_s-b_{k+1}-(2k-\tfrac{1}{2})\tfrac{\lambda_s}{\lambda}b_k+\gamma_s\eta_k)(\tfrac{1}{2}yQ,yQ\chi)_r
		\\
		&=(\eqref{eq:jth modul RHS},(-i)^{k}yQ\chi)_r.
	\end{align*}
    Now, we claim that
    \begin{align*}
        (\eqref{eq:jth modul RHS},(-i)^{k}yQ\chi)_r
		=O(\lambda^2b_k)=O(b_1^{k+1}\lambda^{\frac{1}{2L}}).
    \end{align*}
    For $(\frac{\lambda_s}{\lambda}+b_1)y\partial_yP_{2k-1}$, since $|y\partial_yP_{2k-1}|\lesssim b_k$ in the pointwise sense, we have 
    \begin{align*}
        |(\tfrac{\lambda_s}{\lambda}+b_1)y\partial_yP_{2k-1},(-i)^kyQ\chi)_r|
        \lesssim |\tfrac{\lambda_s}{\lambda}+b_1||b_k|\lesssim \lambda^2|b_k|.
    \end{align*}
    From \eqref{eq:w2k-1 decompose}, we have $(\partial_s\eps_{2k-1},(-i)^{k}yQ\chi)_r=0$. Using \eqref{eq:e2k-1 coercivity}, we derive
    \begin{align*}
        \begin{split}
            |((-\tfrac{\lambda_s}{\lambda}\Lambda_{-(2k-1)} +\gamma_s i)\eps_{2k-1},(-i)^kyQ\chi)_r|&\lesssim (|\tfrac{\lambda_s}{\lambda}|+|\gamma_s|)\|\eps_{2k-1}\|_{\dot\calH^1}
            \\
            &\lesssim b_1\lambda^{2k}\sim b_1^{k+1}\lambda^{\frac{k}{2L}}.
        \end{split}
    \end{align*}
    Similarly, we have $|(i\eps_{2k+1},(-i)^kyQ\chi)_r|\lesssim \lambda^{2k+2}\sim b_1^{k+1}\lambda^{\frac{k+1}{2L}}$.

    Finally, we estimate $|(\calN_{2k-1},(-1)^kyQ\chi)_r|$. We claim that
    \begin{align}
        |(\calN_{2k-1},(-1)^kyQ\chi)_r|\lesssim b_1^{k+1}\lambda^{\frac{k}{2L}}. \label{eq:jth modul non goal}
    \end{align}
    For \eqref{eq:jth modul non 1}, we have
    \begin{align*}
        \eqref{eq:jth modul non 1}=
        -\tfrac{i}{2}\eps_1\calH(\overline{w}w_{2k-1})
        -\tfrac{i}{2}P_1\calH(\overline{\wt\eps}w_{2k-1})
        -\tfrac{i}{2}P_1\calH(Q\eps_{2k-1}).
    \end{align*}
    Thus, using \eqref{eq:e2k-1 coercivity}, we obtain
    \begin{align*}
        |(\eqref{eq:jth modul non 1},(-1)^kyQ\chi)_r|
        \lesssim \lambda^2\|\calH(\overline{w}w_{2k-1})\|_{L^\infty}+b_1(\|\calH(\overline{\wt\eps}w_{2k-1})\|_{L^\infty}+\lambda^{2k}).
    \end{align*}
    Thanks to $\|\calH(f)\|_{L^\infty}^2\lesssim \|f\|_{L^2}\|f\|_{\dot H^1}$, the decomposition $w_{2k-1}=P_{2k-1}+\eps_{2k-1}$, \eqref{eq:ehat H1 coercivity}, and \eqref{eq:e2k-1 coercivity}, we have 
    \begin{align}
        \|\calH(\overline{\wt\eps}w_{2k-1})\|_{L^\infty}\lesssim |b_k|^{\frac{1}{2}}\lambda^{k+\frac{1}{2}}+\lambda^{2k}\lesssim \lambda^{2k}. \label{eq:jth modul non 1-1}
    \end{align}
    For $\|\calH(\overline{w}w_{2k-1})\|_{L^\infty}$, using $w=Q+\wt\eps$, we have
    \begin{align}
        \|\calH(\overline{w}w_{2k-1})\|_{L^\infty}\leq \|\calH(Qw_{2k-1})\|_{L^\infty}
        +\|\calH(\overline{\wt\eps}w_{2k-1})\|_{L^\infty}. \label{eq:jth modul non goal 1}
    \end{align}
    The second term in \eqref{eq:jth modul non goal 1} is exactly the same as \eqref{eq:jth modul non 1-1}. For the first term in \eqref{eq:jth modul non goal 1}, again using the fact $\|\calH(f)\|_{L^\infty}^2\lesssim \|f\|_{L^2}\|f\|_{\dot H^1}$, the decomposition \eqref{eq:mod pf decomposition}, and \eqref{eq:ehat H1 coercivity}, we have $\|\calH(Qw_{2k-1})\|_{L^\infty}\lesssim |b_k|$. These conclude that
    \begin{align*}
         |(\eqref{eq:jth modul non 1},(-1)^kyQ\chi)_r|\lesssim |b_k|\lambda^2+b_1\lambda^{2k}\lesssim b_1^{k+1}\lambda^{\frac{k}{2L}}.
    \end{align*}
    Now, we control \eqref{eq:jth modul non 2}. Since the first line in \eqref{eq:jth modul non 2} can be treated in a similar manner as the second, we will focus on estimating the second line only.
    We omit $\frac{i}{2}$ for a convenience. We rewrite the second line of \eqref{eq:jth modul non 2} as
    \begin{align}
        y\wt\eps\mathcal{H}(y\langle y\rangle^{-2}\overline{w_1}w_{2k-1})
        &+yQ\mathcal{H}(y\langle y\rangle^{-2}\overline{P_1}\eps_{2k-1})\label{eq:jth modul non 2-1}
        \\
        &+yQ\mathcal{H}(y\langle y\rangle^{-2}\overline{\eps_1}w_{2k-1}).\label{eq:jth modul non 2-2}
    \end{align}
    Using the decompositions $\wt\eps=T_1+ \eps $ and $w_{2k-1}=P_{2k-1}+\eps_{2k-1}$, we have
    \begin{align*}
        |(\eqref{eq:jth modul non 2-1},(-1)^kyQ\chi)_r|\lesssim (\wt b_1+b_1)\lambda^{2k}\lesssim  b_1^{k+1}\lambda^{\frac{k}{2L}}.
    \end{align*}
    For the last term \eqref{eq:jth modul non 2-2}, we use \eqref{eq:CommuteHilbert}. Then, we have
    \begin{align}
        \eqref{eq:jth modul non 2-2}=y^2Q\mathcal{H}(\langle y\rangle^{-2}\overline{\eps_1}w_{2k-1})-\tfrac{1}{\pi}yQ\textstyle{\int_{\bbR}}\langle y\rangle^{-2}\overline{\eps_1}w_{2k-1} dy. \label{eq:jth modul non 2-3}
    \end{align}
    From $w_{2k-1}=P_{2k-1}+\eps_{2k-1}$, we have
    \begin{align}
        |(\eqref{eq:jth modul non 2-3},(-1)^kyQ\chi)_r|\lesssim \lambda^2 |b_k|\sim b_1^{k+1}\lambda^{\frac{k}{2L}}. \label{eq:jth modul non goal 2}
    \end{align}
    Thanks to \eqref{eq:jth modul non goal 1} and \eqref{eq:jth modul non goal 2}, we conclude our claim \eqref{eq:jth modul non goal}. Therefore, we prove the estimate of $b$ parameter part of \eqref{eq:mod equ 2}. In a similar manner, we can also prove the $\eta$ parameter side, so we omit the proof for these. This concludes the proof.
\end{proof}
Simply using the orthogonality conditions as above does not provide sufficient smallness for the modulation estimates of $b_L$ and $\eta_L$. We therefore use the refined modulation argument from \cite{KimKimKwon2024arxiv}, inspired by \cite{JendrejLawrieRodriguez2022ASENS}, to gain the additional smallness needed to close the top-order estimates.
\begin{lem}[Refined modulation estimates]
	Define the refined modulation parameters as
	\begin{align}
		\td b_{L}\coloneqq \frac{(B_Qw_{2L-1},i^L\chi_R)_r}{AR},\quad 
		\td \eta_{L}\coloneqq \frac{(B_Qw_{2L-1},i^{L-1}\chi_R)_r}{AR}, \label{eq:def refine mod}
	\end{align}
	where $R=\lambda^{-(1-\frac{1}{4L})}$, and $A$ is given by
	\begin{align*}
		A\coloneqq R^{-1}(-\tfrac{\sqrt{2}}{2},\chi_R)_r=(-\tfrac{\sqrt{2}}{2},\chi)_r.
	\end{align*}
	Then, we have 
	\begin{align}
		|b_L-\td b_{L}|+|\eta_L-\td \eta_{L}|\lesssim |b_L|\lambda^{\frac{1}{8L}}, \label{eq:refined mod gap}
	\end{align}
	and
    \begin{equation}\label{eq:refined mod equ}
        |(\td b_{L})_s+(2L-\tfrac{1}{2})b_1\td b_{L}+\tfrac{1}{2}\eta_1\td \eta_{L}| + |(\td \eta_{L})_s+(2L-\tfrac{1}{2})b_1\td \eta_{L}-\tfrac{1}{2}\eta_1\td b_{L}|\lesssim b_1^{L+1}\lambda^{\frac{1}{8L}}.
    \end{equation}
\end{lem}
\begin{proof}
	We first prove \eqref{eq:refined mod gap}.
	Using the decomposition $w_{2L-1}=P_{2L-1}+\eps_{2L-1}$ with
	\begin{align*}
		B_QP_{2L-1}&=-\tfrac{\sqrt{2}}{2}i^{L-1}(ib_L+\eta_L),
	\end{align*}
	we have
	\begin{align*}
		AR\td b_{L}&=(B_Qw_{2L-1},i^L\chi_R)_r
		=b_L(-\tfrac{\sqrt{2}}{2},\chi_R)_r
		+(B_Q\eps_{2L-1},i^{L}\chi_R)_r,
		\\
		AR\td \eta_{L}&=(B_Qw_{2L-1},i^{L-1}\chi_R)_r
		=\eta_L(-\tfrac{\sqrt{2}}{2},\chi_R)_r
		+(B_Q\eps_{2L-1},i^{L-1}\chi_R)_r.
	\end{align*}
	Thanks to $B_Q=(y-\calH)\langle y\rangle^{-1}$, we obtain
	\begin{align*}
		(B_Qf,g)_r=(\langle y\rangle^{-1}f,(y+\calH)g)_r.
	\end{align*}
	Due to this and $(-\tfrac{\sqrt{2}}{2},\chi_R)_r=AR$, \eqref{eq:coercivity A1} and \eqref{eq:b lambda relation} imply
	\begin{align*}
		|b_L-\td b_{L}|+|\eta_L-\td \eta_{L}|&\lesssim \|\eps_{2L-1}\|_{\dot\calH^1}\|y\chi_R\|_{L^2}R^{-1} \\
        &\lesssim \lambda^{2L}R^{\frac{1}{2}}\sim |b_L| \lambda^{\frac{1}{2}}R^{\frac{1}{2}}\sim |b_L|\lambda^{\frac{1}{8L}}.
	\end{align*}
	Now, we show \eqref{eq:refined mod equ}. 
	From the definition of $\td b_{L}$ and $B_QB_Q^*=I$, we compute by
	\begin{align}
		AR(\partial_s-(2L-\tfrac{1}{2})\tfrac{\lambda_s}{\lambda})\td b_{L}=
		&((\partial_s-(2L-\tfrac{1}{2})\tfrac{\lambda_s}{\lambda})w_{2L-1},i^LB_Q^*\chi_R)_r \label{eq:refined mod equ b-1}
		\\
		&-\tfrac{R_s}{R}[(w_{2L-1}, i^LB_Q^*(y\partial_y\chi_R+\chi_R))_r]. \label{eq:refined mod equ b-2}
	\end{align}
	From $w_{2L-1}$ equation \eqref{eq:w2k-1 equ}, we compute by 
	\begin{align}
		\begin{split}
			\eqref{eq:refined mod equ b-1}=
			&\tfrac{\lambda_s}{\lambda}(y\partial_yw_{2L-1},i^LB_Q^*\chi_R)_r-\gamma_s(iw_{2L-1},i^LB_Q^*\chi_R)
			\\
			&+(-iH_ww_{2L-1},i^LB_Q^*\chi_R).
		\end{split} \label{eq:refined mod equ b-1-1}
	\end{align}
	Note that $(w_{2L-1},i^{L-1}B_Q^*\chi_R)=AR\td \eta_{L}$. We rewrite by \eqref{eq:refined mod equ b-2}
	\begin{align}
		\begin{split}
			\eqref{eq:refined mod equ b-2}&=
			-\tfrac{R_s}{R}[(w_{2L-1}, i^L\partial_y(yB_Q^*\chi_R))_r+(w_{2L-1}, i^L[B_Q^*,\partial_yy]\chi_R)_r]
			\\
			&=\tfrac{R_s}{R}(y\partial_yw_{2L-1}, i^LB_Q^*\chi_R)_r-\tfrac{R_s}{R}(w_{2L-1}, i^L[B_Q^*,\partial_yy]\chi_R)_r.
		\end{split}  \label{eq:refined mod equ b-2-1}
	\end{align}
	By \eqref{eq:refined mod equ b-1-1} and \eqref{eq:refined mod equ b-2-1}, we derive
	\begin{align}
		AR[(\partial_s-(2L-\tfrac{1}{2})\tfrac{\lambda_s}{\lambda})\td b_{L}+\gamma_s \td \eta_{L}]
		&=
		\tfrac{(\lambda R)_s}{\lambda R}(y\partial_yw_{2L-1},i^LB_Q^*\chi_R)_r \label{eq:refined mod equ b-3-1}
		\\
		&-\tfrac{R_s}{R}(w_{2L-1}, i^L[B_Q^*,\partial_yy]\chi_R)_r \label{eq:refined mod equ b-3-2}
		\\
		&+(-iH_ww_{2L-1},i^LB_Q^*\chi_R). \label{eq:refined mod equ b-3-3}
	\end{align}
	Here, from $R= \lambda^{-(1-\frac{1}{4L})}$, we have
	\begin{align*}
		\left|\frac{(\lambda R)_s}{\lambda R}\right|\sim\left|\frac{\lambda_s}{\lambda}\right|\sim \left|\frac{R_s}{R}\right|\sim b_1.
	\end{align*}
	Using the decomposition $w_{2L-1}=P_{2L-1}+\eps_{2L-1}$, we estimate \eqref{eq:refined mod equ b-3-1} as
	\begin{align*}
		|\eqref{eq:refined mod equ b-3-1}|\lesssim b_1|b_L(y\partial_y(yQ),B_Q^*\chi_R)_r|+b_1|(y\partial_y\eps_{2L-1},i^LB_Q^*\chi_R)_r|.
	\end{align*}
	Using $\partial_y(yQ)=2^{-1}Q^3$, $B_Q=(y-\calH)\langle y\rangle^{-1}$, \eqref{eq:algebraic 1}, and \eqref{eq:algebraic 2}, we have
	\begin{align*}
		(y\partial_y(yQ),B_Q^*\chi_R)_r
		=2^{-1}(B_Q(yQ^3),\chi_R)_r
		=\tfrac{\sqrt{2}}{2}(\langle y\rangle^{-2},\chi_R)_r\lesssim 1.
	\end{align*}
	We also deduce
	\begin{align*}
		|(y\partial_y\eps_{2L-1},i^LB_Q^*\chi_R)_r|
		\lesssim\|\partial_y\eps_{2L-1}\|_{L^2}\|yB_Q^*\chi_R\|_{L^2}\lesssim \lambda^{2L}R^{\frac{3}{2}}.
	\end{align*}
	We note that
	\begin{align*}
		\|yB_Q^*\chi_R\|_{L^2}
		\lesssim \|y\chi_R\|_{L^2}+\|\calH\chi_R\|_{L^2}\lesssim R^{\frac{3}{2}}+R^{\frac{1}{2}}.
	\end{align*}
	Therefore, we obtain
	\begin{align}
		|\eqref{eq:refined mod equ b-3-1}|\lesssim b_1(|b_L|+ \lambda^{2L}R^{\frac{3}{2}}). \label{eq:refined mod equ b goal 1}
	\end{align}
	To estimate \eqref{eq:refined mod equ b-3-2}, we see that
    \begin{align*}
    	[B_Q^*,\partial_yy]&=-y\langle y\rangle^{-3}+y^2\langle y\rangle^{-3}\calH
        +\langle y\rangle^{-1}\partial_y[\calH, y]
        \\
        &=-y\langle y\rangle^{-3}+y^2\langle y\rangle^{-3}\calH
    \end{align*}
    by \eqref{eq:CommuteHilbertDerivative}.
    Thus, we deduce
	\begin{align}
		|\eqref{eq:refined mod equ b-3-2}|\lesssim b_1|b_L|R^{\frac{1}{2}}. \label{eq:refined mod equ b goal 2}
	\end{align}
	For \eqref{eq:refined mod equ b-3-3}, using $H_Q=A_Q^*A_Q$, we have
	\begin{align*}
		\eqref{eq:refined mod equ b-3-3}=(A_Qw_{2L-1},i^{L+1}A_QB_Q^*\chi_R)+(i\textnormal{NL}_{2L-1},i^LB_Q^*\chi_R),
	\end{align*}
    where $\textnormal{NL}_{2L-1}=(H_{Q}-H_{w})w_{2L-1}$.
	For the first term, thanks to $A_Q=\partial_yB_Q$, $B_QB_Q^*=I$, and $A_QP_{2L-1}=0$ (since $A_Q(yQ)=0$), we have
	\begin{align*}
		(A_Qw_{2L-1},i^{L+1}A_QB_Q^*\chi_R)=-(A_Q\eps_{2L-1},i^{L-1}\partial_{y}\chi_R).
	\end{align*}
	This leads us to
	\begin{align*}
		|(A_Qw_{2L-1},i^{L+1}A_QB_Q^*\chi_R)|\lesssim \lambda^{2L}R^{-\frac{1}{2}}.
	\end{align*}
	To control $(i\textnormal{NL}_{2L-1},i^LB_Q^*\chi_R)$, we claim that
	\begin{align}
		\|\textnormal{NL}_{2L-1}\|_{L^2}\lesssim \lambda^{2L+1}. \label{eq:ref mod nonlinear est}
	\end{align}
    We remark that
    \begin{align*}
		\textnormal{NL}_{2L-1}=\tfrac{1}{4}(Q^4-|w|^4)w_{2L-1}-(Q|D|Q-w|D|\ol w)w_{2L-1}.
	\end{align*}
	We first estimate quadratic terms, $\tfrac{1}{4}(Q^4-|w|^4)w_{2L-1}$. We have
	\begin{align*}
		\|(Q^4-|w|^4)w_{2L-1}\|_{L^2}\lesssim \|\wt\eps Q^3w_{2L-1}\|_{L^2}+\|\wt\eps^4w_{2L-1}\|_{L^2}
		\lesssim \lambda^{2L+1}.
	\end{align*}
	For the non-local term, we have
	\begin{align}
		\|(Q|D|Q-w|D|\ol w)w_{2L-1}\|_{L^2}
		\leq \|Q|D|\ol{\wt\eps}w_{2L-1}\|_{L^2}+\|\wt\eps|D|Qw_{2L-1}\|_{L^2}. \label{eq:refine mod nonlocal term}
	\end{align}
	For \eqref{eq:refine mod nonlocal term}, we have
	\begin{align*}
		Q|D|\ol{\wt\eps}w_{2L-1}&=yQ\calH[y\langle y\rangle^{-2}\partial_y(\ol{\wt\eps}w_{2L-1})]+Q\calH[\langle y\rangle^{-2}\partial_y(\ol{\wt\eps}w_{2L-1})],
		\\
		\wt\eps|D|Qw_{2L-1}
		&=
		y\langle y\rangle^{-2}\wt\eps\calH[y\partial_y(Qw_{2L-1})]+
		\langle y\rangle^{-2}\wt\eps|D|(Qw_{2L-1}).
	\end{align*}
	Thus, we deduce
	\begin{align*}
		\eqref{eq:refine mod nonlocal term}
		\lesssim b_1|b_L|\lesssim \lambda^{2L+1},
	\end{align*}
	and \eqref{eq:ref mod nonlinear est} is proved. Consequently, we obtain
	\begin{align*}
		|(i\textnormal{NL}_{2L-1},i^LB_Q^*\chi_R)|\lesssim \lambda^{2L+1}R^{\frac{1}{2}},
	\end{align*}
	and we arrive at
	\begin{align}
		|\eqref{eq:refined mod equ b-3-3}|
		\lesssim
		\lambda^{2L}R^{-\frac{1}{2}}+\lambda^{2L+1}R^{\frac{1}{2}}. \label{eq:refined mod equ b goal 3}
	\end{align}
	Combining \eqref{eq:refined mod equ b goal 1}, \eqref{eq:refined mod equ b goal 2}, and \eqref{eq:refined mod equ b goal 3}, we have
	\begin{align*}
		|(\partial_s-(2L-\tfrac{1}{2})\tfrac{\lambda_s}{\lambda})\td b_{L}+\gamma_s \td \eta_{L}|
		\lesssim b_1^{L+1}(\lambda^{\frac{1}{2}}R^{\frac{1}{2}}+\lambda^{-2+\frac{L+1}{2L}}R^{-\frac{3}{2}}+\lambda^{-1+\frac{L+1}{2L}}R^{-\frac{1}{2}}).
	\end{align*}
	Thanks to \eqref{eq:mod equ 1} with $R= \lambda^{-(1-\frac{1}{4L})}$, we conclude \eqref{eq:refined mod equ} for $\td b_{L}$.
	
	Similarly, for $\td \eta_{L}$, we have
	\begin{align*}
		AR[(\partial_s-(2L-\tfrac{1}{2})\tfrac{\lambda_s}{\lambda})\td \eta_{L}-\gamma_s \td \eta_{L}]
		&=
		\tfrac{(\lambda R)_s}{\lambda R}(y\partial_yw_{2L-1},i^{L-1}B_Q^*\chi_R)_r
		\\
		&-\tfrac{R_s}{R}(w_{2L-1}, i^{L-1}[B_Q^*,\partial_yy]\chi_R)_r 
		\\
		&+(-iH_ww_{2L-1},i^{L-1}B_Q^*\chi_R). 
	\end{align*}
    Using a similar argument as in \eqref{eq:refined mod equ b-3-1}--\eqref{eq:refined mod equ b-3-3}, the proof is concluded.
\end{proof}

\subsection{Proof of Proposition \ref{prop:bootstrap} and Theorem \ref{thm:main}}
We complete the proof of our main bootstrap proposition and the main theorem.
\begin{proof}[Proof of Proposition \ref{prop:bootstrap}]
	To apply the bootstrap argument, we aim to improve the bootstrap assumption \eqref{eq:bootstrap assump} except the assumption for unstable modes. More precisely, for the solutions satisfying \eqref{eq:bootstrap assump} on $[s_0, s^*]$, we aim to show
	\begin{align}
		\frac{1}{1.5}\leq \frac{b_1^L}{\lambda^{2L-\frac{1}{2}}}(s) \leq 1.5\quad \text{and} \quad |\calU_1(s)|\leq 5s^{-\kappa} \quad \text{for} \quad s\in [s_0, s^*]. \label{eq:bootstrap goal}
	\end{align}
	From \eqref{eq:bootstrap assump} and \eqref{eq:mod equ 1}, we have
	\begin{align*}
		-\frac{\lambda_s}{\lambda}=b_1+O(\lambda^2)=c_1s^{-1}+O(s^{-(1+\kappa)}), \text{ or simply, }   \partial_s \log(s^{c_1}\lambda(s))=O (s^{-(1+\kappa)})
	\end{align*}
	where $c_1=\frac{2L}{4L-1}$ given in \eqref{eq:ck}. After integrating this, we derive
	\begin{align}
		\lambda(s)=\bigg(\frac{s_0}{s}\bigg)^{c_1}\bigg[1+O\bigg(\frac{1}{s_0^{\kappa}}\bigg)\bigg]. \label{eq:lambda s rate}
	\end{align}
	From \eqref{eq:lambda s rate} and \eqref{eq:bootstrap assump} with $b_1^e=c_1s^{-1}$, we have
	\begin{align*}
		\frac{b_1^L}{\lambda^{2L-\frac{1}{2}}}(s)=\bigg(\frac{c_1}{s_0}\bigg)^L\bigg[1+O\bigg(\frac{1}{s_0^{\kappa}}\bigg)\bigg]
		=\frac{b_1^L}{\lambda^{2L-\frac{1}{2}}}(s_0)\bigg[1+O\bigg(\frac{1}{s_0^{\kappa}}\bigg)\bigg]. 
	\end{align*}
	Thus, we obtain the first part of \eqref{eq:bootstrap goal} by taking large $s_0$ with the initial condition \eqref{eq:init lambda b relation}. 
	
	We use the modified modulation parameters 
    \begin{align*}
        (\td b_1,\cdots,\td b_{L-1}, \td b_{L})=\td b\coloneqq (b_1,\cdots,b_{L-1}, \td b_{L}),
    \end{align*}
    where $\td b_{L}$ is given by \eqref{eq:def refine mod}.
    Similarly, we define $\td \eta\coloneqq (\eta_1,\cdots,\eta_{L-1}, \td \eta_{L})$. We also use the corresponding fluctuation $(\td\calU,\td\calV)=(P\td U,Q\td V)$ where $(\td U, \td V)$ is defined by
	\begin{align*}
		\td U_k\coloneqq s^k(\td b_k-b_k^e),\quad 
		\td V_k\coloneqq s^k(\td \eta_k-\eta_k^e) \quad \text{for} \quad 1\leq k\leq L.
	\end{align*}
	We remark that, from \eqref{eq:refined mod gap}, we have
	\begin{align}
		|\calU-\td\calU|+|\calV-\td\calV|\lesssim s^L(|b_L-\td b_{L}|+|\eta_L-\td \eta_{L}|)\lesssim \lambda^{\frac{1}{8L}}. \label{eq:calUB tilde similar}
	\end{align}
	From \eqref{eq:linearization}, \eqref{eq:bootstrap assump}, \eqref{eq:mod equ 2}, and \eqref{eq:refined mod equ}, we have 
	\begin{align}
		|s\td\calU_s-D_{U}\td\calU|+|s\td\calV_s-D_{V}\td\calV|=O(\max(s^{-2\kappa},\lambda^{\frac{1}{8L}}))=O(s^{-2\kappa}). \label{eq:UV diag equ}
	\end{align}
	Now, we finish the proof of \eqref{eq:bootstrap goal}. From \eqref{eq:calUB tilde similar}, it suffices to show for the first coordinate of $\td \calU$, namely $\td \calU_1$ instead of $\calU_1$. \eqref{eq:UV diag equ} implies that $\td \calU_1$ satisfies
	\begin{align*}
		s (\td\calU_1)_s+\td\calU_1=(s\td\calU_1)_s=O(s^{-2\kappa}),
	\end{align*}
	and thus we obtain the desired bound by integrating the above equation: 
	\begin{align*}
		|\td \calU_1(s)|\leq C|s_0\td \calU_1(s_0)|s^{-1}+Cs^{-1}\textstyle{\int_{s_0}^s}\tau^{-2\kappa}d\tau\leq C^\prime s^{-2\kappa}\leq 4s^{-\kappa}.
	\end{align*}
	
    Now, we show $s^*=\infty$ by contradiction. For a convenience, we denote by the unstable modes
    \begin{align*}
        (\calU,\calV)_{un}(s)\coloneqq s^\kappa( \calU_2(s),\cdots,\calU_L(s), \calV_1(s),\cdots,\calV_L(s)).
    \end{align*}
    We also use a similar notation for $(\td\calU, \td\calV)$ or $(\mathring\calU, \mathring\calV)$. Remark that $(\mathring\calU, \mathring\calV)$ is defined by \eqref{eq:calUV def}, which represents the fluctuation of initial data.
    Since we assumed that $s^*<\infty$ for all $2(\mathring\calU,\mathring\calV)_{un}(s_0)\in \calB^{2L-1}$, $ (\calU,\calV)_{un}(s^*)$ and $(\td\calU,\td\calV)_{un}(s^*)$ are well-defined. Moreover, by the definition of $s^*$ \eqref{eq:bootstrap exit time} with \eqref{eq:bootstrap assump}, we have $(\calU,\calV)_{un}(s^*)\in 1.01\partial\calB^{2L-1}$. This and \eqref{eq:calUB tilde similar} implies $(\td\calU,\td\calV)_{un}(s^*)\in \bbR^{2L-1}\setminus\{0\}$. 
    Now, we can define a map $\Phi:\calB^{2L-1}\to \bbR^{2L-1}\setminus\{0\}$ as follows:
    \begin{align*}
        \Phi((\mathring\calU,\mathring\calV)_{un}(s_0))\coloneqq (\td\calU,\td\calV)_{un}(s^*).
    \end{align*}
    For the continuity of this map $\Phi$, it suffices to show the map $(\td\calU,\td\calV)_{un}(s_0)\mapsto s^*$ is continuous because the maps $(\mathring\calU,\mathring\calV)(s_0)\mapsto (\calU,\calV)(s)\mapsto (\td\calU,\td\calV)(s)$ are continuous for $s\in [s_0,s^*]$.
    For all $s\in [s_0, s^*]$ such that $s^{\kappa}(\sum_{j=2}^L\td \calU_j^2+\sum_{i=1}^L\td \calV_i^2)\geq \frac{1}{2}$, using \eqref{eq:UV diag equ} and \eqref{eq:diagonalization}, we have
	\begin{align*}
		&\partial_s\bigg[s^\kappa\bigg(\sum_{j=2}^L\td \calU_j^2+\sum_{i=1}^L\td \calV_i^2\bigg)\bigg]
		\\
        &=\frac{1}{s^{1-\kappa}}\bigg[
        \sum_{j=2}^L\bigg(\kappa+\frac{2j}{4L-1}\bigg)\td \calU_j^2
        +\sum_{i=1}^L\bigg(\kappa+\frac{2i}{4L-1}\bigg)\td \calV_i^2+O(s^{-2\kappa})\bigg]>0.
	\end{align*}
    Hence, the flow map $s\mapsto (\td\calU,\td\calV)_{un}(s)$ has an outgoing behavior, and leads us to the continuity of the map $(\td\calU,\td\calV)_{un}(s_0)\mapsto s^*$ by some standard arguments (see Lemma 6 in \cite{CoteMartelMerle2011TopologicalExample} or Lemma 5.9 in \cite{KimKimKwon2024arxiv}). Therefore, the map $\Phi$ is continuous. We denote by $\Pi_k$ the standard radial projection, which maps each point in $\bbR^{k}\setminus\{0\}$ onto $\partial\calB^{k}$. This projection map is continuous, and thus we obtain a continuous map $\Pi_{2L-1}\circ \Phi:\calB^{2L-1}\to\partial\calB^{2L-1}$. 

    Now, we show that $\Pi_{2L-1}\circ \Phi$ is nearly the identity map on $\partial\calB^{2L-1}$, i.e., 
    \begin{equation*}
        \sup_{x\in \partial\calB^{2L-1}}|\Pi_{2L-1}\circ \Phi(x)-x|\leq 0.1.
    \end{equation*}
    For $(\mathring\calU,\mathring\calV)_{un}(s_0)\in \partial\calB^{2L-1}$, we have $(\td\calU,\td\calV)_{un}(s_0)\in \calB^{2L-1}_{0.01}((\mathring\calU,\mathring\calV)_{un}(s_0))$ by \eqref{eq:init decompose}, \eqref{eq:mod para sim}, and \eqref{eq:calUB tilde similar}. Since the flow map $s\mapsto (\td\calU,\td\calV)_{un}(s)$ tends outward, and using \eqref{eq:calUB tilde similar}, we conclude that $(\td\calU,\td\calV)_{un}(s^*)\in \calB^{2L-1}_{0.1}((\mathring\calU,\mathring\calV)_{un}(s_0))$.
    This implies that $P_{2L-1}\circ \Phi$ is uniformly close to the identity on $\partial\calB^{2L-1}$ in the $C^0$ topology. Therefore, by the Brouwer fixed point theorem, we arrive at a contradiction, which completes the proof.
\end{proof}
Next, we finish the proof of the main theorem.
\begin{proof}[Proof of Theorem \ref{thm:main}]
    Thanks to Proposition \ref{prop:bootstrap} and \eqref{eq:lambda s rate}, there exist $v_0$ of the form \eqref{eq:init form}. Since the argument only uses \eqref{eq:init radiation aux}, we may choose $\mathring\eps_0 \in Z^\perp \cap H^\infty_e$. Hence, $v_0 \in H^\infty$. Moreover, there exists an associated constant $C(v_0)>0$ with
    \begin{align*}
    	\lambda(s)=C(v_0)s^{-\frac{L}{2L-\frac{1}{2}}}(1+O(s_0^{-\kappa}))
    \end{align*}
    for $s\in [s_0,\infty)$. Using $T-t=\int_{s}^\infty \lambda^{2}(\tau)d\tau <\infty$, we derive $T<\infty$ and
    \begin{align*}
    	T-t=C^\prime(v_0)s^{-\frac{1}{4L-1}}(1+o_{t\to T}(1))=C^{\prime\prime}(v_0)\lambda(s)^{\frac{1}{2L}}(1+o_{t\to T}(1)).
    \end{align*}
    Therefore, we conclude
    \begin{align*}
    	\lambda(t)=c(v_0)(T-t)^{2L}(1+o_{t\to T}(1)).
    \end{align*}
    In addition, from \eqref{eq:mod equ 1}, we have
    \begin{align*}
        |\gamma_s|\lesssim |\eta_1|\lesssim b_1^{1+\kappa}, \textnormal{ or equivalently }
        |\gamma_t|\lesssim  \lambda^{-\frac{1}{2L}+\kappa}
        \sim_{v_0} (T-t)^{-1+2\kappa L}.
    \end{align*}
    Hence, $\gamma(t)$ converges to some $\gamma^*$ as $t\to T$. The strong convergence \eqref{eq:thm strong converge} and the regularity of asymptotic profile $v^*\in H^1_{e}$ follows in the same manner as in \cite{KimKimKwon2024arxiv}. By \eqref{eq:init energy aux}, we can make the mass arbitrarily close to $M(Q)$ by taking $\mathring \delta$ and $s_0^{-1}$ smaller. Finally, the analogous result for \eqref{CMdnls} can be obtained by taking the inverse of the gauge transform. For details, refer to the proof of Theorem 1.1 in \cite{KimKimKwon2024arxiv}.
\end{proof} 
\begin{rem}[Regularity of asymptotic profile]\label{rem:asym profile}
    As mentioned in the introduction, we believe it is possible to obtain $v^*\in H^{2L-\frac{1}{2}-}$. Specifically, for $L=1$, following similar ideas as in the proofs of \eqref{eq:e L inf} and \eqref{eq:e2k-1 L inf}, we can derive $\|[\wt\eps]_{\lambda,\gamma}\|_{\dot H^{\frac{3}{2}-}}\lesssim 1$, which implies $v^*\in H^{\frac{3}{2}-}$. However, we cannot achieve $H^{\frac{3}{2}}$, as $T_1\not\in \dot H^{\frac{3}{2}}$. 
    
    For $L\geq 2$, we deduce that $v^*\in H^{2}$ as follows: as in \cite{KimKimKwon2024arxiv}, we have $[\wt\eps]_{\lambda,\gamma}\to v^*$ in $L^2$ and weakly in $H^1$. Setting $\eps^\prime \coloneqq \wt\eps -T_1\chi_{\lambda^{-1}}$, we find that $[T_1\chi_{\lambda^{-1}}]_{\lambda,\gamma}\to0$ in $L^2$, implying $[\eps^\prime]_{\lambda,\gamma}\to v^*$ in $L^2$. Moreover, we have $\|[\eps^\prime]_{\lambda,\gamma}\|_{H^2}\lesssim 1$, and thus, $[\eps^\prime]_{\lambda,\gamma} \rightharpoonup v^*$ weakly in $H^2$ by the uniqueness of limit, indicating that $v^*\in H^2$.
    
    This suggests that if the coercivity for $\calL_k$ holds and higher-order energy estimates for the radiation terms of $w$ are established, such as $\|\eps\|_{\dot \calH^k}\lesssim \lambda^k$, we may derive $v^*\in H^{2L-\frac{1}{2}-}$. This regularity may also be obtainable by using information about the nonlinear variables, but this remains uncertain for now.
\end{rem}

\section{Tail computation}\label{sec:tail compute}
In this section, we prove Lemma~\ref{lem:init approx}. Our goal is to compare the higher-order variables $w_{2k-1}(s_0)$ at the initial time with the corresponding linearized quantities $\td\bfD_Q^{2k-2}L_Q\wt\eps(s_0)$. In general, the decomposition $\wt\eps=T+\eps$ does not by itself yield the smallness needed for this comparison, since nonlinear contributions may remain which are smaller than the leading term, but still not sufficiently small for the estimate in Lemma~\ref{lem:init approx}. In a typical tail computation, one then introduces further corrector profiles $S$ and rewrites $\wt\eps=T+S+\eps$ so as to cancel these leading nonlinear contributions. In the present setting, however, the profiles $T$ already yield the necessary cancellation, and no further corrector is needed. Thus, the point of this section is to show that, for the purpose of Lemma~\ref{lem:init approx}, the initial decomposition already gives the required smallness.

We emphasize that this argument is needed only for the initial data. For the later dynamics, the hierarchy structure already gives the required control, and this is sufficient for the topological shooting argument. In particular, for $s>s_0$, the expected scaling law $\|w_j(s)\|_{L^2}\sim \lambda(s)^j$ is indeed the correct one. The point here is that the implicit constant in this comparison is determined by the initial data, and its smallness is not automatic from the definition of the initial decomposition. At first sight, one might expect the factor $\lambda^j$ to be enough, but this is misleading at $s=s_0$: if the corresponding initial quantity before rescaling is as large as $\lambda^{-j}$, then $w_j(s_0)$ may still be of order one. Therefore, although the scaling law is correct for the later time evolution, the smallness required at $s=s_0$ is a separate issue and must be verified directly. This is exactly what we do in this section.

\begin{rem}
    Compared with the earlier work \cite{KimKimKwon2024arxiv}, we use $i(1+y^2)Q$ instead of $iy^2Q$ in the definition of the profiles $T_{1}$.
    This modification does not affect the kernel structure, since the added $iQ$ component is already contained in the phase direction. Its role appears only in the tail computation for the initial data. More precisely, if one keeps $iy^2Q$, then nonlinear contributions such as the $b_1^2$-term survive in the comparison between $w_{2k-1}(s_0)$ and $\td \bfD_Q^{2k-2}L_Q\wt\eps(s_0)$, so one would need additional correctors to remove them.
\end{rem}

Denote $\wt\eps_0 \coloneqq \wt\eps(s_0) = V_L + \mathring{\eps}_0$ where $V_L$ is defined by
\[V_L\coloneqq \sum_{j=1}^LT_{2j-1}(\mathring b_j, \mathring\eta_j)\chi_{\mathring\delta\mathring\la_{0}^{-1}}.\]
In addition, we define 
\begin{align*}
	\bfD_1f\coloneqq& \wt\eps_0\calH (Qf),
	&&\bfD_2f\coloneqq Q\calH(\ol{\wt\eps_0}f),
	&&\bfD_3f\coloneqq \wt\eps_0\calH (\ol{\wt\eps_0} f).
\end{align*}
One can easily check that
\begin{align*}
    \td\bfD_w=\td\bfD_Q+\tfrac12 \bfD_1+\tfrac12 \bfD_2+\tfrac12 \bfD_3.
\end{align*}

\begin{lem}\label{lem:bfD123 esti}
    We have
    \begin{equation}\label{eq:bfD first}
        \begin{aligned}
            \|\bfD_1 f\|_{L^2}+\|\bfD_2 f\|_{L^2}+\|\bfD_3 f\|_{L^2}\lesssim_{\mathring\delta} \lambda_0\|f\|_{L^2},
            \\
            \|Q \bfD_1 f\|_{L^2}+\|Q \bfD_2 f_e\|_{L^2}+\|Q \bfD_3 f\|_{L^2}\lesssim_{\mathring\delta} \mathring b_1\|f\|_{L^2}.
        \end{aligned}
    \end{equation}
    Here $f_e$ denotes the even part of $f$.
	For $1\leq \ell \leq 2L-3$ and $1\leq j\leq 3$, we have
    \begin{equation}\label{eq:bfD mth}
        \begin{aligned}
            \|\td\bfD_Q^\ell \bfD_j f\|_{L^2}
    		&\lesssim_{\mathring\delta} \lambda_0 \sum_{a=0}^\ell \lambda_0^a
            \|\td\bfD_Q^{\ell-a} f\|_{L^2},
            \\
            \|QB_Q\td\bfD_Q^\ell \bfD_1 f\|_{L^2}
    		&\lesssim_{\mathring\delta} 
            \lambda_0\sum_{a=0}^{\ell+1} \lambda_0^a \|\td\bfD_Q^{\ell+1-a} f\|_{L^2}
            +\lambda_0^{\ell+2-\frac{\ell+2}{4L}}|{\textstyle\int_\bbR} Q f|,
            \\
            \|QB_Q\td\bfD_Q^\ell \bfD_2 f\|_{L^2}
    		&\lesssim_{\mathring\delta} \lambda_0\sum_{a=0}^{\ell+1} \lambda_0^a \|\td\bfD_Q^{\ell+1-a} f\|_{L^2},
            \\
            \|QB_Q\td\bfD_Q^\ell \bfD_3 f\|_{L^2}
    		&\lesssim_{\mathring\delta} 
            \lambda_0^{1-\frac{\ell+1}{4L}}\sum_{a=0}^{\ell+1} \lambda_0^a\|\td\bfD_Q^{\ell+1-a}f\|_{L^2}.
        \end{aligned}
    \end{equation}
\end{lem}
\begin{proof}
    For notational convenience, we omit the ring notation. Also, let $R=\mathring\delta\mathring\la_{0}^{-1}=\delta \la_0^{-1}$.
    Note that using \eqref{eq:BQBQstar equal I} and \eqref{eq:DQ BQ decomp},
    \begin{align*}
        \|\td\bfD_Q^j \bfD_i f\|_{L^2}^2=\|B_Q\td\bfD_Q^j \bfD_i f\|_{L^2}^2. 
    \end{align*}
    Denote
    \begin{align}\label{eq:def phi}
    	\phi_1(y)=\chi_R(y),\qquad \phi_n(y)=y^{2n-4}(1+y^2)\chi_R(y), \quad (n\geq 2)
    \end{align}
    and 
    \begin{align*}
    	{\bm\beta}_n=\mathring{\bm\beta}_n\coloneqq  -\frac{(-i)^{n-1}}{4n\cdot(2n-2)!}(i\mathring{b
        }_n+\mathring{\eta}_n)=-\frac{(-i)^{n-1}}{4n\cdot(2n-2)!}(i{b
        }_n+{\eta}_n).
    \end{align*}
    Then, we have
    \begin{align*}
    	\bfD_1f=\bfD_{1,\eps} f + \sum_{n=1}^L \bfD_{1,n}f,\quad  &\bfD_{1,n}f\coloneqq {\bm\beta}_n \phi_n(1+y^2)Q\calH (Qf), \\
        \bfD_2 f = \bfD_{2,\eps} f +
        \sum_{n=1}^L \bfD_{2,n}f,\quad  &\bfD_{2,n}f
        \coloneqq
        \ol{ {\bm\beta}_n} Q\calH\bigl((1+y^2)Q\phi_n f\bigr), \\
        \bfD_3f =\bfD_{3,\eps} f +\sum_{n,m=1}^L\bfD_{3,n,m}f,\quad &\bfD_{3,n,m}f
    	\coloneqq
    	{\bm\beta}_n \ol{{\bm\beta}_m}\phi_n(1+y^2)Q\calH((1+y^2)Q\phi_m f), \\
        \bfD_{1,\eps} f \coloneqq \mathring{\eps}_0 \calH (Qf),\; \; \bfD_{2,\eps} f \coloneqq Q \calH &(\ol{\mathring{\eps}_0}f),\;  \;\bfD_{3,\eps} f \coloneqq (V_L+\mathring{\eps}_0 )\calH (\ol{\mathring{\eps}_0}f)+\mathring{\eps}_0 \calH (\ol{V_L}f).
    \end{align*}
   
    \textbf{Step 1.} We first prove \eqref{eq:bfD first}. 
    It suffices to show, for $1\leq n,m\leq L$, 
    \begin{align*}
        \|\bfD_{1,n} f\|_{L^2}+\|\bfD_{2,n} f\|_{L^2}+\|\bfD_{3,n,m} f\|_{L^2}\lesssim \lambda_0\|f\|_{L^2},
        \\
        \|Q \bfD_{1,n} f\|_{L^2}+\|Q \bfD_{2,n} f_e\|_{L^2}+\|Q \bfD_{3,n,m} f\|_{L^2}\lesssim  b_1\|f\|_{L^2},
    \end{align*}
    since we can easily estimate $\bfD_{j,\eps}$ as follows: for $j=1,2,3$, $$\|\bfD_{j,\eps}f \|_{L^2} \lesssim \| \mathring{\eps}_0 \|_{L^{\infty}}( \| Q \|_{L^{\infty}}+ \| V_L \|_{L^{\infty}}+ \| \mathring{\eps}_0 \|_{L^{\infty}}) \|f\|_{L^2} \lesssim \mathring{\lambda}_0^{10L} \| f\|_{L^2}.$$ 
    For $\bfD_{1,n}$, thanks to \eqref{eq:CommuteHilbert}, we have
    \begin{align*}
        \bfD_{1,n}f&=\bm\beta_n\phi_n(1+y^2)Q\calH (Qf)
        \\
        &=\bm\beta_n\phi_nQ\calH (Qf)+\bm\beta_n\phi_n yQ\calH (yQf)+\tfrac{1}{\pi}\bm\beta_n\phi_n yQ {\textstyle \int_{\bbR}} Qf.
    \end{align*}
    Thus, we get
    \begin{align*}
        \|\bfD_{1,n} f\|_{L^2} &\lesssim |\bm\beta_n| \left(\|\phi_n \langle y \rangle Q\|_{L^\infty} \|\calH( (1+y)Q f)\|_{L^2} + \|\phi_n yQ^2\|_{L^2}\left| {\textstyle \int_{\bbR}} Qf\right| \right)
        \\
        &\lesssim b_1^n R^{2n-2} \|f\|_{L^2}\sim_{\mathring\delta} \lambda_0^{2-\frac{n}{2L}} \|f\|_{L^2} \lesssim \lambda_0\|f\|_{L^2},\\
    \end{align*}
    and
    \begin{align*}
        \|Q\bfD_{1,n} f\|_{L^2}  &\lesssim  b_1^n R^{2n-3} \|f\|_{L^2} \sim_{\mathring\delta} b_1\lambda_0^{1-\frac{n-1}{2L}} \|f\|_{L^2} \lesssim b_1\|f\|_{L^2},  \quad &&(n\neq 1),
        \\
         &\lesssim  b_1 \|f\|_{L^2}.  \quad &&(n= 1).
    \end{align*}
    Similarly, for $\bfD_{2,n}$, again using \eqref{eq:CommuteHilbert}, we obtain
    \begin{align*}
        \bfD_{2,n}f=\ol{\bm\beta_n}Q\calH (Q\phi_nf)+\ol{\bm\beta_n} yQ\calH (yQ\phi_nf)-\tfrac{1}{\pi}\ol{\bm\beta_n} Q {\textstyle \int_{\bbR}} yQ\phi_n f.
    \end{align*}
    Thus, we have
    \begin{align*}
        \|\bfD_{2,n} f\|_{L^2} &\lesssim b_1^n \left(\|\langle y \rangle Q \phi_n f\|_{L^2} +\left| {\textstyle \int_{\bbR}} yQ\phi_n f\right| \right)
        \\
        &\lesssim  b_1^n R^{2n-\frac32} \|f\|_{L^2}  \sim_{\mathring\delta} \lambda_0^{1+\frac12-\frac{n}{2L}} \|f\|_{L^2} \leq \lambda_0\|f\|_{L^2}.
    \end{align*}
    For the $Q$-weighted part, we need the evenness $f_e$. Thanks to this and \eqref{eq:CommuteHilbert}, we have
    \begin{align*}
        \bfD_{2,n}f_e&=\ol{\bm\beta_n}Q\calH (Q\phi_nf_e)+\ol{\bm\beta_n} yQ\calH (yQ\phi_nf_e)
        \\
        &=\ol{\bm\beta_n}(1+y^2)Q\calH (Q\phi_nf_e)
        -\tfrac{1}{\pi}\ol{\bm\beta_n} yQ {\textstyle \int_{\bbR}} Q\phi_n f_e,
    \end{align*}
    which yields
    \begin{align*}
        \|Q\bfD_{2,n} f_e\|_{L^2}  &\lesssim b_1^n R^{2n-\frac52}\|f_e\|_{L^2} \sim_{\mathring\delta} b_1\lambda_0^{\frac12-\frac{n-1}{2L}} \|f_e\|_{L^2} \lesssim b_1\|f_e\|_{L^2}, \quad &&(n\neq 1),
        \\
        &\lesssim   b_1 \|f_e\|_{L^2}. \quad &&(n= 1).
    \end{align*}
    Finally, for $\bfD_{3,n,m} f$, by a direct computation, we obtain
    \begin{align*}
        \|\bfD_{3,n,m} f\|_{L^2} &\lesssim b_1^{n+m} R^{2n-1+2m-1}\|f\|_{L^2} \lesssim_{\mathring\delta} \lambda_0 \|f\|_{L^2},
        \\
        \|Q\bfD_{3,n,m} f\|_{L^2} &\lesssim b_1^{n+m} R^{2n-2+2m-1} \|f\|_{L^2}
        \lesssim_{\mathring\delta} b_1\|f\|_{L^2}.
    \end{align*}
    
    \textbf{Step 2.} In this step, we show \eqref{eq:bfD mth} for $\bfD_{1,n}$. For $1\le n \le L$, define
    \begin{align*}
    	p_n\coloneqq yQ^2\phi_n-\calH(Q^2\phi_n)=\begin{cases}
    	    yQ^2(\chi_R-1)+\calH(Q^2(1-\chi_R)), & n=1, \\
             2 y^{2n-3}\chi_R -2\calH( y^{2n-4}\chi_R), & n\ge 2.
    	\end{cases}
    \end{align*}
    We can easily check that
    \begin{align}\label{eq:phi and p esti}
    	\|\partial_y^a\phi_n\|_{L^\infty}\lesssim R^{2n-2-a}, \quad \|\partial_y^a p_n\|_{L^\infty}\lesssim R^{2n-3-a}.
    \end{align}
    Next, we remark that
    \begin{align*}
    	B_Q \bfD_{1,n}f={\bm\beta}_n B_Q(\phi_n (1+y^2)Q\calH (Qf)).
    \end{align*}
    We claim that for any $g\in H^{\frac{1}{2}+}$,
    \begin{align*}
        B_Q[(1+y^2)Q\calH(QB_Q^*g)]=2\calH g.
    \end{align*}
    Indeed, using \eqref{eq:HilbertProductRule} with $f=Q^2$ and $g=g$, we get
    \begin{align*}
        B_Q[(1+y^2)Q\calH(QB_Q^*g)]&=(y-\calH)[\langle y\rangle^{-1}(1+y^2)Q\calH(Q\langle y\rangle^{-1}\{(y+\calH)g\})]
        \\
        &=(y-\calH)\calH [Q^2(y+\calH)g ] \\
        &=y\calH [yQ^2g]+y\calH[Q^2\calH g]+yQ^2g+Q^2\calH g\\
        &= y^2Q^2\calH g-yQ^2f+yQ^2f+Q^2\calH g= 2 \calH g.
    \end{align*}
    Thus, from \eqref{eq:BQBQstar equal I}, we deduce
    \begin{align*}
    	(1+y^2)Q\calH (Qf)=2B_Q^*\calH B_Qf +\tfrac{1}{\pi} yQ {\textstyle\int_\bbR} Q f.
    \end{align*}
    Thus, we get
    \begin{align*}
    	B_Q \bfD_{1,n}f=2{\bm\beta}_n B_Q(\phi_n B_Q^*\calH B_Qf)
    	+\tfrac{1}{\pi} {\bm\beta}_n B_Q(\phi_n\cdot yQ)  {\textstyle\int_\bbR} Q f.
    \end{align*}
    Next, we claim
    \begin{align*}
    	B_Q(\phi_n B_Q^*f)=\phi_n f +\tfrac12 p_n \calH f-\tfrac12 \calH(p_n f).
    \end{align*}
    Applying \eqref{eq:HilbertProductRule} with $f=f$ and $g=Q^2\phi_n$, we have
    \begin{align*}
    	2B_Q(\phi_n B_Q^*f)&=(y-\calH) [Q^2\phi_n (y+\calH) f]
    	\\
    	&= y^2 Q^2\phi_n f + y Q^2\phi_n\calH f-\calH (yQ^2\phi_n f)-\calH[Q^2\phi_n\calH f]
        \\
        &= 2\phi_n f + p_n \calH f-\calH(p_n f).
    \end{align*}
    Let $g=B_Qf$. Then, we get
    \begin{align*}
    	&\|\td\bfD_Q^\ell \bfD_{1,n} f\|_{L^2}=\|B_Q^*\partial_y^\ell B_Q\bfD_{1,n} f\|_{L^2}=\|\partial_y^\ell B_Q\bfD_{1,n} f\|_{L^2}
    	\\
    	&\lesssim|b_1|^n\left(\|\partial_y^\ell (\phi_n \calH g -\tfrac12 p_n g-\tfrac12 \calH(p_n \calH g))\|_{L^2}+ \|\partial_y^\ell B_Q(\phi_n  yQ)\|_{L^2} |{\textstyle\int_\bbR} Q f|\right),
        \\
        &\|QB_Q\td\bfD_Q^\ell \bfD_{1,n} f\|_{L^2}=\|Q\partial_y^\ell B_Q\bfD_{1,n} f\|_{L^2}
    	\\
    	&\lesssim|b_1|^n\left(\|Q\partial_y^\ell (\phi_n \calH g -\tfrac12 p_n g-\tfrac12 \calH(p_n \calH g))\|_{L^2}+ \|Q\partial_y^\ell B_Q(\phi_n  yQ)\|_{L^2} |{\textstyle\int_\bbR} Q f|\right).
    \end{align*}
    We first estimate the second part, $\partial_y^\ell B_Q(\phi_n  yQ)$.
    We have
    \begin{align*}
        \|\partial_y^\ell B_Q(\phi_n  yQ)\|_{L^2}&\lesssim R^{2n-2-\ell+\frac12},  &&(n\neq 1),
        \\
    	\|\partial_y^\ell B_Q(\phi_1  yQ)\|_{L^2}=\|\partial_y^\ell B_Q(yQ(1-\chi_R))\|_{L^2}&\lesssim R^{-\ell+\frac12}, \quad &&(n=1).
    \end{align*}
    Similarly, we get, for $2n-2\neq \ell$,
    \begin{align*}
        \|Q\partial_y^\ell B_Q(\phi_n  yQ)\|_{L^2}
        \lesssim  R^{2n-2-\ell-\frac{1}{2}}.
    \end{align*}
    For the case $2n-2= \ell$, we obtain
    \begin{align*}
        \|Q\partial_y^\ell B_Q(\phi_n \ yQ)\|_{L^2}
        \lesssim 1.
    \end{align*}
    Thus, we have
    \begin{equation}\label{eq:bfD1 worst}
        \begin{aligned}
            |b_1|^n \|\partial_y^\ell B_Q(\phi_n \ yQ)\|_{L^2}&\lesssim \lambda_0^{\ell+2-\frac12-\frac{n}{2L}}\lesssim \lambda_0^{\ell+1}.
            \\
            |b_1|^n \|Q\partial_y^\ell B_Q(\phi_n  yQ)\|_{L^2}&\lesssim \lambda_0^{\ell+2+\frac12-\frac{n}{2L}}\lesssim \lambda_0^{\ell+2}, &&(2n-2\neq\ell),
            \\
            |b_1|^n \|Q\partial_y^\ell B_Q(\phi_n  yQ)\|_{L^2}&\lesssim \lambda_0^{\ell+2-\frac{\ell+2}{4L}}, &&(2n-2=\ell).
        \end{aligned}
    \end{equation}
    These conclude
    \begin{align} \label{eq:bfD1 1}
        b_1^n\|Q\partial_y^\ell B_Q(\phi_n  yQ)\|_{L^2} |{\textstyle\int_\bbR} Q f|\lesssim 
        \lambda_0^{\ell+2-\frac{\ell+2}{4L}}|{\textstyle\int_\bbR} Q f|.
    \end{align}
    Next, we estimate $\partial_y^\ell (\phi_n \calH g -\tfrac12 p_n g-\tfrac12 \calH(p_n \calH g))$ part. By \eqref{eq:phi and p esti}, we derive
    \begin{align}
    	|b_1|^n\|\partial_y^\ell (\phi_n \calH g) \|_{L^2}\lesssim |b_1|^n\sum_{a=0}^\ell\|\partial_y^a \phi_n\|_{L^\infty}\|\partial_y^{\ell-a} g\|_{L^2}
    	\lesssim  \lambda_0^{2-\frac{n}{2L}}\sum_{a=0}^\ell \lambda_0^a \|\partial_y^{\ell-a}B_Q f\|_{L^2}. \label{eq:bfD1 2}
    \end{align}
    Similarly, again using \eqref{eq:phi and p esti}, we have
    \begin{align} \label{eq:bfD1 3}
    	|b_1|^n(\|\partial_y^\ell (p_n g) \|_{L^2}+ \|\partial_y^\ell \calH(p_n \calH g) \|_{L^2})
        &\lesssim \lambda_0^2 \sum_{a=0}^\ell \lambda_0^a\|\partial_y^{\ell-a}B_Q f\|_{L^2}.
    \end{align}
    For the $Q$-weighted part, we derive
    \begin{align}
    	|b_1|^n\|Q\partial_y^\ell (\phi_n \calH g) \|_{L^2}&\lesssim
        |b_1|^n\sum_{a=0}^\ell\|Q\partial_y^a \phi_n\|_{L^2}\|\calH\partial_y^{\ell-a} g\|_{L^\infty} \nonumber
        \\
        &\lesssim
        |b_1|^n\sum_{a=0}^\ell R^{2n-2-a}\|\partial_y^{\ell-a} g\|_{L^2}^{\frac12}\|\partial_y^{\ell+1-a} g\|_{L^2}^{\frac12} \nonumber
        \\
        &\lesssim
        |b_1|^n\sum_{a=0}^\ell R^{2n-2-a}(\lambda_0^{\frac12}\|\partial_y^{\ell-a} g\|_{L^2} + \lambda_0^{-\frac12}\|\partial_y^{\ell+1-a} g\|_{L^2}) \nonumber
        \\
        &\lesssim \lambda_0 \sum_{a=0}^{\ell+1} \lambda_0^a\|\partial_y^{\ell+1-a}B_Q f\|_{L^2}, \label{eq:bfD1 4}
    \end{align}
    and
    \begin{equation}\label{eq:bfD1 5}
        \begin{aligned}
            \|Q\partial_y^\ell (p_n g) \|_{L^2}+ \|Q\partial_y^\ell \calH(p_n \calH g) \|_{L^2} &\lesssim \|\partial_y^\ell (p_n g) \|_{L^2}+ \|\partial_y^\ell \calH(p_n \calH g) \|_{L^2}
            \\
            &\lesssim \lambda_0^2 \sum_{a=0}^\ell \lambda_0^a\|\partial_y^{\ell-a}B_Q f\|_{L^2}.
        \end{aligned}
    \end{equation}
    Collecting \eqref{eq:bfD1 1}--\eqref{eq:bfD1 5}, we get \eqref{eq:bfD mth} for $\bfD_1$.

    \textbf{Step 3.} We show \eqref{eq:bfD mth} for $\bfD_{2,n}$. For $1\le n\le L$, define
    \begin{align*}
    	q_n\coloneqq \calH (p_n) = Q^2\phi_n+\calH(yQ^2\phi_n).
    \end{align*}
    Then, \eqref{eq:phi and p esti} implies
    \begin{align}\label{eq: q esti}
    	\|\partial_y^a q_n\|_{L^\infty}\lesssim R^{2n-3-a}.
    \end{align}
    Set $g\coloneqq B_Qf$. Since \eqref{eq:BQBQstar equal I}, we have $f=B_Q^*g+\tfrac{1}{2\pi}Q{\textstyle \int_{\bbR}}Qf$, or
    \begin{align}
    	Qf=\tfrac1{\sqrt2}(yQ^2g+Q^2\calH g)+\tfrac{1}{2\pi}Q^2 {\textstyle \int_{\bbR}}Qf. \label{eq:bfD2 BQ g}
    \end{align}
    Now, we consider
    \begin{align*}
    	B_Q\bfD_{2,n}f=
    	\tfrac{1}{\sqrt2} \ol{{\bm\beta}_n} (y-\calH)[Q^2\calH((1+y^2)Q\phi_n f)].
    \end{align*}
    By \eqref{eq:HilbertProductRule}, we get
    \begin{align*}
    	(y-\calH)(Q^2\calH h)=Q^2h+\calH(yQ^2h).
    \end{align*}
    Applying this identity with $h=(1+y^2)Q\phi_n f$, we find
    \begin{align*}
    	B_Q\bfD_{2,n}f=\sqrt2 \ol{{\bm\beta}_n}(Q\phi_nf+\calH(yQ\phi_n f)).
    \end{align*}
    Substituting \eqref{eq:bfD2 BQ g} for $Qf$, we obtain
    \begin{align*}
    	B_Q\bfD_{2,n}f=\, &
    	\ol{{\bm\beta}_n}
    	[
    	yQ^2\phi_n g+Q^2\phi_n\calH g+\calH(y^2Q^2\phi_n g)+\calH(yQ^2\phi_n \calH g)
    	] +\tfrac{\ol{{\bm\beta}_n}}{\sqrt2 \pi} q_n {\textstyle \int_{\bbR}}Qf.
    \end{align*}
    Again using \eqref{eq:HilbertProductRule}, we have
    \begin{align*}
    	\calH(yQ^2\phi_n\calH g)
    	=\calH(yQ^2\phi_n)\calH g-yQ^2\phi_n g -\calH(\calH(yQ^2\phi_n) g).
    \end{align*}
    With this and the fact that $B_Q\td\bfD_Q^\ell=\partial_y^\ell B_Q$, we obtain 
     \begin{align*}
    	B_Q \td\bfD_Q^\ell \bfD_{2,n}f
    	=\ol{{\bm\beta}_n} \partial_y^\ell[2\calH(\phi_n g)+q_n \calH g -\calH(q_n g)]
    	+\tfrac{1}{\sqrt2 \pi}\ol{{\bm\beta}_n} \partial_y^\ell q_n\cdot {\textstyle \int_{\bbR}}Qf.
    \end{align*}
    $\partial_y^\ell[2\calH(\phi_n g)+q_n \calH g -\calH(q_n g)]$ can be estimated similarly to the previous step:
    \begin{align*}
        |\ol{{\bm\beta}_n} |\| \partial_y^\ell[2\calH(\phi_n g)+q_n \calH g -\calH(q_n g)] \|_{L^2} &\lesssim (\lambda_0^{2-\frac{n}{2L}}+\lambda_0^2)\sum_{a=0}^\ell \lambda_0^a \|\partial_y^{\ell-a}B_Q f\|_{L^2},\\
        |\ol{{\bm\beta}_n}|\|  Q\partial_y^\ell[2\calH(\phi_n g)+q_n \calH g -\calH(q_n g)] \|_{L^2} &\lesssim \lambda_0\sum_{a=0}^{\ell+1} \lambda_0^a \|\partial_y^{\ell+1-a}B_Q f\|_{L^2}.
    \end{align*}
    Using \eqref{eq: q esti}, the remaining term also can be easily checked as follows:
    \begin{align*}
        | \ol{ {\bm\beta}_n} {\textstyle \int_{\bbR}}Qf|\|\partial_y^\ell q_n\|_{L^2} \lesssim |b_1|^n \| f\|_{L^2} R^{2n-3-\ell} \lesssim \lambda_0^{\ell+3- \frac{n}{2L} } \|f\|_{L^2}
        \leq \lambda_0^{\ell+2}\|f\|_{L^2}.
    \end{align*}
    
    \textbf{Step 4.} We show \eqref{eq:bfD mth} for $\bfD_{3,n,m}$. Set $g\coloneqq B_Qf$. Similarly to the previous step, \eqref{eq:BQBQstar equal I} implies
    \begin{align*}
    	(1+y^2)Qf=\sqrt2(y+\calH)g+\tfrac{1}{\pi}{\textstyle\int_{\bbR}}Qf.
    \end{align*}
    Hence,
    \begin{align*}
    	\calH((1+y^2)Q\phi_m f)
    	=\sqrt2\calH(\phi_m(y+\calH)g)+\tfrac{1}{\pi}\calH(\phi_m){\textstyle\int_{\bbR}}Qf.
    \end{align*}
    Therefore,
    \begin{align*}
    	B_Q\bfD_{3,n,m}f
    	=\,&\sqrt2{\bm\beta}_n \ol{{\bm\beta}_m}(y-\calH)[\phi_n\calH((1+y^2)Q\phi_m f)]
    	\\
    	=\,&2{\bm\beta}_n \ol{{\bm\beta}_m}(y-\calH)[\phi_n\calH(\phi_m(y+\calH)g)]
    	\\
    	&+\tfrac{\sqrt2}{\pi}{\bm\beta}_n \ol{{\bm\beta}_m}(y-\calH)(\phi_n\calH(\phi_m))\cdot{\textstyle\int_{\bbR}}Qf.
    \end{align*}
    Since $B_Q\td\bfD_Q^\ell=\partial_y^\ell B_Q$, we obtain
    \begin{align*}
    	B_Q\td\bfD_Q^\ell \bfD_{3,n,m}f
    	=\,&2{\bm\beta}_n \ol{{\bm\beta}_m}\partial_y^\ell[(y-\calH)\bigl(\phi_n\calH(\phi_m(y+\calH)g)\bigr)]
    	\\
    	&+\tfrac{\sqrt2}{\pi}{\bm\beta}_n \ol{{\bm\beta}_m}\partial_y^\ell[(y-\calH)\bigl(\phi_n\calH(\phi_m)\bigr)]\cdot{\textstyle\int_{\bbR}}Qf.
    \end{align*}
    First, we recall the bounds for $\phi_{n}$,
    \begin{align}
    	\|\partial_y^a\phi_n\|_{L^\infty}\lesssim R^{2n-2-a},
    	\quad
    	\|\partial_y^a(y\phi_n)\|_{L^\infty}\lesssim R^{2n-1-a},\quad \|\partial_y^a\phi_n\|_{L^2}
    	\lesssim R^{2n-2-a+\frac12}. \label{eq:phi bound}
    \end{align}
     The $\partial_y^\ell[(y-\calH)\bigl(\phi_n\calH(\phi_m)\bigr)]$ term can be checked easily as follows:
    \begin{align*}
    	{\bm\beta}_n \ol{{\bm\beta}_m}\|\partial_y^\ell[(y-\calH)(\phi_n\calH(\phi_m))]\|_{L^2}
    	&\lesssim b_1^{n+m}\sum_{a=0}^\ell \||\partial_y^a(y\phi_n)| + |\partial_y^a\phi_n|\|_{L^\infty}\|\partial_y^{\ell-a}\phi_m\|_{L^2}
    	\\
    	&\lesssim \lambda_0^{2n+2m - \frac{n+m}{2L}} R^{2n+2m-\frac52-\ell} \lesssim \lambda_0^{\ell + 2 +\frac{1}{2} - \frac{n+m}{2L}},
    \end{align*}
    and 
    \begin{align*}
    	&{\bm\beta}_n \ol{{\bm\beta}_m}\|Q\partial_y^\ell((y-\calH)(\phi_n\calH(\phi_m)))\|_{L^2}\\
        &\lesssim b_1^{n+m} \sum_{a=0}^{\ell}\|\partial_y^a(y\phi_n)\|_{L^\infty}\|\partial_y^{\ell-a}\phi_m\|_{L^{2}}^{\frac12}\|\partial_y^{\ell+1-a}\phi_m\|_{L^{2}}^{\frac12} +\| \partial_y^a\phi_n\|_{L^\infty}\|\partial_y^{\ell-a}\phi_m\|_{L^2} \\ 
    	&\lesssim  \lambda_0^{2n+2m - \frac{n+m}{2L}} (R^{2n+2m-3-\ell} +R^{2n+2m-\frac{7}{2}-\ell} )  \lesssim \lambda_0^{\ell + 2 +1 - \frac{n+m}{2L}}.
    \end{align*}
    Next, we estimate the $\partial_y^\ell[(y-\calH)\bigl(\phi_n\calH(\phi_m(y+\calH)g)\bigr)]$ term. Using $R\|\partial_y^{a_1}\phi_n\|_{L^{\infty}}\sim \|\partial_y^{a_1}(y\phi_n)\|_{L^{\infty}}$, we have
    \begin{align*}
    	&{\bm\beta}_n \ol{{\bm\beta}_m}\|\partial_y^\ell((y-\calH)(\phi_n\calH(\phi_m(y+\calH)g)))\|_{L^2}
    	\\
    	&\lesssim b_1^{n+m}\sum_{a_1+a_2\le \ell}\|\partial_y^{a_1}(y\phi_n)\|_{L^\infty}\|\partial_y^{a_2}(y\phi_m)\|_{L^\infty}\|\partial_y^{\ell-a_1-a_2}g\|_{L^2}
    	\\
    	&\lesssim \lambda_0^{2n+2m - \frac{n+m}{2L}}\sum_{a=0}^{\ell} R^{2n+2m-2-a}\|\partial_y^{\ell-a}B_Q f\|_{L^2} \lesssim \lambda_0^{2-\frac{n+m}{2L}}\sum_{a=0}^\ell \lambda_0^a \|\partial_y^{\ell-a}B_Q f\|_{L^2},
    \end{align*}
    which concludes the first estimate of \eqref{eq:bfD mth} for $\bfD_3$.
    We now consider the $Q$-weighted term, $\|Q\partial_y^\ell((y-\calH)(\phi_n\calH(\phi_m(y+\calH)g)))\|_{L^2}$. One can find that the worst term is
    \begin{align*}
    	\|Q\partial_y^\ell(y\phi_n\calH(y\phi_m g))\|_{L^2}.
    \end{align*}
    The remaining terms can be easily estimated by replacing at least one of $y\phi$ to $\phi$ in the weight-free bound obtained immediately before. Thus, it suffices to estimate this worst term.
    We have
    \begin{align}
    	\|Q\partial_y^\ell(y\phi_n\calH(y\phi_m g))\|_{L^2}
    	\lesssim& \sum_{\substack{0\leq a_1\leq \ell,\\ a_1\neq 2n-1}} \|Q\partial_y^{a_1}(y\phi_n)\|_{L^\infty}\|\partial_y^{\ell-a_1}\calH(y\phi_m g)\|_{L^2} \label{eq:bfD3 1}
    	\\
    	&+\|Q\partial_y^{2n-1}(y\phi_n)\partial_y^{\ell-2n+1}\calH(y\phi_m g)\|_{L^2}. \label{eq:bfD3 2}
    \end{align}
    For the first term, thanks to \eqref{eq:phi bound}, we obtain
    \begin{align*}
        |{\bm\beta}_n \ol{{\bm\beta}_m}|\text{RHS}\eqref{eq:bfD3 1} &\lesssim b_1^{n+m}\sum_{a_1+a_2\leq \ell} R^{2n-2-a_1}R^{2m-1-a_2}\|\partial_y^{\ell-a_1-a_2}g\|_{L^2}
        \\
        &\lesssim \lambda_0^2\sum_{a=0}^{\ell} \lambda_0^a\|\partial_y^{\ell-a}g\|_{L^2}.
    \end{align*}
    For \eqref{eq:bfD3 2}, following the proof of \eqref{eq:bfD1 4}, we have
    \begin{equation*}
        \begin{aligned}
            b_1^{n+m} \eqref{eq:bfD3 2}\lesssim 
            \lambda_0^{1+\frac12-\frac{n+m}{2L}}\sum_{a=0}^{\ell+1} \lambda_0^a\|\partial_y^{\ell+1-a}g\|_{L^2}.
        \end{aligned}
    \end{equation*}
    Similarly to \eqref{eq:bfD1 worst}, this term can only appear when $2n-1\leq \ell$ (due to $\partial_y^{\ell-2n +1}$ in \eqref{eq:bfD3 2}). Hence, we get
    \begin{align*}
        b_1^{n+m} \eqref{eq:bfD3 2}\lesssim 
        \lambda_0^{1-\frac{\ell+1}{4L}}\sum_{a=0}^{\ell+1} \lambda_0^a\|\partial_y^{\ell+1-a}g\|_{L^2}.
    \end{align*}
    
     \textbf{Step 5.} We show \eqref{eq:bfD mth} for $\bfD_{j,\eps}$. We set $g\coloneqq B_Qf$ so that $f=B_Q^*g+\tfrac{1}{2\pi}Q{\textstyle \int_{\bbR}}Qf$. Since $\| V_L\|_{H^{2L}} \lesssim 1$, it suffices to show that for any $h \in H^{2L}$, 
     \begin{align}
          \|  \partial_y^{\ell} B_Q (\mathring{\eps}_0 \calH (\ol{h}B_Q^*g)+ h\calH(\ol{\mathring{\eps}_0}B_Q^*g))\|_{L^2} & \lesssim \lambda_0^2 \sum_{a=0}^\ell \lambda_0^a
            \|\partial_y^{\ell-a} g\|_{L^2},\label{eq:D eps 1}\\
            |\textstyle\int_{\bbR} Qf| \|  \partial_y^{\ell} B_Q (\mathring{\eps}_0 \calH (\ol{h}Q)+ h\calH(\ol{\mathring{\eps}_0}Q))\|_{L^2}&\lesssim \lambda_0^{\ell +2} \|f\|_{L^2}.\label{eq:D eps 2}
     \end{align}
     Similarly to the previous steps, we can easily check that
     \begin{align*}
         \text{LHS}\eqref{eq:D eps 1} &\lesssim \sum_{0\le a_1,a_2,a_3 \le \ell}\| \partial_y^{a_1}  \mathring{\eps}_0  \|_{L^{\infty}}\| \partial_y^{a_2}  h  \|_{L^{\infty}}  \|\partial_y^{a_3} g  \|_{L^{2}} \lesssim \lambda_0^{10L} \sum_{a=0}^\ell 
            \|\partial_y^{\ell-a} g\|_{L^2}, \\
            \text{LHS}\eqref{eq:D eps 2} &\lesssim\sum_{0\le a_1,a_2 \le \ell}\| \partial_y^{a_1}  \mathring{\eps}_0  \|_{L^{\infty}}\| \partial_y^{a_2}  h  \|_{L^{\infty}}  \|f \|_{L^2} \lesssim \lambda_0^{10L}
            \|f\|_{L^2}.
     \end{align*}
     Therefore, we get \eqref{eq:bfD mth}, which finishes the proof of Lemma~\ref{lem:bfD123 esti}.
\end{proof}

Finally, we prove Lemma~\ref{lem:init approx}.

\begin{proof}[Proof of Lemma~\ref{lem:init approx}]
    For a notational convenience, we omit the time dependency $s_0$. Set
	\begin{align*}
		\bfD \coloneqq \tfrac12 \bfD_1+\tfrac12 \bfD_2+\tfrac12 \bfD_3.
	\end{align*}
	Then
	\begin{align*}
		\wt\eps = V_L+\mathring\eps_0,\qquad
		\td\bfD_w=\td\bfD_Q+\bfD,\qquad
		w_m=\td\bfD_Qw_{m-1}+\bfD w_{m-1}\quad (m\geq 2).
	\end{align*}
    We first claim that
    \begin{align}\label{eq:w1 scale bound}
        \|\td\bfD_Q^\ell w_1\|_{L^2}\lesssim \lambda_0^{\ell+1}.
    \end{align}
    We observe that
    \begin{align}\label{eq:w1 bfD decom}
        w_1=L_Q\wt\eps+N_Q(\wt\eps)=L_Q\wt\eps+\bfD_1\Re(\wt\eps)+\tfrac12 \bfD_2(\wt\eps)+\tfrac12 \bfD_3(\wt\eps).
    \end{align}
    Note that, Lemma~\ref{lem:profile choosing} implies
    \begin{align*}
		\|\td\bfD_Q^\ell L_QV_L\|_{L^2}\lesssim \lambda_0^{m+1}
		\qquad (0\leq \ell\leq 2L-2).
	\end{align*}
    Hence, by Lemma~\ref{lem:bfD123 esti}, we get
    \begin{align*}
        \|\td\bfD_Q^\ell N_Q(\wt\eps)\|_{L^2}\lesssim \lambda_0\sum_{a=0}^{\ell} \lambda_0^a\|\td\bfD_Q^{\ell-a}\wt\eps\|_{L^2}.
    \end{align*}
    In addition, the definition of $B_Q$ with \eqref{eq:DQ BQ decomp} yields
    \begin{align*}
        \|\td\bfD_Q^{\ell-a}\wt\eps\|_{L^2}
        &\lesssim \|\partial_y^{\ell-a}(y\langle y\rangle^{-1}\wt\eps)\|_{L^2}+\|\partial_y^{\ell-a}(\langle y\rangle^{-1}\wt\eps)\|_{L^2}
        \\
        &\lesssim \sum_{n=1}^{L}b_1^{n}(\|\partial_y^{\ell-a}(y\phi_n)\|_{L^2}+\|\partial_y^{\ell-a}(\phi_n)\|_{L^2})
        +\|\mathring\eps_0\|_{\dot\calH^{\ell-a}}
        \lesssim \lambda_0^{\ell-a}.
    \end{align*}
    Here, $\phi_n$ is given by \eqref{eq:def phi}. Thus, we deduce
    \begin{align*}
        \|\td\bfD_Q^\ell N_Q(\wt\eps)\|_{L^2}\lesssim \lambda_0^{\ell+1},
    \end{align*}
    which concludes \eqref{eq:w1 scale bound}. We now claim, for $1\leq m\leq 2L-1$ and $0\leq \ell\leq 2L-1-m$,
	\begin{align}\label{eq:wm scale bound}
		\|\td\bfD_Q^\ell w_m\|_{L^2}\lesssim \lambda_0^{\ell+m}.
	\end{align}
	For $m=1$ this is exactly the previous estimate. If it holds for $m-1$, then
	\begin{align*}
		\|\td\bfD_Q^\ell w_m\|_{L^2}
		&\leq \|\td\bfD_Q^{\ell+1}w_{m-1}\|_{L^2}+\|\td\bfD_Q^\ell\bfD w_{m-1}\|_{L^2}
		\\
		&\lesssim \lambda_0^{\ell+m}
		+\lambda_0\sum_{a=0}^\ell \lambda_0^a \|\td\bfD_Q^{\ell-a}w_{m-1}\|_{L^2}
		\lesssim \lambda_0^{\ell+m},
	\end{align*}
    which implies \eqref{eq:wm scale bound} by the induction.

    Now, we prove \eqref{eq:init approx k any}. For the case $k\geq 2$, we claim
    \begin{align} \label{eq:tail compute goal 1}
        (w_{2k-1}-\td\bfD_Q^{2k-2}w_1,\iota yQ\chi)_r
        \lesssim \mathring b_1^k\lambda_0^{\frac{1}{4L}}.
    \end{align}
    By iterating $w_m=\td\bfD_Qw_{m-1}+\bfD w_{m-1}$,
	\begin{align*}
		w_{2k-1}-\td\bfD_Q^{2k-2}w_1
		=\sum_{\ell=0}^{2k-3}\td\bfD_Q^\ell\bfD w_{2k-2-\ell}.
	\end{align*}
    Since $Q^{-1}B_Q(yQ\chi)\in L^2$, for $\iota=1,i$, we get
    \begin{align*}
        (w_{2k-1}-\td\bfD_Q^{2k-2}w_1,\iota yQ\chi)_r
        \lesssim \|Q\bfD w_{2k-2,e}\|_{L^2}+\sum_{\ell=1}^{2k-3}\|QB_Q\td\bfD_Q^\ell\bfD w_{2k-2-\ell}\|_{L^2}.
    \end{align*}
    For $1\leq \ell\leq 2k-3$, by \eqref{eq:bfD mth} and \eqref{eq:wm scale bound}, we deduce
    \begin{align*}
        \|QB_Q\td\bfD_Q^\ell\bfD w_{2k-2-\ell}\|_{L^2}
        \lesssim 
        \lambda_0^{1-\frac{\ell+2}{4L}}\sum_{a=0}^{\ell+1} \lambda_0^a\|\td\bfD_Q^{\ell+1-a}w_{2k-2-\ell}\|_{L^2}
        \lesssim \lambda_0^{2k-\frac{2k-1}{4L}}\sim \mathring b_1^k\lambda_0^{\frac{1}{4L}}.
    \end{align*}
    Moreover, from \eqref{eq:bfD first} and \eqref{eq:wm scale bound}, we obtain
    \begin{align*}
         \|Q\bfD w_{2k-2,e}\|_{L^2}\lesssim \mathring b_1 \lambda^{2k-2}\lesssim \mathring b_1^k\lambda_0^{\frac{1}{4L}}.
    \end{align*}
    Hence, we arrive at \eqref{eq:tail compute goal 1}. Next, for all $1\leq k\leq L$, by \eqref{eq:bfD mth} and \eqref{eq:w1 bfD decom}, we derive
    \begin{align*}
        &(\td\bfD_Q^{2k-2}w_1-\td\bfD_Q^{2k-2}L_Q\wt\eps,\iota yQ\chi)_r
        \\
        &\lesssim \sum_{j=1}^3\|QB_Q\td\bfD_Q^{2k-2}\bfD_j\wt\eps\|_{L^2}
        \lesssim \lambda_0^{1-\frac{2k-1}{4L}}\sum_{a=0}^{2k-1} \lambda_0^a\|\td\bfD_Q^{2k-1-a}\wt\eps\|_{L^2}
        +\lambda_0^{2k-\frac{k}{2L}}|{\textstyle\int_\bbR} Q \wt\eps|.
    \end{align*}
    From $\wt\eps = V_L+\mathring\eps_0$, one can check
    \begin{align*}
        \|\wt\eps\|_{L^2}\lesssim 1,\qquad |{\textstyle\int_\bbR} Q \wt\eps|\lesssim \lambda^{\frac12}.
    \end{align*}
    For $\|\td\bfD_Q^{\ell}\wt\eps\|_{L^2}$, using $\wt\eps=\sum \mathring{\bm\beta}_n\phi_n(1+y^2)Q$ with \eqref{eq:BQBQstar equal I} and \eqref{eq:DQ BQ decomp}, we get
    \begin{align*}
        \|\td\bfD_Q^{\ell}(\mathring{\bm\beta}_n\phi_n(1+y^2)Q)\|_{L^2}
        \lesssim \mathring b_1^n\|\partial_y^\ell(y-\calH)\phi_n\|_{L^2}
        \lesssim b_1^n\lambda_0^{-(2n-1-\ell+\frac12)}\lesssim \la_0^{\ell}.
    \end{align*}
    Therefore, we have $\|\td\bfD_Q^{\ell}\wt\eps\|_{L^2}\lesssim \lambda_0^\ell$, which yields
    \begin{align}\label{eq:tail compute goal 2}
        (\td\bfD_Q^{2k-2}w_1-\td\bfD_Q^{2k-2}L_Q\wt\eps,\iota yQ\chi)_r
        \lesssim \lambda_0^{2k-\frac{2k-1}{4L}} \sim \mathring b_1^k \lambda_0^{\frac{1}{4L}}.
    \end{align}
    Combining \eqref{eq:tail compute goal 1} and \eqref{eq:tail compute goal 2}, we finish the proof.
\end{proof}

\appendix

\bibliographystyle{abbrv} 
\bibliography{referenceCM}

@article{HerreroVelazquez1992,
  title={A blow up result for semilinear heat equations in the supercritical case},
  author={Herrero, Miguel A. and Vel\'{a}zquez, Juan J. L.},
  journal={preprint},
  pages={50},
  year={1992}
}

@article {MerleRaphaelRodnianskiSzeftel2022Invention,
    AUTHOR = {Merle, Frank and Rapha\"el, Pierre and Rodnianski, Igor and
              Szeftel, Jeremie},
     TITLE = {On blow up for the energy super critical defocusing nonlinear
              {S}chr\"odinger equations},
   JOURNAL = {Invent. Math.},
  FJOURNAL = {Inventiones Mathematicae},
    VOLUME = {227},
      YEAR = {2022},
    NUMBER = {1},
     PAGES = {247--413},
      ISSN = {0020-9910,1432-1297},
   MRCLASS = {35Q55},
  MRNUMBER = {4359478},
       DOI = {10.1007/s00222-021-01067-9},
       URL = {https://doi.org/10.1007/s00222-021-01067-9},
}

@article{JendrejLawrieRodriguez2022ASENS,
	author = {Jendrej, Jacek and Lawrie, Andrew and Rodriguez, Casey},
	doi = {10.24033/asens.2514},
	fjournal = {Annales Scientifiques de l'\'{E}cole Normale Sup\'{e}rieure. Quatri\`eme S\'{e}rie},
	issn = {0012-9593},
	journal = {Ann. Sci. \'{E}c. Norm. Sup\'{e}r. (4)},
	mrclass = {58E20 (35A21 37J51 58E05)},
	mrnumber = {4468859},
	number = {4},
	pages = {1135--1198},
	title = {Dynamics of bubbling wave maps with prescribed radiation},
	volume = {55},
	year = {2022}}

@article{KriegerSchlagTataru2008Invent,
	author = {Krieger, Joachim and Schlag, Wilhelm and Tataru, Daniel},
	doi = {10.1007/s00222-007-0089-3},
	fjournal = {Inventiones Mathematicae},
	issn = {0020-9910},
	journal = {Invent. Math.},
	mrclass = {58J45 (35B40 35L70)},
	mrnumber = {2372807},
	mrreviewer = {Michael Ruzhansky},
	number = {3},
	pages = {543--615},
	title = {Renormalization and blow up for charge one equivariant critical wave maps},
	url = {https://doi.org/10.1007/s00222-007-0089-3},
	volume = {171},
	year = {2008}}

@article{KriegerSchlagTataru2009Duke,
	author = {Krieger, Joachim and Schlag, Wilhelm and Tataru, Daniel},
	doi = {10.1215/00127094-2009-005},
	fjournal = {Duke Mathematical Journal},
	issn = {0012-7094},
	journal = {Duke Math. J.},
	mrclass = {58J45 (35B40 35L70)},
	mrnumber = {2494455},
	mrreviewer = {Michael Ruzhansky},
	number = {1},
	pages = {1--53},
	title = {Slow blow-up solutions for the {$H^1(\mathbb R^3)$} critical focusing semilinear wave equation},
	url = {https://doi.org/10.1215/00127094-2009-005},
	volume = {147},
	year = {2009}}

@article{OrtolevaPerelman2013,
	author = {Ortoleva, C. and Perelman, G.},
	doi = {10.1090/S1061-0022-2014-01290-3},
	fjournal = {Rossi\u{\i}skaya Akademiya Nauk. Algebra i Analiz},
	issn = {0234-0852},
	journal = {Algebra i Analiz},
	mrclass = {35Q55 (35B44)},
	mrnumber = {3114854},
	mrreviewer = {Amin\ Esfahani},
	number = {2},
	pages = {162--192},
	title = {Nondispersive vanishing and blow up at infinity for the energy critical nonlinear {S}chr\"{o}dinger equation in {$\mathbb R^3$}},
	url = {https://doi.org/10.1090/S1061-0022-2014-01290-3},
	volume = {25},
	year = {2013}}

@article{Perelman2014CMP,
	author = {Perelman, Galina},
	doi = {10.1007/s00220-014-1916-1},
	fjournal = {Communications in Mathematical Physics},
	issn = {0010-3616},
	journal = {Comm. Math. Phys.},
	mrclass = {58J35 (35B44 35K45)},
	mrnumber = {3215578},
	mrreviewer = {Atanas G. Stefanov},
	number = {1},
	pages = {69--105},
	title = {Blow up dynamics for equivariant critical {S}chr\"{o}dinger maps},
	url = {https://doi.org/10.1007/s00220-014-1916-1},
	volume = {330},
	year = {2014}}

@article{Schmid2023arXiv,
	author = {Schmid, Tobias},
	journal = {preprint arXiv:2308.01883},
	title = {Blow up dynamics for the 3D energy-critical Nonlinear {S}chr{\"o}dinger equation},
	year = {2023}}

@article {HerreroVelazquez1994CRASPSI,
    AUTHOR = {Herrero, Miguel A. and Vel\'{a}zquez, Juan J. L.},
     TITLE = {Explosion de solutions d'\'{e}quations paraboliques
              semilin\'{e}aires supercritiques},
   JOURNAL = {C. R. Acad. Sci. Paris S\'{e}r. I Math.},
  FJOURNAL = {Comptes Rendus de l'Acad\'{e}mie des Sciences. S\'{e}rie I.
              Math\'{e}matique},
    VOLUME = {319},
      YEAR = {1994},
    NUMBER = {2},
     PAGES = {141--145},
      ISSN = {0764-4442},
   MRCLASS = {35B40 (35K57)},
  MRNUMBER = {1288393},
MRREVIEWER = {Sergey\ A.\ Vakulenko},
}

@article{FilippasHerreroVelazquez2000,
	author = {Filippas, Stathis and Herrero, Miguel A. and Vel\'{a}zquez, Juan J. L.},
	doi = {10.1098/rspa.2000.0648},
	fjournal = {The Royal Society of London. Proceedings. Series A. Mathematical, Physical and Engineering Sciences},
	issn = {1364-5021},
	journal = {R. Soc. Lond. Proc. Ser. A Math. Phys. Eng. Sci.},
	mrclass = {35K55 (35B33 35B40 35C20)},
	mrnumber = {1843848},
	mrreviewer = {Juli\'{a}n Aguirre},
	number = {2004},
	pages = {2957--2982},
	title = {Fast blow-up mechanisms for sign-changing solutions of a semilinear parabolic equation with critical nonlinearity},
	url = {https://doi.org/10.1098/rspa.2000.0648},
	volume = {456},
	year = {2000}}

@article {Mizoguchi2004AdvDE,
    AUTHOR = {Mizoguchi, Noriko},
     TITLE = {Type-{II} blowup for a semilinear heat equation},
   JOURNAL = {Adv. Differential Equations},
  FJOURNAL = {Advances in Differential Equations},
    VOLUME = {9},
      YEAR = {2004},
    NUMBER = {11-12},
     PAGES = {1279--1316},
      ISSN = {1079-9389},
   MRCLASS = {35K55 (35B40 35K15 35K57)},
  MRNUMBER = {2099557},
MRREVIEWER = {Varga\ Kalantarov},
}

@article {Mizoguchi2007Math.Ann.,
    AUTHOR = {Mizoguchi, Noriko},
     TITLE = {Rate of type {II} blowup for a semilinear heat equation},
   JOURNAL = {Math. Ann.},
  FJOURNAL = {Mathematische Annalen},
    VOLUME = {339},
      YEAR = {2007},
    NUMBER = {4},
     PAGES = {839--877},
      ISSN = {0025-5831,1432-1807},
   MRCLASS = {35K55 (35B40 35K15)},
  MRNUMBER = {2341904},
MRREVIEWER = {Juli\'{a}n\ Aguirre},
       DOI = {10.1007/s00208-007-0133-z},
       URL = {https://doi.org/10.1007/s00208-007-0133-z},
}

@article {Mizoguchi2011TranAMS,
    AUTHOR = {Mizoguchi, Noriko},
     TITLE = {Blow-up rate of type {II} and the braid group theory},
   JOURNAL = {Trans. Amer. Math. Soc.},
  FJOURNAL = {Transactions of the American Mathematical Society},
    VOLUME = {363},
      YEAR = {2011},
    NUMBER = {3},
     PAGES = {1419--1443},
      ISSN = {0002-9947,1088-6850},
   MRCLASS = {35K91 (35B07 35B44 35K20 57M07)},
  MRNUMBER = {2737271},
MRREVIEWER = {Christian\ Stinner},
       DOI = {10.1090/S0002-9947-2010-04784-1},
       URL = {https://doi.org/10.1090/S0002-9947-2010-04784-1},
}

@article {GerardLenzmann2024CPAM,
    AUTHOR = {G\'erard, Patrick and Lenzmann, Enno},
     TITLE = {The {C}alogero-{M}oser derivative nonlinear {S}chr\"odinger
              equation},
   JOURNAL = {Comm. Pure Appl. Math.},
  FJOURNAL = {Communications on Pure and Applied Mathematics},
    VOLUME = {77},
      YEAR = {2024},
    NUMBER = {10},
     PAGES = {4008--4062},
      ISSN = {0010-3640,1097-0312},
   MRCLASS = {35Q55 (35C08 37K15 47B35)},
  MRNUMBER = {4814915},
       DOI = {10.1002/cpa.22203},
       URL = {https://doi.org/10.1002/cpa.22203}
}

@article{KimKwon2019,
	author = {Kim, Kihyun and Kwon, Soonsik},
	doi = {10.1090/memo/1409},
	fjournal = {Memoirs of the American Mathematical Society},
	isbn = {978-1-4704-6120-1; 978-1-4704-7447-8},
	issn = {0065-9266,1947-6221},
	journal = {Mem. Amer. Math. Soc.},
	mrclass = {35Q55},
	mrnumber = {4574850},
	number = {1409},
	pages = {vi+128},
	title = {On pseudoconformal blow-up solutions to the self-dual {C}hern-{S}imons-{S}chr\"{o}dinger equation: existence, uniqueness, and instability},
	url = {https://doi.org/10.1090/memo/1409},
	volume = {284},
	year = {2023}}

@article{KimKwon2020blowup,
	author = {Kim, Kihyun and Kwon, Soonsik},
	doi = {10.1007/s40818-023-00147-8},
	fjournal = {Annals of PDE. Journal Dedicated to the Analysis of Problems from Physical Sciences},
	issn = {2524-5317,2199-2576},
	journal = {Ann. PDE},
	mrclass = {35B44 (35Q55)},
	mrnumber = {4564552},
	mrreviewer = {Van\ Duong\ Dinh},
	number = {1},
	pages = {Paper No. 6, 129},
	title = {Construction of blow-up manifolds to the equivariant self-dual {C}hern-{S}imons-{S}chr\"{o}dinger equation},
	url = {https://doi.org/10.1007/s40818-023-00147-8},
	volume = {9},
	year = {2023}}

@article{KimKwonOh2020blowup,
	author = {Kim, Kihyun and Kwon, Soonsik and Oh, Sung-Jin},
	journal = {preprint arXiv:2010.03252, to appear in Ann. Sci. \'{E}c. Norm. Sup\'{e}r},
	title = {Blow-up dynamics for smooth finite energy radial data solutions to the self-dual {C}hern-{S}imons-{S}chr\"{o}dinger equation},
	year = {2020}}

@article{Kim2022CSSrigidityArxiv,
	author = {Kim, Kihyun},
	journal = {preprint arXiv:2210.05412, to appear in J. Eur. Math. Soc.},
	title = {Rigidity of smooth finite-time blow-up for equivariant self-dual {C}hern-{S}imons-{S}chr\"{o}dinger equation},
	year = {2022}}

@article{CoteMartelMerle2011TopologicalExample,
	author = {C\^{o}te, Rapha\"{e}l and Martel, Yvan and Merle, Frank},
	doi = {10.4171/RMI/636},
	fjournal = {Revista Matem\'{a}tica Iberoamericana},
	issn = {0213-2230,2235-0616},
	journal = {Rev. Mat. Iberoam.},
	mrclass = {35C08 (35Q53 35Q55 37K40)},
	mrnumber = {2815738},
	number = {1},
	pages = {273--302},
	title = {Construction of multi-soliton solutions for the {$L^2$}-supercritical g{K}d{V} and {NLS} equations},
	url = {https://doi.org/10.4171/RMI/636},
	volume = {27},
	year = {2011}}

@article{AbanovBettelheimWiegmann2009FormalContinuum,
	author = {Abanov, Alexander G. and Bettelheim, Eldad and Wiegmann, Paul},
	doi = {10.1088/1751-8113/42/13/135201},
	fjournal = {Journal of Physics. A. Mathematical and Theoretical},
	issn = {1751-8113,1751-8121},
	journal = {J. Phys. A},
	mrclass = {37K10 (81R12)},
	mrnumber = {2485800},
	number = {13},
	pages = {135201, 24},
	title = {Integrable hydrodynamics of {C}alogero-{S}utherland model: bidirectional {B}enjamin-{O}no equation},
	url = {https://doi.org/10.1088/1751-8113/42/13/135201},
	volume = {42},
	year = {2009}}

@article{PelinovskyGrimshaw1995IntermediateDefocusingCMDNLSLaxPair,
	author = {Pelinovsky, Dmitry E. and Grimshaw, Roger H. J.},
	doi = {10.1063/1.530956},
	fjournal = {Journal of Mathematical Physics},
	issn = {0022-2488,1089-7658},
	journal = {J. Math. Phys.},
	mrclass = {35Q55},
	mrnumber = {1341985},
	mrreviewer = {Stanislav\ Z.\ Pakuliak},
	number = {8},
	pages = {4203--4219},
	title = {A spectral transform for the intermediate nonlinear {S}chr\"{o}dinger equation},
	url = {https://doi.org/10.1063/1.530956},
	volume = {36},
	year = {1995}}

@article{OlshanetskyPerelomov1976Invent,
	author = {Olshanetsky, M. A. and Perelomov, A. M.},
	doi = {10.1007/BF01418964},
	fjournal = {Inventiones Mathematicae},
	issn = {0020-9910,1432-1297},
	journal = {Invent. Math.},
	mrclass = {58F05},
	mrnumber = {426053},
	mrreviewer = {Mark\ Adler},
	number = {2},
	pages = {93--108},
	title = {Completely integrable {H}amiltonian systems connected with semisimple {L}ie algebras},
	url = {https://doi.org/10.1007/BF01418964},
	volume = {37},
	year = {1976}}

@article{Moser1975ClassicCalogerMoser,
	author = {Moser, J.},
	doi = {10.1016/0001-8708(75)90151-6},
	fjournal = {Advances in Mathematics},
	issn = {0001-8708},
	journal = {Advances in Math.},
	mrclass = {70.34},
	mrnumber = {375869},
	mrreviewer = {Dieter\ S.\ Schmidt},
	pages = {197--220},
	title = {Three integrable {H}amiltonian systems connected with isospectral deformations},
	url = {https://doi.org/10.1016/0001-8708(75)90151-6},
	volume = {16},
	year = {1975}}

@article{Calogero1971ClassicCalogerMoser,
	author = {Calogero, F.},
	doi = {10.1063/1.1665604},
	fjournal = {Journal of Mathematical Physics},
	issn = {0022-2488,1089-7658},
	journal = {J. Math. Phys.},
	mrclass = {81.34},
	mrnumber = {280103},
	mrreviewer = {Z.\ Jankovi\'{c}},
	pages = {419--436},
	title = {Solution of the one-dimensional {$N$}-body problems with quadratic and/or inversely quadratic pair potentials},
	url = {https://doi.org/10.1063/1.1665604},
	volume = {12},
	year = {1971}}

@article{CalogeroMarchioro1974ClassicCalogerMoser,
	author = {Calogero, F. and Marchioro, C.},
	doi = {10.1063/1.1666827},
	fjournal = {Journal of Mathematical Physics},
	issn = {0022-2488,1089-7658},
	journal = {J. Math. Phys.},
	mrclass = {81.35},
	mrnumber = {353869},
	mrreviewer = {H.\ Flaschka},
	pages = {1425--1430},
	title = {Exact solution of a one-dimensional three-body scattering problem with two-body and/or three-body inverse-square potentials},
	url = {https://doi.org/10.1063/1.1666827},
	volume = {15},
	year = {1974}}

@article {Badreddine2024PAA,
	AUTHOR = {Badreddine, Rana},
	TITLE = {On the global well-posedness of the {C}alogero-{S}utherland
	derivative nonlinear {S}chr\"odinger equation},
	JOURNAL = {Pure Appl. Anal.},
	FJOURNAL = {Pure and Applied Analysis},
	VOLUME = {6},
	YEAR = {2024},
	NUMBER = {2},
	PAGES = {379--414},
	ISSN = {2578-5885,2578-5893},
	MRCLASS = {35Q55 (37K10)},
	MRNUMBER = {4746420},
	DOI = {10.2140/paa.2024.6.379},
	URL = {https://doi.org/10.2140/paa.2024.6.379},
}

@article{Badreddine2024AIHPC,
	author = {Badreddine, Rana},
	JOURNAL = {Ann. Inst. H. Poincar\'{e} C Anal. Non Lin\'{e}aire},
    FJOURNAL = {Annales de l'Institut Henri Poincar\'{e} C. Analyse Non
              Lin\'{e}aire},
	title = {Traveling waves and finite gap potentials for the {C}alogero-{S}utherland derivative nonlinear {S}chr\"{o}dinger equation},
    YEAR = {2024},
    DOI = {10.4171/AIHPC/124},
    URL = {https://doi.org/10.4171/AIHPC/124},
}

@article {Badreddine2024SIAM,
    AUTHOR = {Badreddine, Rana},
     TITLE = {Zero-{D}ispersion {L}imit of the {C}alogero--{M}oser
              {D}erivative {NLS} {E}quation},
   JOURNAL = {SIAM J. Math. Anal.},
  FJOURNAL = {SIAM Journal on Mathematical Analysis},
    VOLUME = {56},
      YEAR = {2024},
    NUMBER = {6},
     PAGES = {7228--7249},
      ISSN = {0036-1410},
   MRCLASS = {35Q55 (30H10 35B50)},
  MRNUMBER = {4816606},
       DOI = {10.1137/24M1646935},
       URL = {https://doi-org.libra.kaist.ac.kr/10.1137/24M1646935},
}

@article {Sun2024LMPsystem,
    AUTHOR = {Sun, Ruoci},
     TITLE = {The intertwined derivative {S}chr\"odinger system of
              {C}alogero-{M}oser-{S}utherland type},
   JOURNAL = {Lett. Math. Phys.},
  FJOURNAL = {Letters in Mathematical Physics},
    VOLUME = {114},
      YEAR = {2024},
    NUMBER = {3},
     PAGES = {Paper No. 74, 32},
      ISSN = {0377-9017,1573-0530},
   MRCLASS = {37K10 (47B35)},
  MRNUMBER = {4754345},
       DOI = {10.1007/s11005-024-01815-x},
       URL = {https://doi.org/10.1007/s11005-024-01815-x},
}

@article{Gerard2023BOequexplicitFormula,
	author = {G\'{e}rard, Patrick},
	doi = {10.2140/tunis.2023.5.593},
	fjournal = {Tunisian Journal of Mathematics},
	issn = {2576-7658,2576-7666},
	journal = {Tunis. J. Math.},
	mrclass = {35Q53 (35C05 37K15 47B35)},
	mrnumber = {4662323},
	number = {3},
	pages = {593--603},
	title = {An explicit formula for the {B}enjamin-{O}no equation},
	url = {https://doi.org/10.2140/tunis.2023.5.593},
	volume = {5},
	year = {2023}}

@article{KillipLaurensVisan2023CMDNLSL2plusexplicitformulaarxiv,
	author = {Killip, Rowan and Laurens, Thierry and Vi\c{s}an, Monica},
	journal = {preprint arXiv:2311.12334},
	title = {Scaling-critical well-posedness for continuum {C}alogero-{M}oser models},
	year = {2023}}

@article{MouraPilod2010CMDNLSLocal,
	author = {de Moura, Roger Peres and Pilod, Didier},
	fjournal = {Advances in Differential Equations},
	issn = {1079-9389},
	journal = {Adv. Differential Equations},
	mrclass = {35Q55 (35B30 35B65 35R09 76B03)},
	mrnumber = {2677424},
	mrreviewer = {Makoto\ Nakamura},
	number = {9-10},
	pages = {925--952},
	title = {Local well posedness for the nonlocal nonlinear {S}chr\"{o}dinger equation below the energy space},
	volume = {15},
	year = {2010}}

@article{HadzicRaphael2019JEMS,
	author = {Had\v{z}i\'{c}, Mahir and Rapha\"{e}l, Pierre},
	doi = {10.4171/JEMS/904},
	fjournal = {Journal of the European Mathematical Society (JEMS)},
	issn = {1435-9855},
	journal = {J. Eur. Math. Soc. (JEMS)},
	mrclass = {35R35 (35A20 35K05)},
	mrnumber = {4012340},
	mrreviewer = {Rodica Luca},
	number = {11},
	pages = {3259--3341},
	title = {On melting and freezing for the 2{D} radial {S}tefan problem},
	url = {https://doi.org/10.4171/JEMS/904},
	volume = {21},
	year = {2019}}

@article{HoganKowalski2024PAA,
	author = {Hogan, James and Kowalski, Matthew},
	title = {Turbulent {T}hreshold for {C}ontinuum {C}alogero-{M}oser Models},
	JOURNAL = {Pure Appl. Anal.},
	FJOURNAL = {Pure and Applied Analysis},
	VOLUME = {6},
	YEAR = {2024},
	NUMBER = {4},
	PAGES = {941--954},
	ISSN = {2578-5885,2578-5893},
	MRCLASS = {35Q55 (37K10)},
	DOI = {10.2140/paa.2024.6.941},
	URL = {https://doi.org/10.2140/paa.2024.6.941},
}

@article{RodnianskiSterbenz2010,
	author = {Rodnianski, Igor and Sterbenz, Jacob},
	doi = {10.4007/annals.2010.172.187},
	fjournal = {Annals of Mathematics. Second Series},
	issn = {0003-486X,1939-8980},
	journal = {Ann. of Math. (2)},
	mrclass = {58E20 (35B44 35L70 35Q40 81T10)},
	mrnumber = {2680419},
	mrreviewer = {Andreas\ Gastel},
	number = {1},
	pages = {187--242},
	title = {On the formation of singularities in the critical {${\rm O}(3)$} {$\sigma$}-model},
	url = {https://doi.org/10.4007/annals.2010.172.187},
	volume = {172},
	year = {2010}}

@article{RaphaelRodnianski2012,
	author = {Rapha\"{e}l, Pierre and Rodnianski, Igor},
	doi = {10.1007/s10240-011-0037-z},
	fjournal = {Publications Math\'{e}matiques. Institut de Hautes \'{E}tudes Scientifiques},
	issn = {0073-8301,1618-1913},
	journal = {Publ. Math. Inst. Hautes \'{E}tudes Sci.},
	mrclass = {58E20 (35A20 35B44 35L70 58E15 81T13)},
	mrnumber = {2929728},
	mrreviewer = {Andreas\ Gastel},
	pages = {1--122},
	title = {Stable blow up dynamics for the critical co-rotational wave maps and equivariant {Y}ang-{M}ills problems},
	url = {https://doi.org/10.1007/s10240-011-0037-z},
	volume = {115},
	year = {2012}}

@article{Merle1993Duke,
	author = {Merle, F.},
	doi = {10.1215/S0012-7094-93-06919-0},
	fjournal = {Duke Mathematical Journal},
	issn = {0012-7094,1547-7398},
	journal = {Duke Math. J.},
	mrclass = {35Q55 (34G20 35B35)},
	mrnumber = {1203233},
	mrreviewer = {Song\ Mu\ Zheng},
	number = {2},
	pages = {427--454},
	title = {Determination of blow-up solutions with minimal mass for nonlinear {S}chr\"{o}dinger equations with critical power},
	url = {https://doi.org/10.1215/S0012-7094-93-06919-0},
	volume = {69},
	year = {1993}}

@article{GerardGrellier2015ExplictFormulaCubicSzego1,
	author = {G\'{e}rard, Patrick and Grellier, Sandrine},
	doi = {10.1090/S0002-9947-2014-06310-1},
	fjournal = {Transactions of the American Mathematical Society},
	issn = {0002-9947,1088-6850},
	journal = {Trans. Amer. Math. Soc.},
	mrclass = {37K15 (47B35)},
	mrnumber = {3301889},
	mrreviewer = {Julia\ B.\ Prada},
	number = {4},
	pages = {2979--2995},
	title = {An explicit formula for the cubic {S}zeg{\H{o}} equation},
	url = {https://doi.org/10.1090/S0002-9947-2014-06310-1},
	volume = {367},
	year = {2015}}

@article {GerardPushnitski2024CMP,
	AUTHOR = {G\'{e}rard, Patrick and Pushnitski, Alexander},
	TITLE = {The {C}ubic {S}zeg{\H{o}} {E}quation on the {R}eal {L}ine:
	{E}xplicit {F}ormula and {W}ell-{P}osedness on the {H}ardy
	{C}lass},
	JOURNAL = {Comm. Math. Phys.},
	FJOURNAL = {Communications in Mathematical Physics},
	VOLUME = {405},
	YEAR = {2024},
	NUMBER = {7},
	PAGES = {Paper No. 167},
	ISSN = {0010-3616,1432-0916},
	MRCLASS = {99-06},
	MRNUMBER = {4768537},
	DOI = {10.1007/s00220-024-05040-4},
	URL = {https://doi.org/10.1007/s00220-024-05040-4},
}

@article{MerleRaphaelRodnianski2013Invention,
	author = {Merle, Frank and Rapha\"{e}l, Pierre and Rodnianski, Igor},
	doi = {10.1007/s00222-012-0427-y},
	fjournal = {Inventiones Mathematicae},
	issn = {0020-9910,1432-1297},
	journal = {Invent. Math.},
	mrclass = {35K45 (35B44 35Q55 35Q60 58J99)},
	mrnumber = {3090180},
	mrreviewer = {Dimitrios\ E.\ Tzanetis},
	number = {2},
	pages = {249--365},
	title = {Blowup dynamics for smooth data equivariant solutions to the critical {S}chr\"{o}dinger map problem},
	url = {https://doi.org/10.1007/s00222-012-0427-y},
	volume = {193},
	year = {2013}}

@article{MerleRaphaelRodnianski2015CambJMath,
	author = {Merle, Frank and Rapha\"{e}l, Pierre and Rodnianski, Igor},
	doi = {10.4310/CJM.2015.v3.n4.a1},
	fjournal = {Cambridge Journal of Mathematics},
	issn = {2168-0930,2168-0949},
	journal = {Camb. J. Math.},
	mrclass = {35Q55 (35B44 35C06)},
	mrnumber = {3435273},
	mrreviewer = {Thomas\ Duyckaerts},
	number = {4},
	pages = {439--617},
	title = {Type {II} blow up for the energy supercritical {NLS}},
	url = {https://doi.org/10.4310/CJM.2015.v3.n4.a1},
	volume = {3},
	year = {2015}}

@article{RaphaelSchweyer2013CPAMHeat,
	author = {Rapha\"{e}l, Pierre and Schweyer, Remi},
	doi = {10.1002/cpa.21435},
	fjournal = {Communications on Pure and Applied Mathematics},
	issn = {0010-3640,1097-0312},
	journal = {Comm. Pure Appl. Math.},
	mrclass = {35R01 (35B44 35K91 53Cxx 58J35)},
	mrnumber = {3008229},
	mrreviewer = {Shu-Yu\ Hsu},
	number = {3},
	pages = {414--480},
	title = {Stable blowup dynamics for the 1-corotational energy critical harmonic heat flow},
	url = {https://doi.org/10.1002/cpa.21435},
	volume = {66},
	year = {2013}}

@article{RaphaelSchweyer2014AnalPDEHeatQuantized,
	author = {Rapha\"{e}l, Pierre and Schweyer, Remi},
	doi = {10.2140/apde.2014.7.1713},
	fjournal = {Analysis \& PDE},
	issn = {2157-5045,1948-206X},
	journal = {Anal. PDE},
	mrclass = {35K91 (35B44)},
	mrnumber = {3318739},
	mrreviewer = {Christopher\ P.\ Grant},
	number = {8},
	pages = {1713--1805},
	title = {Quantized slow blow-up dynamics for the corotational energy-critical harmonic heat flow},
	url = {https://doi.org/10.2140/apde.2014.7.1713},
	volume = {7},
	year = {2014}}

@article{vandenBergWilliams2013,
	author = {van den Berg, Jan Bouwe and Williams, J. F.},
	doi = {10.1017/S0956792513000247},
	fjournal = {European Journal of Applied Mathematics},
	issn = {0956-7925,1469-4425},
	journal = {European J. Appl. Math.},
	mrclass = {78A25},
	mrnumber = {3181487},
	number = {6},
	pages = {921--948},
	title = {({I}n-)stability of singular equivariant solutions to the {L}andau-{L}ifshitz-{G}ilbert equation},
	url = {https://doi.org/10.1017/S0956792513000247},
	volume = {24},
	year = {2013}}

@article{Collot2018MemAmer,
	author = {Collot, Charles},
	doi = {10.1090/memo/1205},
	fjournal = {Memoirs of the American Mathematical Society},
	isbn = {978-1-4704-2813-6; 978-1-4704-4379-5},
	issn = {0065-9266,1947-6221},
	journal = {Mem. Amer. Math. Soc.},
	mrclass = {35L71 (35B44 58B99)},
	mrnumber = {3778126},
	mrreviewer = {Herbert\ Koch},
	number = {1205},
	pages = {v+163},
	title = {Type {II} blow up manifolds for the energy supercritical semilinear wave equation},
	url = {https://doi.org/10.1090/memo/1205},
	volume = {252},
	year = {2018}}

@article{delPinoMussoWeiZhang2020HeatQuantized,
	author = {del Pino, Manuel and Musso, Monica and Wei, Juncheng and Zhang, Qidi and Zhang, Yifu},
	journal = {preprint arXiv:2002.05765},
	title = {Type {II} {F}inite time blow-up for the three dimensional energy critical heat equation},
	year = {2020}}

@article {KimKimKwon2024arxiv,
	author={Kim, Kihyun and Kim, Taegyu and Kwon, Soonsik},
	journal = {preprint arXiv:2404.09603, to appear in Mem. Amer. Math. Soc.},
	title={Construction of smooth chiral finite-time blow-up solutions to {C}alogero--{M}oser derivative nonlinear {S}chr\"{o}dinger equation}, 
	year={2024}
}

@article{KimTKwon2024arxivSolResol,
      title={Soliton resolution for {C}alogero--{M}oser derivative nonlinear {S}chr\"{o}dinger equation}, 
      journal = {preprint arXiv:2408.12843, to appear in J. Eur. Math. Soc.},
      author={Kim, Taegyu and Kwon, Soonsik},
      year={2024}
}

@article {JKKK2026arxiv,
	author={Jeong, Uihyeon and Kim, Kihyun and Kim, Taegyu and Kwon, Soonsik},
	journal = {preprint arXiv:2601.07410},
	title={Classification of single-bubble blow-up solutions for {C}alogero--{M}oser derivative nonlinear {S}chr\"{o}dinger equation}, 
	year={2026}
}

@article {CollotGhoulMasmoudiNguyen2022CPAM,
    AUTHOR = {Collot, Charles and Ghoul, Tej-Eddine and Masmoudi, Nader and
              Nguyen, Van Tien},
     TITLE = {Refined description and stability for singular solutions of
              the 2{D} {K}eller-{S}egel system},
   JOURNAL = {Comm. Pure Appl. Math.},
  FJOURNAL = {Communications on Pure and Applied Mathematics},
    VOLUME = {75},
      YEAR = {2022},
    NUMBER = {7},
     PAGES = {1419--1516},
      ISSN = {0010-3640,1097-0312},
   MRCLASS = {35Q92 (92C17)},
  MRNUMBER = {4438587},
}

@article{Jeong2024arxiv,
      title={Quantized slow blow-up dynamics for the energy-critical corotational wave maps problem}, 
      author={Uihyeon Jeong},   
      journal = {preprint arXiv:2312.16452, to appear in Anal. PDE},
      year={2024}, 
}

@article {GhoulIbrahimNguyen2018JDE,
    AUTHOR = {Ghoul, T. and Ibrahim, S. and Nguyen, V. T.},
     TITLE = {Construction of type {II} blowup solutions for the
              1-corotational energy supercritical wave maps},
   JOURNAL = {J. Differential Equations},
  FJOURNAL = {Journal of Differential Equations},
    VOLUME = {265},
      YEAR = {2018},
    NUMBER = {7},
     PAGES = {2968--3047},
      ISSN = {0022-0396,1090-2732},
   MRCLASS = {35L71 (35B40 35B44 35L51 35R01)},
  MRNUMBER = {3812220},
       DOI = {10.1016/j.jde.2018.04.058},
       URL = {https://doi.org/10.1016/j.jde.2018.04.058},
}

@article {Junichi2020AIHPC,
    AUTHOR = {Harada, Junichi},
     TITLE = {A higher speed type {II} blowup for the five dimensional
              energy critical heat equation},
   JOURNAL = {Ann. Inst. H. Poincar\'{e} C Anal. Non Lin\'{e}aire},
  FJOURNAL = {Annales de l'Institut Henri Poincar\'{e} C. Analyse Non
              Lin\'{e}aire},
    VOLUME = {37},
      YEAR = {2020},
    NUMBER = {2},
     PAGES = {309--341},
      ISSN = {0294-1449,1873-1430},
   MRCLASS = {35K91 (35B44)},
  MRNUMBER = {4072807},
MRREVIEWER = {Daniele\ Andreucci},
       DOI = {10.1016/j.anihpc.2019.09.006},
       URL = {https://doi.org/10.1016/j.anihpc.2019.09.006},
}

\end{document}